\documentclass[12pt]{article}
\usepackage[all]{xy}
\usepackage{amsmath,amsthm,amsfonts,amscd,color,amssymb,eucal,bbm}
\usepackage{latexsym,mathrsfs,amsxtra,calc}
\usepackage{url}

\def\RR{{\mathbb R}}
\def\CC{{\mathbb C}}

\def\ZZ{{\mathbb Z}}
\def\s1{{S^1}}
\def\diff{{{\rm Diff}^+ (S^1)}}
\def\difftilde{\widetilde{{\rm Diff^+}}(S^1) }
\def\vect{{{\rm Vect}(S^1)}}
\def\vectC{{{\rm Vect}_\CC(S^1)}}\def\sl2{{{\rm SL}(2,\RR)}}
\def\psl2{{{\rm PSL}(2,\RR)}}

\def\sltwoC{{{\mathfrak sl}(2,\CC)}}
\def\mob{{\rm\textsf{M\"ob}}}

\def\su2{{{\rm SU}(2)}}

\def\so3{{{\rm SO}(3)}}

\def\A{{\mathcal A}}
\def\B{{\mathcal B}}
\def\C{{\mathcal C}}
\def\D{{\mathcal D}}

\def\H{{\mathcal H}}
\def\I{{\mathcal I}}

\def\M{{\mathcal M}}
\def\N{{\mathcal N}}

\def\P{{\mathcal P}}

\def\S{{\mathcal S}}

\def\g{{\mathfrak g}}

\def\Exp{\mathrm{Exp}}
\renewcommand{\restriction}{\mathord{\upharpoonright}}
\newtheorem{theorem}{Theorem}[section]
\newtheorem{definition}[theorem]{Definition}
\newtheorem{corollary}[theorem]{Corollary}
\newtheorem{proposition}[theorem]{Proposition}
\newtheorem{lemma}[theorem]{Lemma}
\newtheorem{remark}[theorem]{Remark}
\newtheorem{example}[theorem]{Example}
\newtheorem{conjecture}[theorem]{Conjecture}
\title{{\huge From Vertex Operator Algebras to Conformal Nets and Back}
\footnotetext{Work supported in part by the ERC advanced grant 669240 QUEST ``Quantum Algebraic
Structures and Models''.}
}
\author{
{\sc Sebastiano Carpi}\footnote{Supported in part by PRIN-MIUR and GNAMPA-INDAM.}
\\
{\small Dipartimento di Economia, Universit\`a di Chieti-Pescara ``G. d'Annunzio''}\\
{\small Viale Pindaro, 42, I-65127 Pescara, Italy} \\
{\small E-mail: {\tt s.carpi@unich.it}}
\\[0,3cm]
{\sc Yasuyuki Kawahigashi}\footnote{Supported in part by Global COE Program ``The research and training center for new development in mathematics'', the Mitsubishi Foundation Research Grants and the Grants-in-Aid for Scientific Research, JSPS.}\\
{\small Department of Mathematical Sciences}\\
{\small The University of Tokyo, Komaba, Tokyo, 153-8914, Japan}
\\[0,05cm]
{\small and}
\\[0,05cm]
{\small Kavli IPMU (WPI), the University of Tokyo}\\
{\small 5-1-5 Kashiwanoha, Kashiwa, 277-8583, Japan}\\
{\small E-mail: {\tt yasuyuki@ms.u-tokyo.ac.jp}}
\\[0,3cm]
{\sc Roberto Longo}$^{*}$
\\
{\small Dipartimento di Matematica,
Universit\`a di Roma ``Tor Vergata''}\\
{\small Via della Ricerca Scientifica, 1, I-00133 Roma, Italy}\\
{\small E-mail: {\tt longo@mat.uniroma2.it}}
\\[0,3cm]
{\sc Mih\'aly Weiner}\footnote{Supported in part by a ``Bolyai J\'anos'' Research Scholarship of the Hungarian Academy of Sciences and the NKFIH Grant No. K124152.}
\\
{\small  Mathematical Institute, Department of Analysis,}\\ 
{\small Budapest University of Technology \& Economics (BME)}\\
{\small M\"uegyetem rk. 3-9, H-1111 Budapest, Hungary}\\
{\small E-mail: {\tt mweiner@renyi.hu}}
}

\date{}
\begin{document}
\maketitle 

\newpage

\begin{abstract} We consider unitary simple vertex operator algebras whose vertex operators satisfy certain energy bounds and a strong form of locality and call them strongly local. We present a general procedure which associates to every strongly local vertex operator algebra $V$ a conformal net $\A_V$ acting on the Hilbert space completion of $V$ and prove that the isomorphism class of $\A_V$ does not depend on the choice of the scalar product on $V$. We show that the class of strongly local vertex operator algebras is closed under taking tensor products and unitary subalgebras and that, for every strongly local vertex operator algebra $V$, the map $W\mapsto \A_W$ gives a one-to-one correspondence between the unitary subalgebras $W$ of $V$ and the covariant subnets of $\A_V$. Many known examples of vertex operator algebras such as the unitary Virasoro vertex operator algebras, the unitary affine Lie algebras vertex operator algebras, the known $c=1$ unitary vertex operator algebras, the moonshine vertex operator algebra, together with their coset and orbifold subalgebras, turn out to be  strongly local.  We give various applications of our results. In particular we show that the even shorter Moonshine vertex operator algebra is strongly local and that the automorphism group of the corresponding conformal net is the Baby Monster group. 
We prove that a construction of Fredenhagen and J\"{o}r{\ss} gives back the strongly local vertex operator algebra $V$ from the conformal net $\A_V$ and give conditions on a conformal net $\A$ implying that $\A= \A_V$ for some strongly local vertex operator algebra $V$. 
\end{abstract}

\newpage

\setcounter{tocdepth}{1}
\tableofcontents

\newpage

\section{Introduction}

We have two major mathematical formulations of chiral conformal
field theory.  A chiral conformal field theory is
described with a conformal net in one and with a
vertex operator algebra in the other.  The former is 
based on operator algebras and a part of algebraic quantum
field theory, and the latter is based on algebraic
axiomatization of vertex operators on the circle $S^1$.
The two formulations are expected to be equivalent at least
under some natural extra assumptions, but the exact
relations of the two have been poorly understood yet.  In
this paper, we present a construction of a 
conformal net from a vertex operator algebra satisfying
some nice analytic properties. Moreover, we show that the vertex operator algebra can 
be recovered from the associated conformal net.

Algebraic quantum field theory is a general theory to
study quantum field theory based on operator algebras
and has a history of more than 50 years, see \cite{Haag}.  The basic idea
is that we assign an operator algebra generated by
observables to each spacetime region.  In this way, we
have a family of operator algebras called a net of 
operator algebras.  Such a net is subject to a set
of mathematical axioms such as locality (Haag-Kastler axioms).  We study
nets of operator algebras satisfying the axioms, and
their mathematical studies consist of constructing 
examples, classifying them and studying their various
properties.  We need to fix a spacetime and its
symmetry group for a quantum field theory, and 
the 4-dimensional Minkowski space with the
Poincar\'e symmetry has been studied by many authors.
In a chiral conformal field theory, space and time
are mixed into the one-dimensional circle $S^1$ 
and the symmetry group is its orientation
preserving diffeomorphism group.

A quantum field $\Phi$
on $S^1$ is a certain operator-valued
distribution assumed to satisfy the chiral analogue of Wightman axioms \cite{StrWight}, see also \cite{DMS}, \cite{FST} and \cite[Chapter 1]{Kac}.

For an interval $I\subset S^1$,
take a test function supported in $I$.  Then the
pairing $\langle\Phi,f\rangle = \Phi(f)$ produces a
(possibly unbounded) operator ({\em smeared field}) which corresponds
to an observable on $I$ (if the operator is
self-adjoint).  We consider a von Neumann 
algebra ${\mathcal{A}}(I)$ generated by such
operators for a fixed $I$.  More generally we can consider this construction for a family 
$\Phi_i$, $i \in \mathscr{I}$ of (Bose) quantum fields, where $\mathscr{I}$ is an index set, not necessarily finite.

Based on this idea,
we axiomatize a family $\{{\mathcal{A}}(I)\}$ as 
follows and call it a {\emph{(local) conformal net}}.

Let ${\mathcal{I}}$ be the family of 
open, connected, non-empty and non-dense subsets (intervals) of $S^1$. 
A (local) {\emph{M\"obius covariant net}}
$\mathcal{A}$ of von Neumann algebras on $S^1$ is a map 
$$
\I \ni I \mapsto{\mathcal{A}}(I)
\subset B({\mathcal{H}})
$$
from ${\mathcal{I}}$ to the set of 
von Neumann algebras on a fixed 
Hilbert space $\mathcal{H}$ satisfying the following properties.
\begin{itemize}
\item {\it Isotony}. If $I_{1}\subset I_{2}$ 
belong to ${\mathcal{I}}$, then
we have ${\mathcal{A}}(I_{1})\subset{\mathcal{A}}(I_{2})$.
\item {\it Locality}. If $I_{1},I_{2}\in{\mathcal{I}}$ and $I_1\cap 
I_2=\varnothing$, then we have
$[{\mathcal{A}}(I_{1}),{\mathcal{A}}(I_{2})]=\{0\}$,
where brackets denote the commutator.
\end{itemize}
\begin{itemize}
\item {\it M\"obius covariance} 
There exists a strongly 
continuous unitary representation $U$ of the group 
$\mob \simeq \psl2$ of M\"obius transformations of $\s1$ on ${\mathcal{H}}$ such that
we have $ U(\gamma){\mathcal{A}}(I) U(\gamma)^*\ =\ {\mathcal{A}}(\gamma I)$, 
$\gamma\in \mob$, $I\in{\mathcal{I}}$. 

\item {\it Positivity of the energy}. The generator of the 
one-parameter rotation subgroup of $U$ 
(conformal Hamiltonian) is positive. 
\item {\it Existence of the vacuum}. There exists a unit 
$U$-invariant vector $\Omega\in{\mathcal{H}}$
called the {\emph{vacuum vector}}, and $\Omega$ 
is cyclic for the von Neumann algebra 
$\bigvee_{I\in{\mathcal{I}}}{\mathcal{A}}(I)$,
where the lattice symbol $\bigvee$ denotes the von Neumann algebra 
generated.
\end{itemize}

These axioms imply the Haag duality,
${\mathcal{A}}(I)'={\mathcal{A}}(I'),\quad I\in{\mathcal{I}}$,
where $I'$ is the interior of $S^1\setminus I$.

We say that a M\"obius covariant net ${\mathcal{A}}$ is 
\emph{irreducible} if 
$\bigvee_{I\in{\mathcal{I}}}{\mathcal{A}}(I)=B({\mathcal{H}})$. 
The net ${\mathcal{A}}$ is irreducible if and only if
$\Omega$ is the unique $U$-invariant vector up to scalar multiple,
and if and only if the local von Neumann 
algebras ${\mathcal{A}}(I)$ are factors. In this case they are
automatically type III$_1$ factors 
(except for the trivial case
${\mathcal{A}}(I)={\mathbb{C}}$).

Let $\diff$ be the group of orientation-preserving
diffeomorphisms of $S^1$.
By a {\emph{conformal net}} ${\mathcal{A}}$,
we  mean a M\"obius covariant 
net with the following property called {\emph{conformal covariance}} (or diffeomorphism covariance).

There exists a strongly continuous projective unitary 
representation $U$ of $\diff$
on ${\mathcal{H}}$ extending the unitary 
representation of $\mob$ such
that for all $I\in{\mathcal{I}}$ we have
\begin{eqnarray*}
 U(\gamma){\mathcal{A}}(I) U(\gamma)^*& =&{\mathcal{A}}(\gamma I ),
\quad  \gamma \in \diff, \\
 U(\gamma)AU(\gamma)^*& =& A,\quad A\in{\mathcal{A}}(I),\ \gamma \in{\mathrm{Diff}}(I'),
\end{eqnarray*}
where ${\mathrm{Diff}}(I)$ denotes the group of 
orientation preserving
diffeomorphisms $\gamma$ of $S^1$ such that $\gamma (z)=z$ for all $z\in I'$.

It should be pointed out that it is not known  whether or not the map $\I \ni I \mapsto \A(I)$ defined from a family 
$\{\Phi_i \}_{i\in \mathscr{I}}$ of chiral conformal covariant quantum fields on $\s1$ will satisfy in general the axioms of conformal nets described above. The main difficulty is given by locality. The problem is due to the fact that the smeared fields are typically unbounded operators and the von Neumann algebras generated by two unbounded operators commuting on a common invariant domain need not to commute as shown by a well known example by Nelson \cite{Nelson1959}. This difficulty is part of the more general problem of the mathematical equivalence of Wightman and Haag-Kastler axioms for quantum field theory a problem which has been studied rather extensively in the literature but which has not yet been completely solved, see e.g. 
\cite{BorchYng90,BorchYng92,BorchZimm,Bost2005,Buch90,DF1977,DSW,FH1981,FJ,GJ}. As it will become clear in this paper we deal with some special but mathematically very interesting aspect of this general problem.

\smallskip

A vertex operator is an algebraic formalization of a quantum
field on $S^1$, see \cite{FHL,Kac}.  A certain family of vertex operators is called
a vertex operator algebra.  This first appeared in the study of
Monstrous Moonshine, where one constructs a special vertex 
operator algebra called the {\emph{Moonshine vertex operator
algebra}} whose automorphism group is the Monster group, see \cite{FLM}.
Extensive studies have been made on vertex operator algebras
in the last 30 years.  

Since a conformal net and a vertex operator algebra (with a unitary structure)
both give mathematical axiomatization of a (unitary) chiral conformal
field theory, one expects that these two objects are in
a bijective correspondence, at least under some nice conditions,
but no such correspondence has been known so far. Actually, the axioms of vertex operator algebras are 
deeply related to Wightman axioms for quantum fields, see \cite[Chapter 1]{Kac}. Hence, the problem of the 
correspondence between vertex operator algebras and conformal nets can bee seen as a variant of the problem of the 
correspondence between Wightman field theories and algebraic quantum field theories discussed above. 

We would like to stress here an important difference between the Wightman approach and the vertex operator algebra approach to conformal field theory. In the Wightman approach the emphasis is on a family of fields $\{\Phi_i \}_{i\in \mathscr{I}}$ which generate the theory. For this point of view it is natural to start from this family in order do define an associated conformal net, see e.g. \cite{BS-M}. Many models of conformal nets are more or less directly defined in this way from a suitable family of generating fields. 
On the other hand, in the vertex operator algebra approach one considers, in a certain sense, all possible fields (the vertex operators) compatible to a given theory and corresponding to the Borchers class of the generating family $\{\Phi_i \}_{i\in \mathscr{I}}$, cf. 
\cite[II.5.5]{Haag}. In this sense the vertex operator algebras approach is closer in spirit to the algebraic approach, see \cite[III.1]{Haag}  and in this paper we will take this fact quite seriously. Another important similarity between the vertex operator approach and the conformal nets approach is the emphasis on representation theory. The latter will play only a marginal role in this paper but we believe that our work gives a solid basis for further investigations in this direction and we plan to come back to the representation theory aspects in the future. 
\smallskip

In this paper we present for the first time a correspondence between unitary vertex operator algebras and conformal nets. 
The basic idea is the following. We start with a simple unitary vertex operator algebra $V$ and we assume that the vertex operators satisfy
certain (polynomial) energy bounds. This assumption  guarantees a nice analytic behaviour of the vertex operators. It is rather standard 
in axiomatic quantum field theory and does not appear to be restrictive but it is presently not known if it guarantees the existence of an associated conformal net. Then, on the Hilbert space completion $\H_V$ of $V$ we can consider the smeared vertex operators 
$Y(a,f)$, $f\in C^\infty(\s1)$, $a \in V$ corresponding to the vertex operators $Y(a,z)$ of $V$. We then define a family of von Neumann 
algebras $\{\A_V(I)\}_{I \in \I}$ as described at the beginning of this introduction by using all the vertex operators. 
We say that $V$ is strongly local if the map $\I \ni I \mapsto \A_V(I)$ satisfies locality. In this case we prove $\A_V$ is an irreducible conformal net. The idea of using all the vertex operators instead of a suitable chosen generating family of well behaved generators has the great advantage to make the construction more intrinsic and functorial. In particular every unitary subalgebra $W\subset V$ of a strongly local vertex operator algebra turns out to be strongly local and the map $W \mapsto \A_W$ gives rise to a one-to-one correspondence 
between the unitary vertex subalgebras $W$ of $V$ and the covariant subnets of $\A_V$. Moreover, we prove that the automorphism group 
of $\A_V$ coincides with the unitary automorphism group of $V$ and that, if the latter is finite, it coincides with the full automorphism group of $V$.

Although the strong locality condition appears {\it a priori} to be rather restrictive and difficult to prove we show, inspired by the work of Driessler, Summers and Wichmann \cite{DSW} that if a generating family $\mathscr{F}$ of quasi-primary (i.e. M\"obius covariant) Hermitian vertex operators, generates a conformal net $\A_{\mathscr{F}}$ then  $V$ is strongly local and $\A_V=\A_{\mathscr{F}}$.
This result heavily relies on the deep connection between the Tomita-Takesaki modular theory of von Neumann algebras and the 
space-time symmetries of quantum field theories first discovered by Bisognano and Wichmann \cite{Bi-Wi}.
 As a consequence, standard arguments (see e.g. \cite{BS-M}) shows that if $V$ is generated by fields of conformal dimension one and 
by Virasoro fields then it is strongly local. This gives many examples of strongly local vertex operator algebras, e.g. the affine vertex operator algebras and their subalgebras, orbifold vertex operator algebras and coset vertex operator algebras. Moreover, the Moonshine vertex operator algebra $V^\natural$ turns out to be strongly local and the automorphism group of the net $\A_{V^\natural}$ is the monster, a result previously proved in \cite{KL06}. As a consequence we can construct an irreducible conformal net $\A_{VB^\natural_{(0)}}$ associated with the even shorter Moonshine vertex operator algebra $VB^\natural_{(0)}$ constructed by H\"{o}hn. Moreover, we show that the automorphism group of $\A_{VB^\natural_{(0)}}$ is the Baby Monster group. 
As far as we can see there is no known example of simple unitary vertex operator algebra which can be shown to be not strongly local and we conjecture that such an example does not exist. 

We also show that one can reconstruct the strongly local vertex operator algebra $V$ from the corresponding conformal net $\A_V$  by using the work of Fredenhagen and J\"{o}r{\ss} \cite{FJ}. Actually we consider a variant of the construction in \cite{FJ} which is directly obtained from the Tomita-Takesaki modular theory of von Neumann algebras. We also find a set of natural conditions on a irreducible conformal net $\A$, including energy bounds for the Fredenhagen-J\"{o}r{\ss} fields, which are equivalent to the requirement that $\A$ coincides with the net $\A_V$ associated with a simple unitary strongly local vertex operator algebra $V$. The existence of irreducible conformal nets not satisfying these condition appears to be an open problem. 

\smallskip

In order to keep this paper reasonably self-contained, the first four sections are devoted to various preliminaries on operator algebras, conformal nets and vertex operator algebras. In  Sect.  \ref{SectionUnitaryVOA} we define and study the notion of unitary vertex operator algebra. The definition has previously appeared more or less explicitly in the literature e.g. \cite{matsuo03} where unitary vertex operator algebras appear as vertex operator algebras having a real form with a positive definite invariant bilinear form, see also   
\cite[Sect.12.5]{FLM}. Here we prefer an alternative definition which is easily seen to be equivalent and we replace the real forms by the corresponding antilinear automorphisms (PCT operators).  The same definition, with a slightly different language, has also been recently used by Dong and Lin  \cite{DL2014}. In this paper we give a further equivalent definition of unitarity based 
on the requirement of locality for the adjoint (with respect to the scalar product) vertex operators. Our proof of the equivalence of the two definitions gives a 
vertex algebra version of the PCT theorem in Wightman quantum field theory \cite{StrWight}. Moreover, we study the question of the uniqueness of the unitary structure, i.e. of the {\it invariant scalar product}, in relation to the properties of the automorphism group of the underlying vertex operator algebra. We show that the scalar product is unique if and only if the automorphism group is compact, that in this case the automorphism group coincides with the unitary automorphism group and that it must also be totally disconnected. This happens in the special but important examples in which the automorphism group is finite as in the case of the Moonshine vertex operator algebra. 

\section{Preliminaries on von Neumann algebras}
In this section we introduce some of the basic concepts of the theory of von Neumann algebras and related facts on Hilbert space operators which will be frequently used in the following. 
Most of the topics discussed in this subsection can be found in any standard introductory book on operator algebras on Hilbert spaces,  see e.g. \cite{Bla2006,BR1,KR1,KR2}. We refer the reader to these books for more details and for the proofs of the results described in this section. 

\subsection{Von Neumann algebras}
 Let $\H$ be a (complex) Hilbert space with scalar product $(\cdot | \cdot)$, let $B(\H)$ denote the algebra of bounded linear operators
$\H\to \H$ and denote by $1_\H \in B(\H)$ the identity operator. Moreover, we denote by $U(\H)$ the group of unitary operators on $\H$. 

Recall that $B(\H)$ equipped with the usual operator norm is a Banach space (in fact it is a Banach algebra). 
For every $A\in B(\H)$ we denote by $A^* \in B(\H)$ its (Hilbert space) adjoint so that 
$(b|Aa)=(A^*b|a)$ for all $a,b \in \H$. The map $A \mapsto A^*$ is an antilinear involution $B(\H) \to B(\H)$ satisfying 
$(AB)^*= B^*A^*$ for all $A,B \in B(\H)$.     

For a given subset $\S \in B(\H)$ we denote by $\S^*$ the subset of $B(\H)$ defined by 
\begin{equation}
\S^* \equiv \{A \in B(\H): A^* \in \S\}.
\end{equation}
We say that $\S$ is self-adjoint if $\S=\S^*$. 

Given $\S \subset B(\H)$ we denote by $\S'$ the commutant of $\S$, namely the subset of $B(\H)$ defined by 
\begin{equation}
\S'= \{A\in B(\H): [A,B]=0\; \textrm{for all}\; B\in S \},
\end{equation}
where, for any $A,B \in B(\H)$, $[A,B]$ denotes the commutator $AB-BA$. The commutant $\S''$ of $\S'$ is called the bicommutant of $\S$. 
We denote by $\S'''$ the commutant of $\S''$ and so on. It turns out that $\S'''=\S'$ for every subset $\S\subset B(\H)$. Moreover, 
if $\S \subset B(\H)$ is self-adjoint then $\S'$ is a self-adjoint subalgebra of $B(\H)$ which is also unital, i.e. $1_\H \in \S'$. 

A self-adjoint subalgebra $\M \subset B(\H)$ is called a {\bf von Neumann algebra} if  $\M = \M''$.  Accordingly, $(\S \cup \S^*)'$ is a von Neumann algebra for all subsets $\S \subset B(\H)$ and $W^*(\S) \equiv (\S \cup \S^*)''$ is the smallest von Neumann algebra containing 
$\S$. 

A von Neumann algebra $\M$ is said to be a {\bf factor} if $\M'\cap \M =\CC1_\H$, i.e. $\M$ has a trivial center. 
$B(\H)$ is a factor for any Hilbert space $\H$. Its isomorphism class as an abstract complex $*$-algebra only depends on the  Hilbertian dimension of $\H$. A von Neumann algebra $\M$ isomorphic to some $B(\H)$ (here $\H$ is not necessarily the same Hilbert space on which $\M$ acts) is called a {\bf type I factor}. If $\H$ has dimension $n \in \ZZ_{>0}$ then $\M$ is called a type ${\rm I}_n$ factor while if 
$\H$ is infinite-dimensional then $\M$ is called a type ${\rm I}_\infty$ factor. 

There exist factors which are not of type I. They are divided in two families: the {\bf type II factors} (type ${\rm II}_1$ or type 
${\rm II}_\infty$) and {\bf type III factors} (type $\textrm{III}_\lambda$, $\lambda \in [0,1]$, cf. \cite{Connes1973}).

If $\M$ and $\N$ are von Neumann algebras and $\N \subset \M$ then $\N$ is called a {\bf von Neumann subalgebra} of $\M$. 
If $\M$ is a factor then a von Neumann subalgebra $\N \subset \M$ which is also a factor is called a {\bf subfactor}. 
The theory of subfactors is a central topic in the theory of operator algebras and in its applications to quantum field theory. 
Subfactor theory  was initiated in the seminal work \cite{Jones1983} where V. Jones introduced and studied an index $[M:N]$ for type 
${\rm II}_1$ factors. Subfactor theory and the notion of {\bf Jones index} was later generalized to type III subfactors and also to more general inclusions of von Neumann algebras, see \cite{EK1998,Kosaki1998,LongoCMP1989,LR1995,Popa1995}.

 A central result in the theory of von Neumann algebras is von Neumann bicommutant theorem which states that a self-adjoint 
 unital subalgebra $\M \subset B(\H)$ is a von Neumann algebra if and only if it is closed with respect to strong operator topology of
 $B(\H)$. In fact the statement remains true if one replace the strong topology on $B(\H)$ with one of the following: the weak topology, 
 the $\sigma$-weak topology (also called ultra-weak topology) and the  $\sigma$-strong topology (also called ultra-strong topology).
 In particular, every von Neumann algebra is also a (concrete) C*-algebra, namely it is a norm closed self-adjoint subalgebra of 
 $B(\H)$. Moreover, if $\H$ is separable, as it will typically be the case in this paper, a self-adjoint unital subalgebra $\M \subset B(\H)$
 is a von Neumann algebra if and only if for any $A\in B(\H)$, the existence of a sequence $A_n \in \M$, $n\in \ZZ_{> 0}$, such that 
 $A_n a \to Aa$ for all $a\in \H$, implies that $A\in \M$. Note also that von Neumann bicommutant theorem 
 implies that, for any subset $\S \subset B(\H)$, $W^*(\S)$ coincides with the strong closure of the unital self-adjoint subalgebra of $B(\H)$ generated by $\S$. If, for every $\alpha \in \mathscr{I}$, with $\mathscr{I}$ any index set, $\M_\alpha \subset B(\H)$ is a von Neumann algebra then 
$\bigcap_{\alpha \in \mathscr{I}}\M_{\alpha}$ is a von Neumann algebra. Moreover, 
\begin{equation}
\bigvee_{\alpha \in \mathscr{I}} \M_\alpha \equiv W^*(\bigcup_{\alpha \in \mathscr{I}}\M_\alpha) = 
\left( \bigcap_ {\alpha \in \mathscr{I}} \M_\alpha' \right)'
\end{equation}
is the von Neumann algebra generated by the von Neumann algebras $\M_\alpha$, $\alpha \in \mathscr{I}$.

If $\M$ is a von Neumann algebra on the Hilbert space $\H$ and $e \in \M'$ is an (orthogonal) projection commuting with $\M$ 
then the closed subspace $e\H$ is $\M$-invariant and 
\begin{equation}
\M_e \equiv \M\restriction_{e\H}= \{ A\restriction_{e\H}  :  A\in \M \}
\end{equation}
is a von Neumann algebra on $e\H$, the {\bf von Neumann algebra induced by $e$}.

Now let $\H_1$ and $\H_2$ be two Hilbert spaces and let $\H_1\otimes  \H_2$ be their algebraic tensor product. 
Then, $\H_1\otimes \H_2$ has a natural scalar product and we denote by $\H_1\overline{\otimes}\H_2$ 
the corresponding Hilbert space completion, the {\bf Hilbert space tensor product}. If $\M_1$ (resp. $\M_2$) is a von Neumann algebra on
$\H_1$ (resp. $\H_2$) then the algebraic tensor product $\M_1 \otimes \M_2$ is a $*$-subalgebra of $B(\H_1\overline{\otimes}\H_2)$ and the {\bf von Neumann tensor product} $\M_1\overline{\otimes}\M_2$ is defined by 
$$\M_1\overline{\otimes}\M_2 \equiv \left(\M_1 {\otimes}\M_2 \right)''.$$ 
It can be shown that 
$$\left(\M_1\overline{\otimes}\M_2\right)'= \M_1'\overline{\otimes}\M_2'.$$ 
Moreover, 
$$B(\H_1)\overline{\otimes}B(\H_2) = B(\H_1\overline{\otimes}\H_2).$$

\subsection{Unbounded operators affiliated with von Neumann algebras}
\label{SubsectUnboundedOp&VNA}
 
 By a linear operator (or simply an operator) on a Hilbert space $\H$ we always mean a linear map $A: \D \to \H$, where the domain $\D$ is a linear subspace of $\H$. If the domain $\D(A) \equiv D$ of $A$ is dense in $\H$ we say that $A$ is densely defined. Recall that $A$ is said to be closed if its graph is a closed subset of $\H \times \H$ with respect to the product topology and that $A$ is said to be closable if the closure of its graph is the graph of an operator ${\overline{A}}$ called the closure of $A$. 
 
The adjoint $A^*$ of a densely defined operator $A$ on $\H$ is always a closed operator on $\H$. A densely defined operator $A$ on $\H$ is closable if and only if its adjoint $A^*$ is densely defined. If this is the case then $\overline{A}=A^{**}$. A bounded densely defined operator  $A$ on $\H$ is always closable and it belongs to $B(\H)$ if and only if it is closed. If the graph of $A$ is a subset of the graph of $B$ then $B$ is said to be an extension of  $A$ and as usual we will write $A \subset B$. Let $A$ be closed operator with domain $\D(A)$, let $\D_0$ be a linear subspace of $\D(A)$ and let $A_0$ be the restriction of $A$ to $\D_0$. Then, $A_0$ is closable and it closure obviously satisfies $\overline{A_0} \subset A$.  One says that $\D_0$ is a {\bf core} for $A$ if $\overline{A_0} = A$.  
If $A$ is a closed densely defined operator then  $A^*A$ is self-adjoint (in particular densely defined and closed) with non-negative spectrum. The absolute value of $|A|$ of $A$ is defined through the spectral theorem by $|A|\equiv (A^*A)^{1/2}$. Then there is a unique 
$C \in B(\H)$ such that ${\rm Ker}(C)={\rm Ker}(A)$ and $C|A| = A$. $C$ is a partial isometry, i.e. $C^*C$ and $CC^*$ are 
(orthogonal) projections. The decomposition $A=C|A|$ is called the polar decomposition of $A$. $A$ is injective with dense range 
if and only if $C$ is a unitary operator.   Similar definitions, with analogous results, can be given for antilinear operators on $\H$.

An operator $A$ on $\H$ with domain $\D(A)$ is said to commute with a bounded operator $B\in B(\H)$ 
(and {\it viceversa}) if  $BA \subset AB$, namely if $B\D(A) \subset \D(A)$ and $ABa=BAa$ for all $a \in \D(A)$. If $A$ is densely defined and closable and if $A$ commutes with $B\in B(\H)$ then $A^*$ commutes with $B^*$. 
The following fact is very useful: if the densely defined operator $A$ is closed and $\D_0$ is a core for $A$ then $A$ commute with 
$B\in \B(\H)$ if and only if $B\D_0 \subset \D(A)$ and $ABa=BAa$ for all $a \in \D_0$.   

\medskip

A closed densely defined operator $A$ on $\H$ is said to be affiliated with a von Neumann algebra $\M \subset B(\H)$ if $A$ commutes 
with all operators in $\M'$. It turns out that a closed densely defined operator  $A$ is affiliated with $\M$ if and only if there is a sequence
$A_n \in \M$, $n \in \ZZ_{> 0}$ such that $A_na \to Aa$ and $A_n^*b \to A^*b$ for all $a \in \D(A)$ and all $b\in D(A^*)$. 

For any closed densely defined operator $A$ on $\H$ the set 
\begin{equation}
\{B\in B(\H): BA \subset AB,\; B^*A \subset AB^* \}
\end{equation}
is a von Neumann algebra and 
\begin{equation}
W^*(A)\equiv  \{B\in B(\H): BA \subset AB,\;  B^*A \subset AB^* \}'
\end{equation}
is the smallest von Neumann algebra to which $A$ is affiliated called the von Neumann algebra generated by $A$. 

If $A$ is a self-adjoint operator on a separable Hilbert space then, as a consequence of the spectral theorem, 
\begin{equation}
W^*(A) = \{ f(A): f \in  \mathscr{B}_b(\RR) \}
\end{equation}
where $\mathscr{B}_b(\RR)$ is the set of bounded Borel functions on $\RR$.  

More generally, if $A$ is densely defined and closed with polar decomposition $A=C|A|$, then $W^*(A)=W^*(C)\vee W^*(|A|)$ and hence 
$B, B^* \in B(\H)$ both commute with $A$ 
if and only if they both commute with $C$ and with the spectral projections of $|A|$.  

If $\mathscr{I}$ is an index set and $\{A_\alpha: \alpha \in\mathscr{I}\}$ is a family of closed densely defined operators on $\H$ 
then we put
\begin{equation}
W^*\left(\{A_\alpha: \alpha \in\mathscr{I}\} \right) \equiv \bigvee_{\alpha \in \mathscr{I} } W^*(A_\alpha),
\end{equation}
and we say that $W^*\left(\{A_\alpha: \alpha \in\mathscr{I}\} \right)$ is the von Neumann algebra generated by 
$\{A_\alpha: \alpha \in\mathscr{I}\}$. 

If $\D\subset \H$ is a linear subspace and $A_\alpha$, $\alpha \in \mathscr{I}$ are operators on $\H$ then $\D$ is called a common invariant domain for the operators $A_\alpha$, $\alpha \in \mathscr{I}$, if $\D \subset \D(A_\alpha)$ and $A_\alpha\D \subset \D$ for all 
$\alpha \in \mathscr{I}$.

The following proposition is well known and will be frequently used in this paper. 

\begin{proposition}  
\label{PropositionABcommute} 
Let $A$, $B$ be closed densely defined operators on a Hilbert space $\H$ and let $\D$ be a common 
invariant domain for $A$ and $B$. Assume that $W^*(A) \subset W^*(B)'$. Then $ABa=BAa$ for all $a\in \D$.
\end{proposition}
\begin{proof} Let $B_n\in W^*(B)$, $n \in \ZZ_{>0}$ be a sequence such that $B_na \to Ba$ for all $a\in \D(B)$. Since $B_n$ commutes with $A$ for all $n\in \ZZ_{>0}$ then, for any $a\in \D$, $B_na$ is in the domain of $A$, $Aa$ is in the domain of $B$ and $AB_na=B_nAa \to BAa$. Since $A$ is closed it follows that $ABa=BAa$. 
\end{proof}

The converse is known to be false thanks to examples due to Nelson \cite[Sect.10]{Nelson1959}, see also \cite[Sect.VIII.5]{ReedSimonI}. We summarize this fact in the following proposition.

\begin{proposition} Let $\H$ be a separable infinite-dimensional Hilbert space. Then there exists two self-adjoint operators $A$ and $B$ on 
$\H$ and a common invariant core $\D$ for $A$ and $B$ such that $ABc=BAc$ for all $c\in \D$ but $W^*(A)$ is not a subset of $W^*(B)'$. 
\end{proposition} 

\subsection{Tomita-Takesaki modular theory}
Let $\H$ be a Hilbert space and let $\M \subset B(\H)$ a von Neumann algebra. A vector $\Omega$ is said to be {\bf cyclic} for $\M$ if the linear subspace $\M\Omega$ is dense in $\H$. A vector $\Omega$ is said to be separating for $\M$ if, for every $A\in \M$, $A\Omega=0$ implies that $A=0$. It can be shown that a  vector $\Omega \in \H$ is cyclic for $\M$ if and only if it is separating for $\M'$ and symmetrically that $\Omega \in \H$ is separating for $\M$ if and only if it is cyclic for $\M'$. 
\medskip 

Let $\M\subset B(\H)$ be a von Neumann algebra and let $\Omega \in \H$ be cyclic and separating for $\M$. 
Then the map $A\Omega \mapsto A^*\Omega$ is well defined and injective and give rise to an antilinear operator $S_0$ on $\H$
with domain $\M\Omega$ and range $\M\Omega$. Hence $S_0$ is densely defined and has dense range. Moreover, 
$S_0^2= 1_{\M\Omega}$. If in the definition 
of $S_0$ we replace $\M$ with $\M'$ we obtain another antilinear operator $F_0$ on $\H$ with domain $\M'\Omega$ and range 
$\M'\Omega$. It is easy to see that $F_0 \subset S_0^*$ and, symmetrically that $S_0 \subset F_0^*$. Accordingly, $S_0$ and 
$F_0$ are closable and we denote by $S$ and $F$ respectively their closures and by $\D(S)$ and $\D(F)$ the domain of $S$ and the domain of $F$ respectively. It turns out that $S$ and $F$ are injective with dense range. Moreover, $F=S^*$.
Now, let $\Delta =S^*S$ and let $J\Delta^{1/2}$ be the polar decomposition of $S$. $S$ is called the {\bf Tomita operator}, $\Delta$ is called the {\bf modular operator} and $J$ the {\bf modular conjugation}. 

 Since $S$ is injective with dense range then the 
self-adjoint operator $\Delta^{1/2}$ is injective and $J$ is antiunitary i.e. it is  antilinear and satisfies $J^*J=J J^*=1_\H$. Moreover, 
$S^2=1_{\D(S)}$. It follows that $J\Delta^{1/2}J=\Delta^{-1/2}$, that $J^2=1_\H$ and hence that $J=J^*$. Note also that since
$S \Omega=F\Omega=\Omega$, then $\Omega \in \D(\Delta)$ and $\Delta\Omega=\Omega$. Thus, $J\Omega = \Omega$. 
Since $\Delta$ is self-adjoint with non-negative spectrum and injective then $\log \Delta$ is densely defined and self-adjoint. 
Accordingly the map $\RR \ni t \mapsto \Delta^{it}=e^{it\log\Delta}$ defines a strongly continuous one-parameter group of 
unitary operators on $\H$ leaving $\Omega$ invariant.  Note that $J\Delta^{it}J=\Delta^{it}$.

Now let $g:\H \to \H$ be a unitary operator and assume that $g\M g^{-1}=\M$ and 
$g\Omega=\Omega$. Then, for every $A\in \M$, $A\Omega$ is in the domain of $gSg^{-1}$ and $gSg^{-1}A\Omega =SA\Omega$. 
It follows that $gSg^{-1}=S$ and hence that $g\Delta g^{-1} =\Delta$ and $gJg^{-1} = J$. More generally, if $\M_1 \subset B(\H_1)$ and 
$\M_2 \subset B(\H_2)$  are two von Neumann algebras on the Hilbert spaces $\H_1$, $\H_2$ with cyclic and separating vectors
$\Omega_1 \in \H_1$, $\Omega_2 \in \H_2$ respectively and if $\phi:\H_1 \to \H_2$ is a unitary operator satisfying 
$\phi\M_1\phi^{-1} = \M_2$ and $\phi \Omega_1 =\Omega_2$ then $\phi S_1 \phi^{-1} = S_2$, where $S_1=J_1\Delta_1^{1/2}$ 
(resp. $S_2= J_2 \Delta_2^{1/2}$) is the Tomita operator associated with the pair 
$\left(\M_1, \Omega_1 \right)$ (resp. $\left(\M_2, \Omega_2 \right)$). Accordingly we also have $\phi \Delta_1 \phi^{-1} = \Delta_2$ and 
$\phi J_1 \phi^{-1} = J_2$. 
\medskip

The following theorem is the central result of the Tomita-Takesaki theory. 

\begin{theorem} (Tomita-Takesaki theorem) Let $\M \subset B(\H)$ be a von Neumann algebra, let $\Omega \in \H$ be cyclic and separating for $\M$ and let 
$S=J\Delta^{1/2}$ be the polar decomposition of the Tomita operator $S$ associated with $M$ and $\Omega$. Then, 
$$J\M J = \M' ,\quad \textrm{and} \quad \Delta^{it} \M \Delta^{-it}=\M, $$
for all $t\in \RR$.
\end{theorem}

As a consequence for every  $t \in \RR$ the  map $\M \ni A \mapsto \Delta^{it}A\Delta^{-it}$ defines an ($*$-) automorphism of 
$\M$ depending only on $t$ and on the {\bf state} (i.e. normalized positive linear functional) $\omega$ defined by 
$\omega(A) = \frac{1}{\|\Omega\|^2}(\Omega|A\Omega)$. 
This automorphism is denoted by $\sigma^\omega_t$, $t\in \RR$ and the corresponding one-parameter automorphism 
group $\RR \ni t \mapsto \sigma^\omega_t$ is called the {\bf modular group} of $\M$ associated with the state $\omega$.

\section{Preliminaries on conformal nets}
We gave the definition of M\"obius covariant net and of conformal net in the introduction. In this section we discuss some of the main properties of conformal nets that will be used in the next sections for more details and proofs see e.g. 
\cite{BGL1993,Carpi2004,FJ,FrG,GuLo96,KL04,LongoCMP2003}, see also the lecture notes in preparation \cite{LongoLectureNotes} and the PhD thesis \cite{WeiPhD}.
Note that in the literature, in some cases, M\"obius covariant nets are called conformal nets and conformal nets are called 
diffeomorphism covariant nets.

\subsection{$\diff$ and its subgroup $\mob$}
\label{SectDiff}
In this subsection we recall some notions about the ``spacetime" symmetry group of conformal nets.

Let $\s1\equiv \{z\in \CC: |z|=1\}$ be the unite circle. Moreover, let  $S_+^1 \equiv \{z\in \s1: \Im z >0 \}$ be the (open) upper semicircle and let $S_-^1 \equiv \{z\in \s1: \Im z <0 \}$ be the lower semicircle. Note that  $S_+^1,  S_-^1 \in \I$ and  $S_-^1 = (S_+^1)'$. 

The group $\diff$ is an infinite 
dimensional Lie group modeled on the real topological vector space 
$\vect$ of smooth real vectors fields on $\s1$ with the usual $C^\infty$ 
Fr\'echet topology \cite[Sect.6]{Milnor}. Its Lie algebra coincides with $\vect$ 
with the bracket given by the negative of the usual brackets of vector 
fields. Hence if $g(z), f(z)$, $z=e^{i\vartheta}$, are functions in $C^\infty(\s1,\RR)$ then 
\begin{equation}
[g(e^{i\vartheta})\frac{d}{d\vartheta}, 
f(e^{i\vartheta})\frac{d}{d\vartheta}]= 
\left((\frac{d}{d\vartheta}g(e^{i\vartheta}))f(e^{i\vartheta})-
(\frac{d}{d\vartheta}f(e^{i\vartheta}))g(e^{i\vartheta})\right)\frac{d}{d\vartheta}.
\end{equation} 
$\diff$ is connected but not simply connected, see  \cite[Sect.10]{Milnor} and \cite[Example 4.2.6]{Hamilton} and we will denote by 
$\widetilde{{\rm Diff^+}}(S^1)$ its universal covering group. The corresponding covering map will be denoted by 
$\difftilde \ni \gamma \to \dot{\gamma} \in \diff$.

For every $f\in C^\infty(S^1,\RR)$ we denote by $\RR \ni t \mapsto \Exp(tf\frac{\rm d}{{\rm d}\vartheta})$ the one-parameter subgroup of 
$\diff$ generated by the vector field $f\frac{\rm d}{{\rm d}\vartheta}$ and we denote by 
$\RR \ni t \mapsto \widetilde{\Exp}(tf\frac{\rm d}{{\rm d}\vartheta})$ the corresponding one-parameter group in $\difftilde$.
\begin{remark}
\label{simplediff} 
{\rm By a result of Epstein, Herman and Thurston \cite{Epstein1970,Herman1971,Thurston1974} $\diff$ is a simple group (algebraically). It follows that $\diff$ is generated by exponentials i.e. by the subset $\{\Exp(f\frac{\rm d}{{\rm d}\vartheta}): f\in C^\infty(S^1,\RR)  \}$. In fact, by the same reason, $\diff$ is generated by exponentials with non-dense support i.e. by the subset 
$\bigcup_{I\in \I}\{\Exp(f\frac{\rm d}{{\rm d}\vartheta}): f\in C_c^\infty(I,\RR)  \}$, cf. \cite[Remark 1.7]{Milnor}
}
\end{remark}

Recall from the introduction that, for every $I\in \I$, 
${\rm Diff}(I)$ denotes the subgroup of $\diff$ whose elements act as the identity on $I'$. 
Note that accordingly  ${\rm Diff}(I)$ does not coincide with the group of diffeomorphisms of 
the open interval $I$, as the notation might suggest, but it only corresponds to proper subgroup of the latter. If $f\in C^\infty_c(I, \RR)$, $I\in \I$, namely $f\in C^\infty(\s1,\RR)$ and ${\rm supp} f \subset I$, then $\Exp(f\frac{\rm d}{{\rm d}\vartheta}) \in {\rm Diff}(I)$.  Note that by Remark \ref{simplediff} $\diff$ is generated by $\bigcup_{I \in \I} {\rm Diff}(I)$.

For any $I\in \I$ let ${\rm Diff}_c(I)$ be the dense subgroup of ${\rm Diff}(I)$ whose elements are the orientation preserving diffeomorphisms with support in $I$ i.e. 
\begin{equation}
{\rm Diff}_c(I)\equiv \bigcup_{I_1 \in \I,\; \overline{I_1} \subset I} {\rm Diff}(I_1).
\end{equation}
By \cite{Epstein1970,Epstein1984}, see also \cite{Mather1974}, ${\rm Diff}_c(I)$ is a simple group and hence it is generated by 
$$\{\Exp(f\frac{\rm d}{{\rm d}\vartheta}): f\in C_c^\infty(I,\RR)  \}.$$

Now, let $\vectC$ be the complexification of $\vect$ and let $l_n \in \vectC$ be defined by 
$l_n=-ie^{in\vartheta}\frac{\rm d}{{\rm d}\vartheta}$. 
Then
\begin{equation}
\label{EqWitt}
[l_n,l_m]=(n-m)l_{n+m} 
\end{equation} 
for all $n,m\in\ZZ$, i.e. the complex valued vector fields $l_n$, $n\in \ZZ$, span a complex Lie subalgebra 
$\mathfrak{Witt} \subset \vectC$, the (complex) {\bf Witt algebra}. As it is well known $\mathfrak{Witt}$ admits a nontrivial central extension 
$\mathfrak{Vir}$ defined by the relations 

\begin{eqnarray}
\label{EqVir}
&&[l_n,l_m]=(n-m)l_{n+m} +\delta_{n+m,0}\frac{n^3-n}{12}k \\
\nonumber
&&[l_n,k]=0,
\end{eqnarray}  
called the {\bf Virasoro algebra}, see \cite[Lecture 1]{KaRa}.

For any $f\in C^\infty(\s1) \equiv C^\infty(\s1,\CC)$ let 
\begin{equation}
\label{EqFourier}
\hat{f}_n\equiv \frac{1}{2\pi}\int_{-\pi}^{\pi}f(e^{i\vartheta})e^{-in\vartheta} {\rm d}\vartheta \quad n\in \ZZ
\end{equation}
be its Fourier coefficients. Then the Fourier series 
$$\sum_{n\in \ZZ}\hat{f}_nl_n$$ 
is convergent in $\vectC$ to the vector field $f\frac{\rm d}{{\rm d}\vartheta}$. Thus $\mathfrak{Witt}$ is dense in $\vectC$. 

The vector fields $l_n$, $n=-1,0,1$ generated a Lie subalgebra of $\mathfrak{Witt}$ isomorphic to $\mathfrak{sl}(2,\CC)$. 
Moreover, the real vector fields $il_0$, $\frac{i}{2}(l_1 +l_{-1})$ and  $\frac{1}{2}(l_1 - l_{-1})$ generate a Lie subalgebra of 
$\vect$ isomorphic to $\mathfrak{sl}(2,\RR)$ which correspond to the three-dimensional Lie subgroup 
$\mob \subset \diff$ of {\bf M\"obius transformations} of $\s1$. It turns out that $\mob$ is isomorphic to 
$\psl2  \simeq \mathrm{PSU}(1,1)$. Moreover, the inverse image of $\mob$ in $\difftilde$ under the covering map 
$:\difftilde \to \diff$ is the universal covering group $\widetilde{\mob}$ of $\mob$.

A generic element of $\mob$ is given by a map 
\begin{equation}
\label{EqDefMob}
z\mapsto \frac{\alpha z + \beta}{\overline{\beta}z+\overline{\alpha}},
\end{equation}
where $\alpha, \beta$ are complex numbers satisfying 
$|\alpha|^2 -|\beta|^2 =1$. 

The one-parameter subgroup of rotations $r(t)\in \mob$ is given by $r(t)(z)\equiv e^{it}z$, $z\in \s1$ so that 
$r(t)=\Exp(itl_0)$
 Le $\delta(t)$ be the one-parameter subgroup of $\mob$ defined by 
\begin{equation}
\label{Eqdelta(t)}
\delta(t)(z)= \frac{z\cosh t/2  - \sinh t/2}{-z\sinh t/2 + \cosh t/2},
\end{equation}
(``dilations'') corresponding to the vector field  $\sin \vartheta \frac{{\rm d}}{{\rm d}\vartheta}$
so that 
\begin{equation}
\label{Eqdelta(t)2}
\delta(t)=\Exp\left(t \frac{l_1-l_{-1}}{2} \right).
\end{equation} 

We have $\delta(t)S^1_+ = S^1_+$ for all $t\in \RR$. Moreover, if $\gamma \in \mob$ and $\gamma S^1_+ = S^1_+$ then 
$\gamma =\delta(\alpha)$ for some $\alpha \in \RR$. As a consequence, if $I\in \I$ and 
$\gamma_1, \gamma_2 \in \mob$ satisfy 
$\gamma_i S^1_+ = I$, $i=1,2$ then $\gamma_2^{-1}\gamma_1\delta(t)\gamma_1^{-1}\gamma_2=\delta(t)$ for all $t\in \RR$. 
For every $I \in \I$ there exists $\gamma \in \mob$ such that $\gamma S^1_+ = I$. Then, the one-parameter group 
$\gamma\delta(t)\gamma^{-1}$ does not depend on the choice of $\gamma$ satisfying $\gamma S^1_+ = I$
and it will be denoted by $\delta_I(t)$ . In particular $\delta_{S^1_+}(t)=\delta(t)$, $t\in \RR$. Note also that 
$\gamma\delta_I(t)\gamma^{-1}=\delta_{\gamma I }(t)$ for all
$\gamma \in \mob$ and all $t\in \RR$. Moreover, $\mob$ is generated by $\{\delta_I(t): I\in\I,\; t \in \RR\}$.

We will also consider the orientation reversing diffeomorphism $j:\s1\to \s1$ defined by 
$j(z)\equiv \overline{z}$, $z\in \s1$. Given any $I\in \I$ we put $j_I \equiv \gamma \circ j \circ \gamma^{-1}$ where 
$\gamma \in \mob$ is such that  $\gamma S^1_+ = I$ (again $j_I$ only depends on $I$ and not on the particular choice of $\gamma$).
Clearly, $j_{S^1_+}= j$ and $j_I I= I'$ for all $I \in \I$.

\subsection{Positive-energy projective unitary representations of $\diff$ and of $\difftilde$ and positive-energy representations of 
$\mathfrak{Vir}$}
\label{SubsectPositiveEnergy}
 
By a strongly continuous projective unitary representation 
$U$ of a topological group ${\mathscr G}$ on a Hilbert space we shall always mean a strongly 
continuous homomorphism of ${\mathscr G}$ into the quotient 
${\rm U(\H)}/\mathbb T$ of the unitary group of $\H$ by the circle subgroup 
$\mathbb T$. 

Note that although for $\gamma \in {\mathscr G}$, $U(\gamma)$ is 
defined only up to a phase as an operator on $\H$, expressions 
like $U(\gamma)TV(\gamma)^*$ for any $T\in {\rm B}(\H)$ or $U(\gamma) \in \mathcal{L}$ 
for a (complex) linear subspace $\mathcal{L} \subset {\rm B}(\H)$ are unambiguous for all $\gamma \in {\mathscr G}$ and are
frequently used in this paper.  

If ${\mathscr G}=\difftilde$ then, by \cite{Bargmann}, the restriction of a strongly continuous projective unitary representation $U$ to the
subgroup $\widetilde{\mob} \subset \difftilde $ always lifts to a unique strongly continuous unitary 
representation $\tilde{U}$ of $\widetilde{\mob}$.
We then say that $U$ extends $\widetilde{U}$, and that $U$ is a {\bf positive-energy representation} if 
$\widetilde{U}$ is a positive-energy representation of $\widetilde{\mob}$, namely 
if the self-adjoint generator $L_0$ (the ``conformal Hamiltonian''), of the strongly continuous one parameter group 
$e^{itL_0} \equiv \widetilde{U}(\widetilde{\Exp}(itl_0))$, has non-negative spectrum $\sigma(L_0)\subset [0, +\infty)$, namely, it is a positive operator. 

By a positive-energy unitary representation $\pi$ of the Virasoro algebra $\mathfrak{Vir}$ we shall always mean a Lie algebra representation of $\mathfrak{Vir}$, on a complex  vector space $V$ endowed with a (positive definite) scalar product $(\cdot | \cdot)$, such that the representing operators $L_n\equiv \pi(l_n) \in {\rm End}(V)$, $n\in \ZZ$ and $K\equiv \pi(k) \in {\rm End}(V)$  satisfy the following conditions:
\begin{itemize}
\item[$(i)$] {\it Unitarity:} $(a|L_nb)=(L_{-n}a|b)$ for all $n\in \ZZ$ an all $a,b \in V$; 
\item[$(ii)$] {\it Positivity of the energy:} $L_0$ is diagonalizable on $V$ with non-negative eigenvalues i.e. we have the algebraic direct sum 
$$V=\bigoplus_{\alpha \in \RR_{\geq 0}}V_\alpha$$ 
where $V_\alpha \equiv {\rm Ker}(L_0 - \alpha 1_V)$ for all $\alpha \in \RR_{\geq 0}$;
\item[$(iii)$] {\it Central charge:} $K=c 1_V$ for some $c\in \CC$. 
\end{itemize}

By a well known result of Friedan, Qiu and Shenker \cite{FQS1984,FQS1986}, see also \cite{KaRa}, unitarity implies that the possible values of the {\bf central charge} $c$ are restricted to $c\geq 1$ or $c =c_m \equiv 1- \frac{6}{(m+2)(m+3)}$, $m \in \ZZ_{\geq 0}$. In the case $c=0$ the representation is trivial, i.e. $L_n=0$ for all $n\in \ZZ$. An irreducible unitary positive-energy representation of $\mathfrak{Vir}$ is completely determined by the central charge $c$ and the lowest eigenvalue $h$ of $L_0$. Then $h$ satisfies $h\geq 0$ if 
$c\geq 1$ and $h=h_{p,q}(m) \equiv \frac{((m+3)p-(m+2)q)^2-1}{4(m+2)(m+3)}$, $p=1,...,m+1$, $q=1,...,p$, if $c=c_m$, 
$m\in \ZZ_{\geq 0}$ (discrete series representations) and all such pairs $(c,h)$ are realized for an irreducible positive-energy representation, \cite{GKO,KaRa}. For every allowed pair $(c,h)$ the corresponding irreducible module is denoted $L(c,h)$. Note that 
$(c,0)$ is an allowed pair for every allowed value $c$ of the central charge. 

We now discuss the correspondence between unitary positive-energy representations of the Virasoro algebra $\mathfrak{Vir}$ and 
the strongly-continuous projective unitary positive-energy representations of $\difftilde$. An important tool here is given by the following estimates due to Goodman and Wallach \cite[Prop.2.1]{GoWa1985}, see also \cite[Sect.6]{Tol1999}. Similar estimates are also given in 
\cite{BS-M}.

\begin{proposition}
\label{PropGoWaEstimates} Let $\pi$ be a positive-energy unitary representation of the Virasoro algebra $\mathfrak{Vir}$ with central charge $c \in \RR_{\geq 0}$ on a complex inner product space $V$. Let $L_n \equiv \pi(l_n)$ and let $\|a\|\equiv (a|a)^{1/2}$, for all 
$a\in V$.  Then, 
\begin{equation}
\|L_n a\| \leq \sqrt{c/2}(|n| +1)^{\frac{3}{2}}\|(L_0+1_V)a\|,
\end{equation}
for all $n\in \ZZ$ and all $a\in V$. 
\end{proposition}

\begin{remark}
\label{RemarkGoWaEstimates}
{\rm Starting from Prop. \ref{PropGoWaEstimates} it is easy to prove the following estimates 
\begin{equation}
\|(L_0 + 1_V)^kL_n a\| \leq \sqrt{c/2}(|n| +1)^{k+\frac{3}{2}}\|(L_0+1_V)^{k+1}a\|,
\end{equation}
for all $k\in \ZZ_{\geq 0}$,  $n\in \ZZ$ and all $a\in V$. 
}
\end{remark}

Now let $\pi$ be a positive-energy unitary representation of the Virasoro algebra on a complex inner product space $V$ and let 
$\H$ be the Hilbert space completion of $V$. The operators $L_n$, $n\in \ZZ$ can be considered as densely defined operators with domain $V$. As a consequence of Prop. \ref{PropGoWaEstimates} and of the unitarity of $\pi$ one can define operators $L^0(f)$, 
$f\in C^\infty(\s1)$ on $\H$ with domain $V$, by 

\begin{equation}
L^0(f)a=\sum_{n\in \ZZ} \hat{f}_n L_na,
\end{equation}
for all $a\in V$, where  
\begin{equation}
\hat{f}_n = \int_{-\pi}^{\pi} 
f(e^{i\vartheta})e^{-in\vartheta}\frac{{\rm d}\vartheta}{2\pi}
\end{equation}
is the $n$-th Fourier coefficient of the smooth function $f$.  It follows from the unitarity of $\pi$ that,  $L^0(\overline{f}) \subset L^0(f)^*$,  and hence that $L^0(f)$ is closable for every $f \in C^\infty(S^1)$ and we will denote by $L(f)$ the corresponding closure. 
Let $\H^\infty$ the intersection of the domains of the self-adjoint operators $(L_0 + 1_\H)^k$, $k\in \ZZ_{\geq 0}$. Then, $\H^\infty$ 
is a common core for the operators $L(f)$, $f\in C^\infty(\s1)$ and 
\begin{equation}
\|L(f)a\| \leq \sqrt{c/2} \|f\|_{\frac{3}{2}}\|(L_0+1_\H)a\|
\end{equation}
for all  $f\in C^\infty(\s1)$ and all $a\in \H^\infty$, 
where, for every $s\geq0$ 
\begin{equation}
\|f\|_s \equiv \sum_{n\in \ZZ}(|n| + 1)^s |\hat{f}_n| .
\end{equation}
It follows $\H^\infty$ is a common invariant core for the operators $L(f)$, $f\in C^\infty(\s1)$, see \cite[Sect.6]{Tol1999} and 
\cite{GoWa1985}.
Moreover, the map $\vect \to {\rm End}(\H^\infty)$ defined by $f\frac{\rm d}{{\rm d}\vartheta} \mapsto iL(f)$, defines a projective 
representation, again denoted by $\pi$, of $\vect$ by skew-symmetric operators, namely 
$-\pi(f\frac{\rm d}{{\rm d}\vartheta}) \subset \pi(f\frac{\rm d}{{\rm d}\vartheta})^*$ and 
\begin{equation}
\left[ \pi(f_1\frac{\rm d}{{\rm d}\vartheta}),\pi(f_2\frac{\rm d}{{\rm d}\vartheta})\right] = 
\pi\left( [f_1\frac{\rm d}{{\rm d}\vartheta}, f_2\frac{\rm d}{{\rm d}\vartheta}] \right) + 
iB(f_1\frac{\rm d}{{\rm d}\vartheta}, f_2\frac{\rm d}{{\rm d}\vartheta})1_\H,
\end{equation}
on $\H^\infty$, with real valued two-cocycle $B(\cdot,\cdot)$ given by 
\begin{equation}
B(f_1\frac{\rm d}{{\rm d}\vartheta}, f_2\frac{\rm d}{{\rm d}\vartheta})\equiv 
\frac{c}{12}\int_{-\pi}^\pi \left(\frac{{\rm d}^2}{{\rm d}\vartheta^2}f_1(e^{i\vartheta}) + f_1(e^{i\vartheta} )   \right)  
\frac{\rm d}{{\rm d}\vartheta}f_2(e^{i\vartheta})\frac{{\rm d}\vartheta}{2\pi}.
\end{equation}
Now, as a consequence of Prop. \ref{PropGoWaEstimates} and Remark \ref{RemarkGoWaEstimates} together with the fact that 
$[L_0,L(f)]=iL(f')$, where $f'(e^{i\vartheta})\equiv \frac{\rm d}{{\rm d}\vartheta}f(e^{i\vartheta})$, it can be shown that one can apply 
\cite[Thm. 5.2.1]{Tol1999}, see also \cite[Sect.6]{Tol1999}, so that the projective representation $\pi$ of $\vect$ integrates to a unique strongly-continuous projective 
unitary representation $U_\pi$ of $\difftilde$. More precisely, for every $f\in C^\infty(\s1,\RR)$ the operator 
$\pi(f\frac{\rm d}{{\rm d}\vartheta})=iL(f)$ is skew-adjoint and $U_\pi$ is the unique strongly-continuous projective unitary representation of 
$\difftilde$ on $\H$ satisfying 
\begin{equation}
U_\pi\left(\widetilde{\Exp}(f\frac{\rm d}{{\rm d}\vartheta})\right)AU_\pi\left(\widetilde{\Exp}(f\frac{\rm d}{{\rm d}\vartheta})\right)^*=
e^{\pi(f\frac{\rm d}{{\rm d}\vartheta})}A e^{-\pi(f\frac{\rm d}{{\rm d}\vartheta})},
\end{equation} 
for all $f\in C^\infty(\s1,\RR)$ and all $A\in B(\H)$. Moreover, $U_\pi(\gamma)\H^\infty =\H^\infty$ for all $\gamma \in \difftilde$. 
For any  $I\in \I$ le $f_1 \in C^\infty_c(I)$ and $f_2 \in C^\infty_c(I')$, then $f_1\frac{\rm d}{{\rm d}\vartheta}$, 
$f_2\frac{\rm d}{{\rm d}\vartheta}$ generates a two-dimensional abelian Lie subalgebra of $\vect$. Since 
$B(f_1\frac{\rm d}{{\rm d}\vartheta}, f_2\frac{\rm d}{{\rm d}\vartheta})=0$, the cocycle $B(\cdot,\cdot)$ vanishes on this  Lie subalgebra and hence $\pi$ give rise to an ordinary (i.e. non-projective) Lie algebra representation to the latter which, by 
(the proof in \cite[page 497]{Tol1999} of) \cite[Thm. 5.2.1]{Tol1999}, integrates to a strongly continuous unitary representation of the abelian Lie group $\RR^2$. Hence, 
\begin{equation}
\label{EqCommuteExp}
\left[e^{\pi(f_1\frac{\rm d}{{\rm d}\vartheta})} , e^{\pi(f_2\frac{\rm d}{{\rm d}\vartheta})}\right]=0,
\end{equation}
see \cite{BS-M} for a proof of this fact based on \cite{DF1977}, see also \cite[Sect.19.4]{GJ}.

In fact $U_\pi$ factors through a strongly-continuous projective unitary representation of $\diff$, which we will again denote $U_\pi$, if and only if $e^{i2\pi L_0}$ is a multiple of the identity operator $1_\H$. 
In the latter case, as a consequence of Eq. (\ref{EqCommuteExp}), recalling that the simple group ${\rm Diff}_c(I)$ is generated by exponentials and it is dense in  ${\rm Diff}(I)$, we see that the 
representation $U_\pi$ of $\diff$ satisfies 
\begin{equation}
\label{EqDiffLocality}
U_\pi({\rm Diff}(I)) \subset U_\pi({\rm Diff}(I'))'
\end{equation}
for all $I\in \I$, see also \cite[Sect.V.2]{loke}.

With some abuse of language we simply say that the representation $\pi$ of $\mathfrak{Vir}$ integrates to 
a strongly-continuous projective unitary positive-energy of $\difftilde$. 

Conversely, let $U$ be a strongly-continuous projective unitary positive-energy representation of $\difftilde$ on a Hilbert space $\H$ and assume that the algebraic direct sum 
\begin{equation}
\H^{fin} \equiv \bigoplus_{\alpha \in \RR_{\geq 0}} {\rm Ker}(L_0 -\alpha 1_\H)
\end{equation}
is dense in $\H$. Then, using the arguments in \cite[Chapter 1]{loke}, se also \cite[Appendix A]{Carpi2004}, one can prove that there is a unique positive-energy unitary representation $\pi$ of the Virasoro algebra on $\H^{fin}$ such that $U=U_\pi$, see also \cite{NS} for related results. 
We collect the results discussed above in the following theorem.
\begin{theorem}
\label{TheoremVirDiff}
Every positive-energy projective unitary representation $\pi$ of the Virasoro algebra on a complex inner product space $V$ integrates to a unique strongly-continuous projective unitary positive-energy representation $U_\pi$ of $\; \difftilde$ on the Hilbert space completion $\H$ of $V$. Moreover, every  strongly-continuous projective unitary positive-energy representation of $\difftilde$ on a Hilbert space $\H$ containing $\H^{fin}$ as a dense subspace arises in this way. The map 
$\pi \mapsto U_\pi$ becomes one-to-one after restricting to representations $\pi$ on inner product spaces $V$ whose Hilbert space completion $\H$ satisfies $\H^{fin}=V$. These are exactly those inner product spaces such that 
$V_\alpha\equiv {\rm Ker}(L_0-\alpha1_V) \subset V$ is complete (i.e. a Hilbert space) for all $\alpha \in \RR_{\geq 0}$. $U_\pi$ is irreducible if and only if $\pi$ is irreducible i.e. if and only if the corresponding $\mathfrak{Vir}$-module is $L(c,0)$ for some $c\geq 1$ or 
$c = 1- \frac{6}{(m+2)(m+3)}$, $m=\in \ZZ_{\geq 0}$.
\end{theorem} 

\subsection{M\"{o}bius covariant nets and conformal nets on $\s1$}
\label{SubsectPrelNets}
We now discuss some properties of M\"{o}bius covariant and conformal nets on $\s1$ together with some related notions and definitions.

Here below we describe some of the consequences of the axioms of M\"{o}bis covariant and conformal nets on $\s1$ and give some comments on these referring the reader to the literature \cite{BGL1993,FJ,FrG,GuLo96} for more details and the proofs. Here $\A$ is a
 M\"{o}bius covariant net on $\s1$ acting on its vacuum Hilbert space $\H$. 

\begin{itemize}
\item[$(1)$] \emph{Reeh-Schlieder property.} The vacuum vector $\Omega$ is cyclic and separating for every $\A(I)$, $I\in\I$.
\item[$(2)$] \emph{Bisognano-Wichmann Property.} Let $I\in \I$ and let $\Delta_I$, $J_I$ be the modular operator and the modular conjugation of $\left(\A(I), \Omega \right)$.  Then we have 
\begin{equation}
 U(\delta_I(-2\pi t)) = \Delta_I^{i t},
\end{equation}
\begin{equation}
J_I\A(I_1)J_I = \A(j_I I_1),
\end{equation}
\begin{equation}
J_IU(\gamma)J_I=U(j_I\circ \gamma \circ j_I),
\end{equation}
for all $t\in \RR$, all $I_1 \in \I$ and all $\gamma \in \mob$. 
Hence the unitary representation $U:\mob \to B(\H)$ extends to an (anti-) unitary representation of  
$\mob \rtimes \ZZ_2$ 
\begin{equation}
U(j_I)=J_I 
\end{equation}
and acting covariantly on $\A$.

\item[$(3)$] \emph{Haag duality.} $\A(I')=\A(I)'$, for all $I\in\I$.
\item[$(4)$] \emph{Outer regularity.} 
\begin{equation}
 \A(I_0) = \bigcap_{I\in\I,I\supset \bar{I_0}} \A(I),\quad I_0\in\I.
\end{equation}
\item[$(5)$] \emph{Additivity.} If $I = \bigcup_{\alpha} I_\alpha$, where $I, I_\alpha \in \I$, then 
$\A(I) = \bigvee_\alpha \A(I_\alpha)$.
\item[$(6)$] \emph{Uniqueness of the vacuum.} $\A$ is irreducible if and only if ${\rm Ker}(L_0)=\CC \Omega$
 \item[$(7)$] \emph{Factoriality.} $\A$ is irreducible if and only if either $\A(I)$ is a type ${\rm III}_1$ factor  for all $I\in\I$ or 
 $\A(I)=\CC$ for all $I\in \I$. 

\end{itemize}

Note that Haag duality follows directly from the Bisognano-Wichmann property since 
\begin{equation*}
\A(I)'=J_I \A(I) J_I = \A(j_I I)= \A(I').
\end{equation*}
Note also that if ${\rm Ker}(L_0)=\CC \Omega$ then, for every $I\in \I$, $\CC \Omega$ coincides with the subspace of $\H$ 
of $U(\delta_I(t))$-invariant  vectors, see \cite[Corollary B.2]{GuLo96}. Hence, by the Bisognano-Wichmann property, the 
modular group of the von Neumann algebra $\A(I)$ with respect to $\Omega$ is ergodic i.e. the centralizer 
\begin{equation}
\{A\in \A(I): \Delta_I^{it}A\Delta_I^{-it}=A \; \forall t \in \RR \}
\end{equation}
is trivial (i.e. equal to $\CC1_\H$). It then follows from \cite{Connes1973} that either $\A(I)$ is type ${\rm III}_1$ factor or 
 $\A(I)=\CC$, see e.g. the proof of \cite[Thm.3]{Longo1979}, see also \cite{LongoLectureNotes}. 
 
Since $\mob$ is generated by $\{\delta_I(t): I\in\I,\; t \in \RR\}$, it follows from the Bisognano-Wichmann property that, for a given 
M\"obius covariant net $\A$ on $\s1$ the representation $U$ of $\mob$  is completely determined by the vacuum vector $\Omega$. 
In fact, if $\A$ is a conformal net, then, by \cite[Thm. 6.1.9]{WeiPhD} (see also \cite{CW2005}), the strongly-continuous projective unitary representation $U$ of  $\diff$ making $\A$ covariant is completely determined by its restriction to $\mob$ and hence by the vacuum vector 
$\Omega$ ({\bf uniqueness of diffeomorphism symmetry}). We will give an alternative proof of this result in this paper, see 
Thm. \ref{uniqdifftheorem}. See also \cite{Weiner2011} for related uniqueness results. 

For a given M\"obius covariant net $\A$ on $S^1$ and any subset $E\subset S^1$ with non-empty interior we define a von Neumann 
algebra $\A(E)$ by 
\begin{equation}
\A(E) \equiv \bigvee_{I \subset E, I\in \I} \A(I).
\end{equation} 
Accordingly, $\A$ is irreducible if and only if $\A(\s1)=B(\H)$.

Given two M\"obius covariant nets $\A_1$ and $\A_2$ on $\s1$, acting on the vacuum Hilbert spaces $\H_1$, $\H_2$, vacuum vectors
$\Omega_1$, $\Omega_2$  and representations $U_1$, $U_2$ of $\mob$, one can consider the {\bf tensor product net} 
$\A_1\otimes \A_2$ acting on 
$\H_1\overline{\otimes}\H_2$ with local algebras given by
\begin{equation}
\left( \A_1\otimes \A_2 \right)(I) \equiv \A_1(I) \overline{\otimes}\A_2(I),\; I\in \I,   
\end{equation}
vacuum vector 
\begin{equation}
\Omega \equiv \Omega_1\otimes \Omega_2
\end{equation}
and representation of the M\"obius group given by 
\begin{equation}
U(\gamma)\equiv U_1(\gamma) \otimes U_2(\gamma),\; \gamma \in \mob. 
\end{equation}
It is easy to see that , $\A_1\otimes \A_2$ is a M\"obius covariant net on $\s1$. In fact, if both $\A_1$ and $\A_2$ are conformal nets
also  $\A_1\otimes \A_2$ is a conformal net. One can define also the infinite tensor product (with respect to the vacuum vectors) of an infinite sequence of M\"obius covariant nets. 
However it is not necessarily true that if the nets in the sequence are all conformal nets then their infinite tensor product is also conformal but it will be in general only M\"obius covariant, \cite[Sect.6]{CW2005}.

We say that the M\"obius covariant nets $\A_1$ and $\A_2$ are {\bf unitarily equivalent} or {\bf isomorphic} if there is a unitary operator 
$\phi: \H_1 \to \H_2$ such that $\phi \Omega_1 = \Omega_2$ and $\phi \A_1(I) \phi^{-1}=\A_2(I)$ for all $I \in \I$. 
In this case we say that $\phi$ is an {\bf isomorphism of M\"obius covariant nets}. 
 Since the M\"obius symmetry is completely determined by the vacuum vectors it follows that 
\begin{equation}
\phi U_1(\gamma) \phi^{-1} =U_2(\gamma) 
\end{equation}
for all $\gamma \in \mob$. Actually, if $\A_1$ and $\A_2$ are conformal nets then the uniqueness of diffeomorphism symmetry implies that 
\begin{equation}
\phi U_1(\gamma) \phi^{-1} =U_2(\gamma) 
\end{equation}
for all $\gamma \in \diff$.

An automorphism of a M\"obius covariant net $\A$ is an isomorphism $g$ of $\A$ onto itself. Accordingly the {\bf automorphism group} 
${\rm Aut}(\A)$ of a M\"obius covariant net $\A$ on $\s1$ is given by 
\begin{equation}
{\rm Aut}(\A) \equiv \{g \in U(\H): g\A(I)g^{-1}=\A(I),\; g\Omega=\Omega\quad \textrm{for all}\; I\in \I \}.
 \end{equation}
 By the above discussion every $g \in {\rm Aut}(\A)$ commutes with the representation $U$ of $\mob$ (resp. $\diff$ if $\A$ is a conformal net). ${\rm Aut}(\A)$ with the topology induced by the strong topology of $\B(\H)$ is a topological group. 

A M\"obius covariant net $\A$ on $S^1$ satisfies the {\bf split property} if, given $I_1,I_2\in\I$ such that $\bar{I}_1\subset I_2$, there is a type I factor $\M$ such that
\begin{equation}
 \A(I_1) \subset  \M  \subset \A(I_2).
 \end{equation}

A M\"obius covariant net $\A$ on $S^1$ satisfies {\bf strong additivity} if, given two intervals $I_1,I_2\in\I$ obtained by removing a point from an interval $I\in \I$ then 
\begin{equation}
 \A(I_1) \vee \A(I_2) = \A(I).
\end{equation}

By \cite[Thm. 3.2]{DALR2001} if a M\"obius covariant net $\A$ on $S^1$ satisfies the {\bf trace class condition}, i.e. if 
${\rm Tr}(q^{L_0})< +\infty$ for all $q \in (0,1)$ then $\A$ also satisfies the split property. One can construct many examples M\"obius covariant nets without the split property by infinite tensor product construction, see \cite[Sect.6]{CW2005}. These infinite tensor product nets do not admit diffeomorphism symmetry, see  \cite[Thm. 6.2]{CW2005}. Actually, we don't know examples of irreducible conformal nets without the split property. 

By \cite[Thm. 3.1]{DoplicherLongo1984}, the automorphism group ${\rm Aut}(\A)$ of a M\"obius covariant net $\A$ on $\s1$ with the split property is compact and metrizable.

Let $\A$ be an irreducible M\"obius covariant net on $\s1$ with the split property and let $I_1,I_2,I_3,I_4 \in \I$ be four intervals, in anti-clockwise order, obtained by removing four points from $\s1$.  Let $E \equiv I_1 \cup I_3$. Then $I_2 \cup I_4$ is the interior $E'$ of $\s1 \setminus E$. 

Then,  
\begin{equation}
\A(E) \subset \A(E')' 
\end{equation}
is inclusion of type ${\rm III}_1$ factors (a subfactor). If either $\A$ is strongly additive or if $\A$ is a conformal net then the 
Jones index $[\A(E')':\A(E)]$ does not depend on the choice of the intervals $I_1,I_2,I_3,I_4 \in \I$, \cite[Prop.5]{KLM}. This index is called the {\bf $\mu$-index} of $\A$ and  it is denoted by $\mu_\A$. 

An irreducible M\"obius covariant net $\A$ is called {\bf completely rational} \cite{KLM} if it satisfies the split property and strong additivity and if the $\mu$-index $\mu_\A$ is finite. If $\A$ is an irreducible conformal net with the split property and finite $\mu$-index then it is 
strongly additive and thus completely rational, see \cite[Thm. 5.3]{LX}.

We will give various examples of irreducible conformal nets on $S^1$ starting from vertex operator algebra models in 
 Sect.  \ref{Criteria for strong locality and examples}. In this section we consider the minimal examples namely 
the Virasoro nets, see also \cite{Carpi2004,KL04,loke,WeiPhD}.  

Let $c\geq 0$ or $c =c_m \equiv 1- \frac{6}{(m+2)(m+3)}$, $m=\in \ZZ_{\geq 0}$ and let $L(c,0)$ be the corresponding irreducible unitary module $L(c,0)$ with lowest eigenvalue of $L_0$ equal to $0$. Let $\H$ be the Hilbert space completion of  $L(c,0)$. Then the 
positive-energy unitary representation of the Virasoro algebra on $L(c,0)$ integrates to a unique strongly continuous projective unitary positive-energy representation $U$ of $\diff$ which together with the map
\begin{equation}
I\in{\mathcal{I}}\mapsto{\mathcal{A}_{\mathfrak{Vir}, c}}(I)
\subset B({\mathcal{H}})
\end{equation}
defined by
\begin{equation}
{\mathcal{A}_{\mathfrak{Vir}, c}}(I)\equiv \{U(\gamma): \gamma \in {\rm Diff}(I), \; I\in \I \}'', 
\end{equation}
defines an irreducible conformal net ${\mathcal{A}_{\mathfrak{Vir}, c}}$ on $\s1$. Note that locality follows from Eq. (\ref{EqDiffLocality}). 
The uniqueness of diffeomorphism symmetry implies that two Virasoro nets are unitary equivalent if and only if they have the same central charge. For every allowed value of $c$ the Virasoro net ${\mathcal{A}_{\mathfrak{Vir}, c}}$ satisfies the trace class condition and hence the split property. For $c\leq 1$ ${\mathcal{A}_{\mathfrak{Vir}, c}}$ satisfies strong additivity \cite{KL04,Xu2005}, while for $c>1$ it does not
\cite[Sect.4]{BS-M}.
${\mathcal{A}_{\mathfrak{Vir}, c}}$ is completely rational for all $c<1$ while it has infinite $\mu$-index for all $c \geq 1$.

\subsection{Covariant subnets}
\label{SubsectCovariantSubnet}

A {\bf M\"{obius} covariant subnet} of a M\"obius covariant net
${\mathcal{A}}$ on $\s1$ is a map 
$I \mapsto \B(I)$ from $\I$ in to the set of von Neumann
algebras acting on ${\mathcal{H}}_\A$ satisfying the following properties:
\begin{equation}
{\B}(I)\subset {\A}(I) \quad \textrm{for all}\; I\in \I;
\end{equation}

\begin{equation}
{\B}(I_1)\subset {\B}(I_2)
\quad \textrm{if}\;I_1 \subset I_1, \;\;I_1, I_2 \in \I.
\end{equation}

\begin{equation} 
U(\gamma){\B}(I)U(\gamma)^*={\B}(\gamma I) \quad \textrm{for all}\;  I \in \I,\;\textrm{and all}\; \gamma \in \mob;
\end{equation}

We shall use the notation $\B\subset \A$. If $\B(I) = \CC1_\H$ for one, and hence for all, $I \in \I$ we say that $\B$ is the 
{\bf trivial subnet}.

Let $E \subset  \s1$ be any subset of the circle with non-empty interior and let 
\begin{equation}
\B(E) \equiv \bigvee_{I \subset E, I\in \I} \B(I)
\end{equation} 
so that $\B(E) \subset \A(E)$. 
Then we define $\H_\B \subset \H$ to be the closure of $\B(\s1)\Omega$. Then $\H_\B =\H$ if and only if $\B(I)=\A(I)$ for all $I\in \I$. 
Hence, typically $\Omega$ is not cyclic for $\B(\s1)$ so that $\B$ is not a M\"obius covariant net on $\s1$ 
in the precise sense of the definition. 
However one gets a M\"obius covariant net by restricting the algebras $\B(I),\;I\in\I,$
and the representation $U$ to the common invariant subspace $\H_\B$. More precisely, let $e_\B$ be the orthogonal projection of 
$\H$ onto $\H_\B$. Then $e_\B \in \B(\s1)'\cap U(\mob)'$. Then the map $\I \ni I \mapsto B(I)_{e_\B}$ together with the representation $\mob \ni \gamma \mapsto U(\gamma)\restriction_{\H_\B}$ defines a M\"obius covariant net $\B_{e_\B}$ on $\s1$ acting on $\H_\B$. 
Note that the map $b\in \B(I) \mapsto b\restriction_{\H_\B} \in \B(I)_{e_\B}$ is an isomorphism for every $I\in \I$, because of the Reeh-Schlieder property, so that, in particular, if $\A$ is irreducible and $\B$ is nontrivial then $\B(I)$ is a type $\mathrm{III}_1$ factor on $\H$ for all $I \in \I$. 
As usual, wen no confusion can arise, we will use the symbol $\B$ also to denote the M\"obius covariant net $\B_{e_\B}$ on $\H_\B$ specifying, if necessary, when $\B$ acts on $\H=\H_\A$ or on $\H_\B$.  If $\A$ is irreducible then $\B$ is irreducible on $\H_\B$. 

If $\A$ is an irreducible conformal net and $\B$ is a M\"obius covariant subnet of $\A$ then, by \cite[Thm. 6.2.29]{WeiPhD}, $\B$ is also diffeomorphism 
covariant, i.e. 
\begin{equation}
U(\gamma)\B(I)U(\gamma)^*
\end{equation}
for all $\gamma \in \diff$. Moreover, as a consequence of  \cite[Thm. 6.2.31]{WeiPhD}, there is a strongly-continuous projective 
positive-energy representation  $U_\B$ of $\difftilde$ on $\H$ such that $U(\dot{\gamma})U_\B(\gamma)^* \in \B(\s1)'$ for all 
$\gamma \in \difftilde$. It follows that the subnet $\B$ gives rise to an irreducible conformal net on $\H_\B$. Accordingly in this case we will simply say that $\B$ is a covariant subnet of $\A$.

\begin{example}
{\rm
Every irreducible conformal net   $\A$ we can define a covariant subnet $\B \subset \A$ by 
 \begin{equation}
 \B(I)\equiv \{U(\gamma): \gamma \in {\rm Diff}(I), \; I\in \I \}''.
\end{equation}
It is clear that the corresponding irreducible conformal net $\B$ on $\H_\B$ is unitarily equivalent to the Virasoro net 
${\mathcal{A}_{\mathfrak{Vir}, c}}$,  where $c$ is the central charge of the representation $U$. Accordingly, $\B$ is called the 
{\bf Virasoro subnet} of $\A$ and the inclusion $\B\subset \A$ is often denoted by ${\mathcal{A}_{\mathfrak{Vir}, c}} \subset \A$. 
We say that $c$ is the central charge of $\A$.
}
\end{example}

\begin{example}
{\rm Let $\A$ be a M\"obius covariant net and $G$ be a compact subgroup of $\mathrm{Aut}(\A)$. The {\bf fixed point subnet} 
$\A^G \subset \A$ is the M\"obius covariant subnet of $\A$ defined by
\begin{equation}
\A^G(I) \equiv \A(I)^G  = \{A\in \A(I): gAg^{-1}= A\quad \textrm{for all}\; g\in G  \} \quad I \in \I.
\end{equation}
If $G$ is finite then $\A^{G}$ is called an {\bf orbifold subnet}.
}
\end{example}

\begin{example} 
{\rm 
Let $\A$ be a M\"obius covariant net on $\s1$ and let $\B \subset \A$ be a M\"obius covariant subnet. 
The corresponding {\bf coset subnet} $\B^c \subset \A$ is the  M\"obius covariant subnet of $\A$ defined by 
\begin{equation}
\B^c(I) \equiv \B(\s1)'\cap \A(I),\quad I \in \I,
\end{equation}
see \cite{Koster2004,LongoCMP2003,Xu2000}.
If $\A$ is an irreducible conformal net and $\B$ is a covariant subnet then, by the results in \cite{Koster2004}, we have 
$\B^c(I)= \B(I)' \cap \A(I)$, for all $I \in \I$, see also \cite[Corollary 6.3.6]{WeiPhD}. 
}
\end{example} 

If $\A$ is an irreducible M\"obius covariant net and $\B \subset \A$ is a M\"obius covariant subnet then the Jones index 
$[\A(I):\B(I)]$ of the subfactor $\B(I) \subset \A(I)$ does not depend on the choice of $I\in \I$. The index $[\A:\B]$ of the subnet 
$\B \subset \A$ is then defined by $[\A:\B]\equiv [\A(I):\B(I)]$, $I\in \I$. Assuming that $[\A:\B] < +\infty$ then, by 
\cite[Thm 24]{LongoCMP2003}, $\A$ is completely rational if and only if $\B$ is completely rational. Moreover, the set of subnets 
$\C \subset \A$ such that $\B \subset \C$ (intermediate subnets) is finite as a consequence of \cite[Thm 3]{LongoCMP2003}.

\section{Preliminaries on vertex algebras} 

\subsection{Vertex algebras}
\label{susectVA}
In this paper, unless otherwise stated, vector spaces and vertex 
algebras are assumed to be over the field $\CC$ of complex numbers. 
We shall use the formulation of the book \cite{Kac} with the 
emphasis on locality. For other standard references on the subject see  \cite{FHL,FLM,LL2004,MaNa}. We will mainly consider local 
(i.e.\! not {\it super-local}) vertex algebras. Thus, differently from \cite{Kac}, we will use the term vertex algebra only for the 
local case.

Let $V$ be a vector space. A formal series
$a(z)=\sum_{n\in\ZZ}a_{(n)}z^{-n-1}$ with coefficients $a_{(n)}\in 
{\rm End}(V)$ is called a {\bf field}, if for every $b\in V$ we have
$a_{(n)}b=0$ for $n$ sufficiently large. A {\bf vertex algebra} is 
a (complex) vector space $V$ together with a given vector 
$\Omega\in V$ called the
{\bf vacuum vector}, an operator $T\in {\rm End}(V)$ called the 
{\bf infinitesimal translation operator}, and a linear map from $V$ 
to the space of fields on $V$ (the {\bf state-field correspondence})
\begin{equation}
a\mapsto Y(a,z)=\sum_{n\in\ZZ}a_{(n)}z^{-n-1}
\end{equation}
satisfying:

$(i)$ {\bf Translation covariance:} $[T,Y(a,z)]=\frac{d}{dz}Y(a,z).$

$(ii)$ {\bf Vacuum:} $T\Omega=0, \; Y(\Omega,z)={1}_V, \; 
a_{(-1)}\Omega = a$.

$(iii)$ {\bf Locality:} For all $a, b \in V$, $(z-w)^N[Y(a,z),Y(b,w)]=0$ 
for a sufficiently large non-negative integer $N$.

\noindent 
The fields $Y(a,z)$, $a\in V$, are called {\bf vertex operators}.

\begin{remark}
{\rm Translation covariance is equivalent to 
\begin{equation}
\label{EqT1}
[T,a_{(n)}] = -na_{(n-1)}, \quad a \in V, n\in \ZZ.
\end{equation} 
If $a$ is a given vector in $V$ it follows from the field property 
that there is a smallest non-negative integer $N$ such that 
$a_{(n)}\Omega=0$ for all $n\geq N$. It follows that 
$0=Ta_{(N)}\Omega=[T,a_{(N)}]\Omega= -Na_{(N-1)}\Omega$ and hence 
$N=0$. As a consequence, in the definition of 
vertex algebras, the condition $a_{(-1)}\Omega=a$ can be replaced  
by the stronger one $Y(a,z)\Omega|_{z=0}=a$.} 
\end{remark}

\begin{remark} 
\label{RemarkTcommutator}
{\rm In every vertex algebra one always have 
\begin{equation}
\label{EqT2}
[T,Y(a,z)] = Y(Ta,z), 
\end{equation}
see \cite[Corollary 4.4 c]{Kac}. }
\end{remark}

With the above definition of vertex algebras, the
so-called {\bf Borcherds identity} (or {\it Jacobi identity}), 
i.e.\! the 
equality
\begin{eqnarray}
\label{B-id}
\nonumber
\sum_{j=0}^{\infty}
\left(\begin{matrix} m \\ j \end{matrix}\right)
\left(a_{(n+j)}b\right)_{(m+k-j)}c =
\sum_{j=0}^{\infty}(-1)^j\left(\begin{matrix} n \\ j
\end{matrix}\right)a_{(m+n-j)}b_{(k+j)}c \\ -
\sum_{j=0}^{\infty}(-1)^{j+n}\left(\begin{matrix} n \\ j
\end{matrix}\right)b_{(n+k-j)}a_{(m+j)}c,\;\;\;\;
\;a,b,c\in V, \,m,n,k\in \ZZ,
\end{eqnarray}
is not an axiom, but a consequence, see \cite[Sect.\! 4.8]{Kac}. 

For future use we recall here the following useful identity known as 
Borcherds commutator formula which follows directly from 
Eq. (\ref{B-id}) after setting $n=0$ (see also \cite[Eq. (4.6.3)]{Kac}). 
\begin{equation}
\label{commutatorformula}
[a_{(m)},b_{(k)}] = 
\sum_{j=0}^{\infty}
\left(\begin{matrix} m \\ j \end{matrix}\right)
\left(a_{(j)}b\right)_{(m+k-j)}, \quad m, k \in \ZZ.
\end{equation}

We shall call a linear subspace $W\subset V$ a {\bf vertex subalgebra}, 
if $\Omega \in W$ and $a_{(n)}b\in W$ for all $a,b\in W, n\in \ZZ$. 
Since $Ta=-a_{(-2)}\Omega$, a vertex subalgebra is always $T$-invariant
and thus $W$ inherits the structure of a vertex algebra. The intersection of a family of vertex subalgebras 
is again a vertex subalgebra. Accoringly for every subset $\mathscr{F} \subset V$ there is a smallest vertex subalgebra 
$W(\mathscr{F})$ containing $\mathscr{F}$, the vertex subalgebra generated by $\mathscr{F}$. If $W(\mathscr{F})=V$ we say 
that $V$ {\bf is generated by} $\mathscr{F}$.

We shall call a subspace $\mathscr{J} \subset V$ an {\bf ideal} if it is
$T$-invariant and $a_{(n)}b \in \mathscr{J}$ for $a\in V$, $b\in \mathscr{J}$, $n\in \ZZ$. 
If $\mathscr{J}$ is an ideal then we also have $b_{(n)}a \in \mathscr{J}$ 
for $a\in V$, $b\in \mathscr{J}$, $n\in \ZZ$, see \cite[Eq. (4.3.1)]{Kac}. 
Conversely if a subspace $\mathscr{J} \subset V$ satisfies $a_{(n)}b \in J$ 
and $b_{(n)}a \in \mathscr{J}$ for $a\in V$, $b\in \mathscr{J}$, $n\in \ZZ$ then it is 
T-invariant and hence an ideal.  
A vertex algebra $V$ is {\bf simple} if every ideal $\mathscr{J}\subset V$ is 
either $\{0\}$ or $V$.

A {\bf homomorphism}, resp.\! 
{\bf antilinear homomorphism},
from a vertex algebra $V$ to a vertex algebra $W$ is a linear, resp.\!
antilinear, map $\phi:V\to W$ such that 
$\phi(a_{(n)}b)=\phi(a)_{(n)}\phi(b)$ for all $a,b\in V$ and 
$n\in \ZZ$. Sometimes we shall simply write $\phi a$ instead of $\phi(a)$.

Accordingly, one defines the notions of {\bf automorphisms}
and {\bf antilinear automorphisms}. Note that if $g$ is an
automorphism or an antilinear automorphism, then
\begin{equation}
g(\Omega)  =  g(\Omega)_{(-1)}\Omega =
g(\Omega_{(-1)}g^{-1}(\Omega)) =g(g^{-1}\Omega) = \Omega.
\end{equation}
Moreover, 
\begin{equation}
g(Ta)= g(a_{(-2)}\Omega)= g(a)_{(-2)}\Omega = Tg(a),
\end{equation}
for all $a\in V$, i.e. $g$ commutes with $T$.

Let $M$ be a vector space, and suppose that for each 
$a\in V$ there is a field on $M$
\begin{equation} 
Y^M(a,z)=\sum_{n\in \ZZ}a^M_{(n)}z^{-n-1}, \quad 
a^M_{(n)}\in {\rm End}(M) 
\end{equation} 
that the map $a\mapsto Y^{M}(a,z)$ is linear.
We shall say that $M$ (with this action) is a {\bf module} over the
vertex algebra $V$, if $Y^M(\Omega,z)=1_M$ and the Borcherds identity 
holds on $M$ i.e. 
\begin{eqnarray}
\label{modulesB-id}
\nonumber
\sum_{j=0}^{\infty}
\left(\begin{matrix} m \\ j \end{matrix}\right)
\left(a_{(n+j)}b\right)^M_{(m+k-j)}c =
\sum_{j=0}^{\infty}(-1)^j\left(\begin{matrix} n \\ j
\end{matrix}\right)a^M_{(m+n-j)}b^M_{(k+j)}c \\ -
\sum_{j=0}^{\infty}(-1)^{j+n}\left(\begin{matrix} n \\ j
\end{matrix}\right)b^M_{(n+k-j)}a^M_{(m+j)}c,\;\;\;\;
\;a,b\in V,\;c\in M \,m,n,k\in \ZZ.
\end{eqnarray}
Accordingly, one defines the notions of 
{\bf module-homomorphism}, {\bf submodules} and {\bf irreducibility}.

Every vertex algebra $V$ becomes itself a $V$-module  
by setting $Y^{V}(a,z)=Y(a,z)$. This module is called the 
{\bf adjoint module}. If the adjoint module is irreducible then 
$V$ is clearly a simple vertex algebra  but the converse is not 
true in general since the submodules of the adjoint module $V$ 
need not to be $T$-invariant.

\subsection{Conformal vertex algebras}
\label{subsect.ConformalVA}
Let $V$ be a vector space and let  
$L(z) =\sum_{n\in \ZZ}L_nz^{-n-2}$ be a field on $V$. 
If the endomorphisms 
$\{L_n: n\in\ZZ\}$ satisfy the Virasoro algebra relations
\begin{equation}
\label{vir}
\,[L_n,L_m]=(n-m)L_{n+m}+\frac{c}{12}(n^3-n)\delta_{-n,m}1_V
\end{equation}
with central charge $c\in \CC$, then $L(z)$ is called a 
{\bf Virasoro field}. 

If $V$ is a vertex algebra we shall call $\nu \in V$ a 
{\bf Virasoro vector} if the corresponding vertex operator 
$Y(\nu,z)=\sum_{n\in \ZZ}L_nz^{-n-2}$, $L_n=\nu_{(n+1)}$, is
a Virasoro field.  

As in \cite{Kac} we shall call a Virasoro vector $\nu \in V$ a 
{\bf conformal vector} if $L_{-1}=T$ and $L_0$ is diagonalizable on $V$. 
The corresponding vertex operator $Y(\nu, z)$ is called an 
{\bf energy-momentum field} and $L_0$ a 
{\bf conformal Hamiltonian} for the vertex algebra $V$. 
A vertex algebra V together with a fixed conformal vector 
$\nu \in V$ is called a {\bf conformal vertex algebra}, 
see \cite[Sect.4.10]{Kac}.  If $c$ is the central charge of the representation of the Virasoro algebra given by 
the operators $L_n=\nu_{(n+1)}$, $n\in \ZZ$ we say that $V$ has central charge $c$. 

\begin{remark}
{\rm
Every submodule of the
adjoint module of a conformal vertex algebra $V$  is 
$T$-invariant and hence it is an ideal of $V$. Accordingly 
$V$ is simple if and only if its adjoint module is irreducible.} 
\end{remark}

Let $V$ be a conformal vertex algebra and 
let $Y(\nu,z)=\sum_{n\in \ZZ}L_nz^{-n-2}$ be the corresponding 
energy-momentum field. Set $V_\alpha = {\rm Ker}(L_0-\alpha 1_V)$, 
$\alpha \in \CC$. The fact that $L_0$ is diagonalizable means that 
$V$ is the (algebraic) direct sum of the subspaces $V_\alpha$ i.e. 
\begin{equation}
V = \bigoplus_{\alpha \in \CC} V_\alpha
\end{equation} 
is graded by $L_0$. 

A non-zero element $a\in V_\alpha$ is called a {\bf homogeneous}
element of {\bf conformal weight} (or dimension) 
$d_a=\alpha$.
For such an element we shall set
\begin{equation}
a_{n}\equiv a_{(n+d_a-1)}, \;\;\; n\in \ZZ-d_a.
\end{equation}
With this convention, 
\begin{equation}
\label{EqHomogeneousa_n}
Y(a,z)=\sum_{n\in\ZZ-d_a} a_n z^{-n-d_a} \;. 
\end{equation}
We have the following commutation relations (\cite[Sect.4.9 and Sect.4.10]{Kac})

\begin{equation}
\label{[L_0,a_n]}
[L_0, a_n]=-na_n 
\end{equation}
\begin{eqnarray}
\label{l_-1commutation}
& [L_{-1}, a_n] = (-n -d_a +1) a_{n-1}  \\
\label{l_1commutation}
& [L_{1},a_n]  =  -(n -d_a +1) a_{n+1} + (L_1a)_{n+1}  
\end{eqnarray}
for all homogeneous $a\in V$, $n\in \ZZ$.  

Note that it follows from Eq. (\ref{[L_0,a_n]}) that 
\begin{equation}
\label{L_0cova_n}
e^{\alpha L_0}a_n e^{-\alpha L_0} = e^{- n\alpha}a_n, \quad \alpha \in \CC.
\end{equation}

A homogeneous vector $a$ in a 
conformal vertex algebra $V$ and the corresponding field $Y(a,z)$ 
are called {\bf quasi-primary} if $L_1a=0$ and {\bf primary} 
if $L_na=0$ for every integer $n>0$. Since 
$L_n\Omega = \nu_{(n+1)}\Omega =0$ for every integer $n\geq -1$, 
the vacuum vector $\Omega$ is a primary vector in $V_0$.
Moreover, it follows by the Virasoro algebra relations that 
the conformal vector $\nu$ is a quasi-primary vector in $V_2$ 
which cannot be primary if $c \neq0$.

We have the following commutation relations:
\begin{equation}
\label{EqQuasi-Primary/PrimaryCommutation}
[L_{m}, a_n] = \left( (d_a -1)m- n \right) a_{m+n},  
\end{equation}
for all primary (resp. quasi-primary) $a \in V$, for all $n\in \ZZ$ and all $m \in \ZZ$ (resp. $m\in \{-1,0,1\}$), see e.g. \cite[Cor.4.10]{Kac}.

We shall say that a conformal vertex algebra $V$ and the 
corresponding conformal vector are of {\bf CFT type} if 
$V_\alpha =\{0\}$ for $\alpha \notin \ZZ_{\geq 0}$ and $V_0 =\CC\Omega$. 
If $V$ is of CFT type, or more generally if  $V_\alpha =\{0\}$ for $\alpha \notin \ZZ$ we can define the operators $a_n$, $n\in \ZZ$, for all 
$a \in V$, by linearity. In this way Eq. (\ref{[L_0,a_n]}) still holds for any $a\in V$. Moreover, we have 
\begin{equation}
\label{EqGenerala_n}
Y(z^{L_0}a,z)=\sum_{n\in \ZZ}a_n z^{-n} 
\end{equation}
for all $a\in \ZZ$, cf. \cite[Eq. (5.3.13)]{FHL}. It is obvious that if  $a\in V$ is homogeneous then $a_n=0$ for all $n\in \ZZ$ implies $a=0$ and it is not hard to see that this is true also for arbitrary $a\in V$.

In general a vertex algebra $V$ can have more then one conformal vector 
and hence more then one structure of conformal vertex algebra even after 
fixing the grading. When $V$ has a conformal vector $\nu$ of CFT type 
the conformal vectors of $V$ giving the same grading as $\nu$ are 
described by the following proposition.  

\begin{proposition} \label{morestress}
Let $V$ be a vertex algebra and let $\nu \in V$ be a conformal vector 
of CFT type with corresponding energy-momentum field 
$Y(\nu,z)=\sum_{n \in \ZZ}L_n z^{-n-2}$. 
Then a vector $\tilde{\nu}\in V$ such that $\tilde{\nu}_{(1)}=L_0$ 
is a conformal vector if and only if 
\begin{equation}
\tilde{\nu}=\nu+Ta 
\end{equation} 
where $a \in V$ is such that $L_0a=a$ and $a_{(0)}=0$. 
In this case we have $a=\frac{1}{2}L_1\tilde{\nu}$.
\end{proposition} 
\begin{proof} 
Let $V_n ={\rm Ker}(L_0-n1_V)$, $n\in \ZZ_{\geq 0}$. Since by assumption $\nu$ is of 
CFT type we have 
$$V=\bigoplus_{n\in \ZZ_{\geq 0}}V_n,\; V_0=\CC \Omega .$$
If $\tilde{\nu}$ is a conformal vector satisfying $\tilde{\nu}_{(1)}=L_0$
then $\tilde{\nu} \in V_2$.

From Borcherds identity (\ref{B-id}) with $m=-1$, $n=1$ and $k=0$ 
we get:
\begin{eqnarray*}
\sum_{j=0}^{\infty}
\left(\begin{matrix} -1 \\ j \end{matrix}\right)
\left(\nu_{(1+j)}\tilde{\nu}\right)_{(-1-j)}\Omega =
\sum_{j=0}^{\infty}(-1)^j\left(\begin{matrix} 1 \\ j
\end{matrix}\right)\nu_{(-j)}\tilde{\nu}_{(j)}\Omega \\ -
\sum_{j=0}^{\infty}(-1)^{j+1}\left(\begin{matrix} 1 \\ j
\end{matrix}\right)\tilde{\nu}_{(1-j)}\nu_{(j-1)}\Omega
=\tilde{\nu}_{(1)}\nu=2\nu.
\end{eqnarray*}
Since $\tilde{\nu} \in V_2$ we have $\nu_{(1+j)}\tilde{\nu}=0$
for $j>2$ and since $\nu_{(3)}\tilde{\nu} \in V_0 =\CC \Omega$ 
we have $(\nu_{(3)}\tilde{\nu})_{(-3)}=0$. It follows that 
$2\nu = (\nu_{(1)}\tilde{\nu})_{(-1)}\Omega 
-(\nu_{(2)}\tilde{\nu})_{(-2)}\Omega 
= \nu_{(1)}\tilde{\nu}
-T\nu_{(2)}\tilde{\nu} =2\tilde{\nu}-TL_1\tilde{\nu}.$ 

Now with $a= \frac{1}{2}L_1\tilde{\nu}$ we have $L_0a=a$ and 
$\tilde{\nu}=\nu+Ta$. Hence $a_{(0)}=(Ta)_{(1)}=
\tilde{\nu}_{(1)}-\nu_{(1)}=0$. 

Conversely let us assume that $a\in V_1$ and $a_{(0)}=0$. Let 
$\tilde{\nu}=\nu+Ta$ and let $Y(\tilde{\nu},z)=\sum_{n\in\ZZ}
\tilde{L}_nz^{-n-2}$ be the corresponding vertex operator. 
Then we have $\tilde{L}_{-1}= L_{-1} =T$ and $\tilde{L}_{0}= L_{0}$. 
Moreover, $\tilde{\nu} \in V_{2}$. 
Now $\tilde{L}_{n}\tilde{\nu} \in V_{2-n}$ and hence, using the fact 
that $\nu$ is of CFT type we find 
$\tilde{L}_{n}\tilde{\nu}=0$ for $n >2$ and 
$\tilde{L}_{2}\tilde{\nu}=\tilde{c}\Omega$. Thus $\tilde{\nu}$ 
is a conformal vector by \cite[Thm.4.10 (b)]{Kac}. 
Finally recalling that $V_0=\CC\Omega$ and that $L_1\nu=0$ we 
find $L_1\tilde{\nu}=L_1Ta = [L_1,L_{-1}]a = 2a$, because $L_1a \in V_0 = \CC\Omega$ and hence 
$L_{-1}L_1a=0$. 
\end{proof}

\subsection{Vertex operator algebras and invariant bilinear
forms}\label{subsectionVOA&IBforms}

A {\bf vertex operator algebra (VOA)} is a conformal vertex algebra
such that the corresponding energy-momentum field 
$Y(\nu,z)=\sum_{n\in\ZZ} L_n {z^{-n-2} }$ and homogeneous subspaces
${\rm Ker}(L_0-\alpha{1_V})$ satisfy the following additional
conditions:

$(i)$ $V=\bigoplus_{n\in\ZZ} V_n$, i.e.\! 
${\rm Ker}(L_0 - \alpha 1_V)=\{ 0 \}$ for $\alpha \notin \ZZ$. 

$(ii)$ $V_n=\{0\}$ for $n$ sufficiently small.

$(iii)$ ${\rm dim}(V_n) < \infty$.

\begin{remark}
\label{CFTtypeRemark}
{\rm If $V_0=\CC\Omega$, then condition 
(ii) is in fact equivalent to the stronger condition $V_n=\{0\}$ 
for all $n<0$. Indeed, by \cite[Prop.\! 1]{roitman}, for $n<0$ we have 
that $V_n=L_1^{1-n}V_1\subset L_1^{-n}V_0$ and hence if $V_0=\CC\Omega$, 
then $V_n=\{0\}$. Hence in this case $V$ is of CFT type. 
}
\end{remark}

To introduce the notion of invariant bilinear forms, first we
shall talk about the restricted dual $V'$ of $V$. As a graded
vector space it is defined as
\begin{equation}
V'= \mathop{\bigoplus}_{n \in \ZZ} V_n^*
\end{equation}
i.e.\! it is the direct sum of the duals $V_n^*$ 
of the finite-dimensional vector spaces $V_n$, $n\in \ZZ$. The  
point is that $V'$ can be naturally endowed with a $V$-module
structure. Denote by $\langle\cdot,\!\cdot\rangle$ 
the pairing between $V'$ and $V$. For each $a\in V$,
the condition
\begin{equation}
\label{contragredient}
\langle Y'(a,z) b', c\rangle = \langle b' ,Y 
(e^{zL_1}(-z^{-2})^{L_0}a,z^{-1})c  
\rangle \;\;\;c\in V, b'\in V' 
\end{equation}
determines a field 
$Y'(a,z)$ on $V'$ and one has that
the map $a\mapsto Y'(a,z)$ makes $V'$ a $V$-module,  see
\cite[ Sect.5.2]{FHL}.
The module $V'$ is called the {\bf contragredient module} and the fields 
$Y'(a,z)$ {\bf adjoint vertex operators}. Note however that 
the endomorphisms $a'_{(n)}\in {\rm End}(V)$ in the formal series 
$Y(a'_{(n)},z) =\sum_{n\in \ZZ}a'_{(n)}z^{-n-1}$ are not the adjoint of 
the endomorphisms $a_{(n)}$ in the usual sense. 
Note also that we have 
\begin{equation}
\langle L'_na', c\rangle = 
\langle a', L_{-n}b \rangle \quad a' \in V', \;b \in V,
\;n\in \ZZ, 
\end{equation}
where $L'_n = \nu'_{(n+1)}$. It follows that $V'$ is 
a $\ZZ$-graded $V$ module in the sense that $L'_0a'=na'$ for 
$a'\in V'_n \equiv V_n^*$.  
  
It should be clear from the definition that the $V$-module structure on 
$V'$ depends on the conformal vector $\nu$ (more precisely on $L_1$) and 
not only on the vertex algebra structure of $V$. 

An {\bf invariant bilinear form} on $V$ is a bilinear form $(\cdot,
\cdot)$ on $V$ satisfying 
\begin{equation} \label{invbil} 
(Y(a,z) b, c) =(b, Y (e^{zL_1}(-z^{-2})^{L_0}a,z^{-1})c) 
\quad a,b,c\in V. 
\end{equation}
As the module structure on $V'$, whether a bilinear form is invariant on 
$V$ depends on the choice of the conformal vector giving to the vertex 
algebra $V$ the structure of a VOA. 
By straightforward calculation one finds that a 
bilinear form $(\cdot,\!\cdot)$ on a vertex operator algebra $V$ is invariant if and only if
\begin{equation} \label{invbil2} 
(a_n b, c)  = (-1)^{d_a}\sum_{l \in \ZZ_{\geq 0}}
\frac{1}{l!}(b, (L^l_1a)_{-n}c ) 
\end{equation} 
for all $b, c \in V$ and all homogeneous $a\in V$. In particular, in case 
of invariance, it follows that 
\begin{equation}
(L_na,b)=(a,L_{-n}b)\;\;a,b\in V, n\in\ZZ 
\end{equation} 
and hence, by considering $n=0$, that $(V_k,V_l)=0$ whenever $k\neq l$. 
Thus the linear functional $(a,\cdot)$ is in the restricted dual for every 
$a\in V$ and one can see that the map $a\mapsto (a,\cdot )$ is a module 
homomorphisms from $V$ to $V'$. Conversely, if $\phi:V\to V'$ is a 
module homomorphism, then the bilinear form defined by the formula
\begin{equation} 
(a,b) \equiv \langle \phi (a) , b \rangle 
\end{equation}
is invariant. 
Since the homogeneous subspaces $V_n\; (n\in \ZZ)$ are 
finite-dimensional, every $V$-module homomorphism from $V$ to $V'$, being 
grading preserving, is injective if and only if it is surjective. In 
particular, there exists a non-degenerate invariant 
bilinear form on $V$ if and only if $V'$ is isomorphic to $V$ as a 
$V$-module. In the following proposition we list some useful
facts concerning invariant bilinear forms for later use.

\begin{proposition}
\label{factsIB-forms}
Let $V$ be a VOA. Then:
\begin{itemize} 
\item[(i)]
Every invariant bilinear form on $V$ is symmetric.

\item[(ii)]
The map $(\cdot,\!\cdot) \mapsto (\Omega,\cdot)\restriction_{V_0}$ gives a linear 
isomorphism from the space of invariant bilinear forms onto 
$(V_0/L_1V_1)^*$.

\item[(iii)] If $V$ is a simple VOA then every non-zero invariant 
bilinear form on $V$ is non-degenerate. Moreover, if $V$ has a 
non-zero invariant bilinear form $(\cdot,\cdot)$ then every 
invariant bilinear form on $V$ is of the form 
$\alpha(\cdot,\cdot)$ for some complex number $\alpha$. 

\item[(iv)] If $V$ has a non-degenerate invariant bilinear form and 
$V_0 =\CC\Omega$ then $V$ is a simple VOA. 
\end{itemize}
\end{proposition}

\begin{proof}
For $(i)$ and $(ii)$ see \cite[Prop.\! 2.6]{li} and \cite[Thm.3.1]{li}, respectively. 

$(iii)$. Let $(\cdot, \cdot)$ a non-zero invariant bilinear form on 
$V$. 
As a consequence of Eq. (\ref{invbil2}), the subspace $$\mathscr{N} \equiv \{a\in V : (a,b)=0\; \forall b \in V\}$$ is an
ideal of $V$, which by assumption is not equal to $V$.
Hence if $V$ is simple, then $\mathscr{N}=\{0\}$, i.e.\! $(\cdot,\!\cdot)$
is non-degenerate. Now let $\{\cdot, \cdot\}$ be another  invariant
bilinear form on $V$. If it is zero there is nothing to prove
and hence we can assume that it is non-degenerate. Then there 
exists a $V$-module isomorphism
$\phi: V \mapsto V$ such that $\{a, b\} = (\phi (a), b)$ for all
$a,b \in V$. Since $\phi$ commutes with every $a_n$, $a\in V$,
$n\in \ZZ$, $V$ is a simple VOA and hence an irreducible 
$V$-module, $\phi$ must be a multiple of the identity by 
Schur's lemma because every VOA has countable dimension, 
see e.g. \cite[Lemma 2.1.3]{CG}. 
Hence there is a complex number $\alpha$ such that 
$\{a, b\} = \alpha (a, b)$ and the claim follows.

$(iv)$ If $\mathscr{J}$ is an ideal of $V$, then $L_0 \mathscr{J} \subset \mathscr{J}$ 
and hence $\mathscr{J}=\bigoplus_{n\in \ZZ} (V_n\cap \mathscr{J})$. If 
$\Omega\in \mathscr{J}$ then of course $\mathscr{J}=V$. On the other hand, 
if $\Omega\notin \mathscr{J}$ and $V_0=\CC\Omega$ we have that
$V_0\cap \mathscr{J}=\{0\}$ and so 
$\Omega\in \mathscr{J}^\circ \equiv \{a\in V: (a,b)=0 \; \forall b \in \mathscr{J}\}$. 
However,
$\mathscr{J}^\circ$ is clearly an ideal, and hence it coincides with  
$V$. Thus $\mathscr{J} =\{0\}$ by the non-degeneracy of $(\cdot,\cdot)$.
\end{proof}

Note that by $(ii)$, if $V_0=\CC\Omega$, then a non-zero invariant bilinear 
form exists if and only if $L_1V_1=\{0\}$. In this case, again by $(iii)$, there 
is exactly one invariant bilinear form $(\cdot,\cdot)$ which is 
{\bf normalized} i.e. such that $(\Omega,\Omega)=1$. Similarly, if we assume 
that $V$ is simple, then we see from $(iii)$ that there is at most one 
normalized invariant bilinear form on $V$. 

\begin{remark} {\rm One can define invariant bilinear forms with similar properties for conformal vertex algebras such that $L_0$ has only 
integer eigenvalues but without assuming that the corresponding eigenspaces $V_n$, $n\in \ZZ$ have finite dimension, see \cite{roitman}.
}
\end{remark}

\begin{proposition}\label{uniquestress}
Let $V$ be a vertex algebra with a conformal vector $\nu$ and assume that 
the corresponding conformal vertex algebra is a VOA 
such that $V_0=\CC\Omega$ and having a non-degenerate invariant bilinear 
form.
Moreover, let $\tilde{\nu}\in V$ be another conformal vector such that 
$\tilde{\nu}_{(1)}=\nu_{(1)}$ and assume that there is 
still a non-degenerate invariant bilinear form on $V$ for the 
corresponding VOA structure. Then $\tilde{\nu}=\nu$.  
\end{proposition}
\begin{proof}
Let $(\cdot,\!\cdot)$ be the unique normalized invariant bilinear 
form on $V$ with respect to the conformal vector $\nu$ and let 
$\tilde{\nu}$ be another conformal vector with the properties in the 
proposition. By Remark \ref{CFTtypeRemark} $\nu$ is a conformal 
vector of CFT type and hence, by Prop. \ref{morestress},
$\tilde{\nu}=\nu + T\frac{1}{2}L_1\tilde{\nu}$,
where $L_1= \nu_{(2)}$. 
Hence $\tilde{L}_1 \equiv \tilde{\nu}_{(2)}= L_1  
-(L_1\tilde{\nu})_{(1)}$
Let us assume that $L_1\tilde{\nu}\neq 0$. 
Since $(\cdot,\cdot)$ is non-degenerate, there is $b\in V_1$ such that 
$(L_1\tilde{\nu} , b)\neq 0$. Thus
\begin{equation}
(\Omega,\tilde{L}_1b)=(\Omega,(L_1-(L_1\tilde{\nu})_{(1)})b)=
(L_1\tilde{\nu},b)\neq 0.   
\end{equation}
Hence $\tilde{L}_1V_1\neq\{0\}$ and by Prop. \ref{factsIB-forms} 
$(ii)$ there is no non-zero invariant bilinear form on $V$ corresponding 
to $\tilde{\nu}$. 
\end{proof}

\begin{remark}\label{remarkuniquestress}  {\rm It follows from the results in \cite{roitman}  that Prop. \ref{uniquestress} still holds true if 
$V$ with the conformal vector $\nu$ is a conformal vertex algebra of CFT type, namely the assumption that the $L_0$ eigenspaces $V_n$, $n\in \ZZ$ have finite dimension is not really needed.
}
\end{remark}
\begin{remark} {\rm Prop. \ref{uniquestress} can be considered as a VOA analogue of the uniqueness results for diffeomorphism symmetry proved in \cite{CW2005} and \cite{WeiPhD}. At the end of Sect. \ref{sectionstronglylocalVOA}  it will be shown that the uniqueness results in \cite{CW2005,WeiPhD} can be proved starting from Prop. \ref{uniquestress}. 
}
\end{remark}

\begin{corollary}\label{autcorollary}
Let $V$ be a VOA with energy-momentum 
field $Y(\nu,z)=\sum_{n\in \ZZ}L_n z^{-n-2}$. Assume that  
$V_0=\CC\Omega$ and that $V$
has a non-degenerate invariant bilinear form $(\cdot,\!\cdot)$. 
Then for a vertex algebra automorphism or antilinear automorphism 
$g$ of $V$, the following are equivalent.
\begin{itemize}
\item[{\rm (i)}] $g$ is grading preserving i.e. 
$g(V_n)=V_n$ for all $n\in \ZZ$.
\item[{\rm (ii)}] $g$ preserves $(\cdot,\!\cdot)$ i.e. either 
$(g(a),g(b))=(a,b)$ for all $a,b\in V$
if $g$ is linear, or
$(g(a),g(b))=\overline{(a,b)}$ for all $a,b\in V$ 
if $g$ is antilinear. 
\item[{\rm (iii)}]
$g(\nu)=\nu$. 
\end{itemize}
\end{corollary}
\begin{proof}
(i) $\Rightarrow$ (iii). If $g$ is grading preserving, then $g(\nu)$ 
is a conformal vector such that $g(\nu)_{(1)}=\nu_{(1)}$ and 
$(g(\cdot), g(\cdot))$ (or $\overline{(g(\cdot),g(\cdot))}$,
in the antilinear case) is a non-degenerate invariant bilinear form 
for the corresponding VOA structure. Hence by Prop. 
\ref{uniquestress} $g(\nu)=\nu.$

(iii) $\Rightarrow$ (ii).  If $g(\nu)=\nu$ then 
$(g(\cdot),g(\cdot))$ (or $\overline{(g(\cdot),g(\cdot))}$, in 
the antilinear case) is an invariant bilinear form on $V$ and hence by 
Prop. \ref{factsIB-forms}  
it must coincide with $(\cdot,\!\cdot)$ because $g(\Omega)=\Omega$ 
and $(\Omega,\Omega)\neq 0$ by non-degeneracy.

(ii) $\Rightarrow$ (i). Every vertex algebra 
automorphism or antilinear automorphism $g$ commutes with $T=L_{-1}$. If 
$g$ preserves $(\cdot,\!\cdot)$ then also its inverse does so, and as $g^{-1}$ also 
commutes with $T = L_{-1}$, we have
\begin{equation}
(a,L_1g(b))=(Ta,g(b))=(Tg^{-1}(a),b)=(g^{-1}(a),L_1b)=
(a,g(L_1b))
\end{equation}
for all $a,b\in V$. Thus by the non-degeneracy of $(\cdot,\!\cdot)$
it follows that $L_1g(b)=g(L_1b)$; i.e.\! that $g$
commutes with $L_1$. But then it also commutes with $L_0 =
\frac{1}{2}[L_1,L_{-1}]$ and hence it is grading preserving.
\end{proof}

In the following we shall say that a vertex algebra automorphism, resp. 
antilinear automorphism, 
$g$ of a vertex operator algebra $V$ with conformal vector $\nu$ is a 
{\bf VOA automorphism}, resp. {\bf VOA antilinear automorphism}, if 
$g(\nu) =\nu$ and we shall denote by ${\rm Aut}(V)$ the group 
of VOA automorphisms of $V$. 

The group ${\rm Aut}(V)$ has a natural 
topology making it into a metrizable topological group. 
First note that the group $\prod_{n\in \ZZ}{\rm GL}(V_n)$ 
of grading preserving vector space automorphisms of $V$ is the direct 
product of the finite-dimensional Lie groups ${\rm GL}(V_n)$, $n\in \ZZ$. 
Hence $\prod_{n\in \ZZ}{\rm GL}(V_n)$ with the 
product topology is a metrizable topological group. Now, ${\rm Aut}(V)$ is 
a subgroup of $\prod_{n\in \ZZ}{\rm GL}(V_n)$ and hence it becomes a 
topological group when endowed with the relative topology. 

A sequence $g_n \in \prod_{n\in \ZZ}{\rm GL}(V_n)$ converges to 
$g \in \prod_{n\in \ZZ}{\rm GL}(V_n)$ if and only if for all 
$a\in V$ and all $b'\in V'$  the sequence of complex numbers
$\langle b',g_na \rangle $ converges to $\langle b', ga \rangle $. 
Now let $g_n$ be a sequence in ${\rm Aut}(V)$ converging to an element 
$g$ of $\prod_{n\in \ZZ}{\rm GL}(V_n)$. Then for all $a,b\in V$, 
$c'\in V'$ and $m\in \ZZ$ we have 
$\langle c', g(a_{(m)}b)\rangle$ $= \lim_{n\rightarrow \infty}
\langle c', g_n(a_{(m)}b)\rangle$ $= \lim_{n\rightarrow \infty}
\langle c', g_n(a)_{(m)}g_n(b)\rangle$ 
$=\langle c', g(a)_{(m)}g(b)\rangle$. 
It follows that $g(a_{(m)}b) = g(a)_{(m)}g(b)$ 
and  $g\in {\rm Aut}(V)$.  Thus ${\rm Aut}(V)$ is a 
closed subgroup of $\prod_{n\in \ZZ}{\rm GL}(V_n)$. 

\section{Unitary vertex operator algebras}
\label{SectionUnitaryVOA}
In this section we define and study the notion of unitary VOA. For closely related material cf. \cite{DL2014}.
\subsection{Definition of unitarity}
Now let $V$ be a VOA with conformal vector $\nu$, and let $(\cdot|\cdot)$ 
be a scalar product on $V$, namely a positive-definite sesquilinear form 
(linear in the second variable). 
We say that the scalar product is normalized if $(\Omega|\Omega)=1$ 
and we say that $(\cdot|\cdot)$ is invariant if there is 
a VOA antilinear automorphism $\theta$ of $V$ 
such that  $(\theta \cdot | \cdot)$ is an invariant bilinear form on $V$. 
In this case we will say that $\theta$ is a {\bf PCT operator}
associated with $(\cdot|\cdot)$. 

Now let $\nu$ be the conformal vector of $V$ and let 
$Y(\nu,z)=\sum_{n\in \ZZ}L_nz^{-n-2}$ be the corresponding 
energy-momentum field. Moreover, let $(\cdot|\cdot)$ be a normalized 
invariant scalar product on $V$ with an associated PCT operator $\theta$.   
Since $\theta(\nu)=\nu$, $\theta$ commutes with all $L_n$, $n\in \ZZ$. 
It follows from 
Eq. (\ref{invbil2}) that, for all $a,b, c \in V$ and $n\in \ZZ$, we have
\begin{equation}
\label{INVadjoint}
(a_nb|c)=(\theta (\theta^{-1}a)_n\theta^{-1}b|c)=
(b|(\theta^{-1}e^{L_1}(-1)^{L_0}a)_{-n}c).
\end{equation} 
In particular if $a$ is quasi-primary we have
\begin{equation}
\label{QPadoint}
(a_nb|c)= (-1)^{d_a}(b|(\theta^{-1}a)_{-n}c),
\end{equation}  
for all $b, c \in V$ and $n\in \ZZ$. In particular
\begin{equation} 
\label{uniVir}
(L_na|b)=(a|L_{-n}b),
\end{equation}
for all $a, b \in V$ and $n\in \ZZ$,
i.e. the corresponding representations of the Virasoro algebra 
and of its M\"{o}bius subalgebra $\CC\{L_{-1}, L_0, L_1\}$ 
(isomorphic to $\mathfrak{sl}_2(\CC)$) are unitary and hence completely 
reducible. In particular we have $V_{n}=0$ for $n<0$. 

\begin{proposition}
\label{thetaProperties}
Let $(\cdot|\cdot)$ be a normalized invariant scalar product on the 
vertex operator algebra $V$. Then there exists a unique PCT operator 
$\theta$ associated with $(\cdot|\cdot)$. Moreover, $\theta$ is an 
involution i.e. $\theta^2=1_V$ and it is 
is antiunitary 
i.e. $(\theta a|\theta b) =(b|a)$ for all $a,b \in V$.
\end{proposition}
\begin{proof} Assume that $\tilde{\theta}$ is another PCT operator 
associated with $(\cdot|\cdot)$. Then it follows from 
Eq. (\ref{INVadjoint}) that 
$(\theta^{-1}e^{L_1}(-1)^{L_0}a)_{n} = 
(\tilde{\theta}^{-1}e^{L_1}(-1)^{L_0}a)_{n}$ for all $a\in V$ and 
all $n \in \ZZ$. Hence (see Subsect. \ref{subsect.ConformalVA})
$\theta^{-1}e^{L_1}(-1)^{L_0}a = \tilde{\theta}^{-1}e^{L_1}(-1)^{L_0}a$
for all $a\in V$. Since $e^{L_1}(-1)^{L_0}$ is surjective, it follows 
that $\theta=\tilde{\theta}$. Now, from Eq. (\ref{INVadjoint}) 
it also follows that $a= (e^{L_1}(-1)^{L_0})^2\theta^{-2}a$ for 
all $a\in V$. It follows from Eq. (\ref{L_0cova_n}) that 
$(-1)^{L_0}e^{L_1}(-1)^{L_0}= e^{-L_1}$ and hence 
$(e^{L_1}(-1)^{L_0})^2=1$. Thus $\theta^2=1$.  
Finally, given $a, b \in V$, the symmetry of the invariant bilinear 
form $(\theta \cdot|\cdot)$ implies that 
$(\theta a| \theta b)=(\theta^2 b|a)=(b|a)$ and hence $\theta$ is 
antiunitary. 
\end{proof}

Note that if $(\cdot|\cdot)$ is an invariant normalized scalar product on 
the VOA $V$ and $\theta$ is the corresponding PCT operator then the 
invariant bilinear 
form $(\theta \cdot|\cdot)$ is obviously normalized and non-degenerate. 
Hence as a $V$-module $V$ is equivalent to the contragredient module 
$V'$. 
Note also that  as a consequence of Prop. \ref{thetaProperties}, 
$\theta$ is determined by $(\cdot|\cdot)$. Conversely, if $V$ is simple, we have 
$(\cdot|\cdot) = (\theta \cdot, \cdot)$ where $(\cdot,\cdot)$ is the 
unique normalized invariant bilinear form on $V$ and hence
$(\cdot|\cdot)$ is determined by $\theta$. 

\begin{definition} {\rm A {\bf unitary vertex operator algebra} is 
a pair $(V, (\cdot |\cdot))$ where $V$ is a vertex operator 
algebra and $(\cdot|\cdot)$ is a normalized invariant scalar product
on $V$.}
\end{definition}   
We have the following:

\begin{proposition}
\label{simpleunitary}
Let $(V,(\cdot|\cdot))$ be a unitary $VOA$.
Then $V$ is simple if and only if $V_0=\CC \Omega$. In particular every 
simple unitary VOA is of CFT type.
\end{proposition}
\begin{proof} Let $a\in V_0$. Then
$$(L_{-1}a|L_{-1}a)=(a|L_1L_{-1}a)=2(a|L_0a)=0.$$
Hence $L_{-1}a=0$ and by \cite[Remark 4.4b]{Kac}, $Y(a,z)=a_{(-1)}$. Thus,
by locality $a_{(-1)}$ commutes with every $b_n$, $n\in \ZZ$,
$b \in V$. Accordingly if $V$ is simple $a_{(-1)}$ is a multiple
of the identity by Schur's lemma  because every VOA has countable dimension, 
see e.g. \cite[Lemma 2.1.3]{CG}. Thus $a \in \CC\Omega$. 
Conversely if $V_0=\CC\Omega$ then $V$ is simple 
by Prop. \ref{factsIB-forms} $(iv)$. 
\end{proof}

\begin{remark}
\label{RemarkV_R}
{\rm Let $(V, (\cdot |\cdot))$ be a unitary VOA unitary 
with PCT operator $\theta$. Then the real subspace 
\begin{equation}
\label{realform}
V_\RR = \{a\in V : \theta a = a \}
\end{equation}
contains the conformal vector $\nu$ and the vacuum vector $\Omega$ 
and inherits from $V$ the structure of a real vertex operator algebra. 
Moreover, $V= V_\RR +iV_\RR$ and $V_\RR \cap iV_\RR =\{0 \}$, 
i.e. $V$ is the complexification of $V_\RR$. Such a real subspace  
is called a real form \cite{matsuo03}. The restriction of 
$(\cdot|\cdot)$ to $V_\RR$ is positive definite real valued invariant $\RR$-bilinear form 
on $V_\RR$ and hence $(\cdot|\cdot)$ is a positive definite  invariant 
Hermitian form on $V$ in the sense of \cite[Sect.1.2]{matsuo03}. 
Conversely let $\tilde{V}$ be vertex operator algebra over $\RR$ with 
a positive definite normalized real valued invariant $\RR$-bilinear form $(\cdot, \cdot)$ (see \cite{Miyamoto} for an interesting class of examples)
and let $V$ be the complexification of $\tilde{V}$. Then,
$(\cdot, \cdot)$ extends uniquely to an invariant scalar product 
$(\cdot | \cdot)$ on the complex vertex operator algebra. Moreover, 
$\tilde{V}$ coincide with the corresponding real form 
$V_\RR$ defined in Eq. (\ref{realform}). } 
\end{remark}

\begin{remark} 
{\rm It is straightforward to show that if $V_1$ and $V_2$ are unitary vertex operator algebras then 
then also $V_1 \otimes V_2$ is unitary.}
\end{remark}

We conclude this section with some examples of unitary VOAs.

\begin{example}
\label{ExampleUnitaryVir}
{\rm The vertex algebra $L(c,0)$ associated with the unitary representation of the Virasoro algebra with central charge
$c$ and lowest conformal energy $0$ is a simple unitary VOA. We call it the unitary Virasoro VOA with central charge $c$. The possible value of $c$ are restricted by unitarity, see Subsect. \ref{SubsectPositiveEnergy}.}
\end{example}

\begin{example} 
\label{ExampleUnitaryg_k}
{\rm Let ${\g}$ be a simple complex Lie algebra and let $V_{\g_k}$ be the conformal vertex algebra associated with the unitary representation of the affine Lie algebra $\hat{\g}$ corresponding to $\g$, having level $k$ and lowest conformal energy $0$. 
Then $V_{\g_k}$  is the simple unitary VOA corresponding to the level $k$ chiral current algebra CFT model associated $\g$. 
}
\end{example} 

\begin{example} 
\label{ExampleUnitaryV_H}
{\rm Let $V_H$ be the Heisenberg vertex operator algebra associated with the unitary representation of the (rank one) Heisenberg Lie algebra with lowest conformal energy $0$. Then $V_H$ is a simple unitary VOA corresponding to the ${\rm U}(1)$ chiral current algebra CFT model (free boson).}
\end{example}

\begin{example} 
\label{ExampleUnitaryLattice}
{\rm Let $L$ be an even positive definite lattice. Then the corresponding lattice VOA $V_L$ is unitary, cf. 
\cite[Prop.2.7]{Miyamoto} and \cite[Thm.4.12]{DL2014}. 
}
\end{example}

\begin{example} 
\label{ExampleUnitaryMoonshine}
{\rm Let $V^\natural$ be the moonshine VOA constructed by Frenkel, Lepowsky and Meurman \cite{FLM}, see also 
\cite{Miyamoto}. Then $V^\natural$ is a simple unitary VOA. ${\rm Aut}(V)$ is the Monster group $\mathbb{M}$, the largest among the 26 sporadic finite simple groups, cf. \cite{FLM}.}
\end{example}

\subsection{An equivalent approach to unitarity}
\label{SubsecEquivUnitary}
The definition of unitarity given in the previous section appears 
to be very natural from the point of view of vertex operator algebras 
theory. In this subsection we will show that it is natural also from 
the point of view of quantum field theory (QFT). To simplify the 
exposition we shall consider in detail only the case of vertex operator algebras 
$V$ with $V_{0}=\CC\Omega$.

From the QFT point of view, in agreement with 
Wightman axioms \cite{StrWight} the basic requirements for unitarity 
should reflect the following properties: 

\noindent{\em (1) The spacetime symmetries act unitarily.}

\noindent{\em  (2) The adjoints of  local fields are local.}

To give a precise formulation of these requirements we need some 
preliminaries.  
Let $V$ be a vertex operator algebra with energy-momentum field 
$Y(\nu,z)=\sum_{n\in \ZZ}L_nz^{-n-2}$ and let $(\cdot|\cdot)$
be a normalized scalar product on $V$. We say that the pair  
$(V,(\cdot|\cdot))$ has {\bf unitary M\"{obius} symmetry} if 
for all $a,b \in V$
\begin{equation} 
(L_na|b)=(a|L_{-n}b), \quad n=-1,0, 1.
\end{equation}
Now let $A \in {\rm End}(V)$. We say that $A$ have an adjoint on $V$ 
(with respect to $(\cdot|\cdot)$) if there exists $A^+ \in {\rm End}(V)$ 
such that 
\begin{equation}
(a|Ab)=(A^+a|b),
\end{equation}
for all $a,b \in V$. Clearly if $A^+$ exists then it is unique and we say 
that $A^+$ is the adjoint of $A$ on $V$.  
If $\H_V$ denotes the Hilbert space completion of $(V,(\cdot|\cdot))$
then each $A \in {\rm End}(V)$ may be considered as a densely defined 
operator on $\H_V$. Then $A^+$ exists if and only if the domain of  
Hilbert space adjoint $A^*$ of $A$ contains $V$ and in this case 
we have $A^+ \subset A^*$, i.e. $A^+=A^*\restriction_V$.  

It is easy to see that the set of elements in ${\rm End V}$ having an 
adjoint on $V$ is a subalgebra of ${\rm End V}$  containing the identity 
$1_V$ and closed under the operation $A\mapsto A^+$. In fact if 
$A,B \in {\rm End V}$ admit an adjoint on $V$ then, for all
$\alpha, \beta \in \CC$,  
$(\alpha A + \beta B)^+=\overline{\alpha} A^+ + \overline{\beta} B^+ $, 
$(AB)^+=B^+A^+$ and $A^{++} \equiv (A^+)^+=A$.  

\begin{lemma}
\label{adjointLemma} 
Let $(V,(\cdot|\cdot))$ have  unitary M\"{obius} 
symmetry. Then, for any $a\in V$ and $n\in\ZZ$, the adjoint 
$a_n^+$ of $a_n$ on $V$ exists. Moreover, for any $b\in V$ there exists an
$N\in\ZZ_{\geq 0}$ such that if $n\geq N$ then $a^+_{-n} b =0$ 
\end{lemma}

\begin{proof}
From unitary M\"{obius} symmetry it follows that the finite-dimensional  
subspaces 
$V_n={\rm Ker}(L_0-n1_V)$ of $V$ are pairwise orthogonal. 
Since $a_n(V_k)\subset V_{k-n}$, we may view $a_n\restriction_{V_k}$ as an
operator between two finite-dimensional scalar product spaces, and
so it has a well-defined adjoint 
$(a_n\restriction_{V_k})^* \in {\rm Hom}(V_{k-n},V_k)$. It is easy to check that
\begin{equation}
a_n^+ \equiv {\bigoplus}_{k\in\ZZ}(a_n\restriction_{V_k})^*
\end{equation}
is indeed the adjoint of $a_n$ on $V$ (and so it exists). From its
actual form we also see that $a^+_{-n}(V_k)\subset V_{k-n}$ which
shows that indeed for any $b\in V$ there exists an $N\in\ZZ$ such
that if $n\geq N$ then $a^+_{-n} b =0$. 
\end{proof}

Now let $(V,(\cdot|\cdot))$ have  unitary M\"{o}bius symmetry. 
From the previous lemma it follows that for every $a\in V$ the formal 
series 
\begin{equation}
\label{EqAdjointVO}
Y(a,z)^+ \equiv \sum_{n\in \ZZ}a_{(n)}^+ z^{n+1}= 
\sum_{n\in \ZZ}a_{(-n-2)}^+z^{-n-1} 
\end{equation}  
is well defined and gives a field on $V$ i.e., for every $b \in V$, 
$a_{(-n-2)}^+b=0$ if $n$ is sufficiently large. 

For $a \in V$ we say that the vertex operator $Y(a,z)$ has a 
{\bf local adjoint} if for every $b\in V$ the fields $Y(a,z)^+$, $Y(b,z)$ 
are mutually local i.e. 
\begin{equation} 
(z-w)^N \left[Y(a,z)^+, Y(b,w) \right]=0,
\end{equation}
for sufficiently large $N \in \ZZ_{\geq 0}$ and we denote by 
$V^{(\cdot |\cdot)}$ the subset of $V$ whose elements are the vectors
$a\in V$ such that $Y(a,z)$ has a local adjoint. 

\begin{remark} {\rm The adjoint vertex operator $Y(a,z)^+$ should not be confused with the adjoint 
vertex operator $Y'(a,z)$ in the definition of the contragredient module $V'$ in Subsect. 
\ref{subsectionVOA&IBforms}. }

\end{remark}

\begin{lemma}
\label{adjointcommutatorLemma} 
For $a, b\in V$, 
$Y(a,z)^+$ and $Y(b,z)$ are mutually local if and only if 
$Y(a,z)$ and $Y(b,z)^+$ are mutually local. 
\end{lemma}
\begin{proof} Let $N\in \ZZ_{\geq 0}$. Then 
\begin{eqnarray*}
(z-w)^N \left[ Y(a,z)^+, Y(b,w) \right] = 0 
& \Leftrightarrow & \\
\sum_{j=0}^N  \sum_{(m,n)\in \ZZ^2} \left[a_{(m)}^+, 
b_{(n)}\right]
\left(\begin{matrix} N \\ j \end{matrix}\right)
(-1)^j w^j z^{N-j} z^{m+1} w^{-n-1} = 0 
& \Leftrightarrow & \\
\forall m, n \in\ZZ, \sum_{j=0}^N \left[a_{(m+j-N)}^+, 
b_{(n+j)}\right]
\left(\begin{matrix} N \\ j \end{matrix}\right)
(-1)^j = 0  
& \Leftrightarrow &\\
\forall m, n \in \ZZ, 
\left(\sum_{j=0}^N\left[a_{(m+j-N)}^+,b_{(n+j)}\right]
\left(\begin{matrix} N \\ j \end{matrix}\right)
(-1)^j \right)^+ = 0 
& \Leftrightarrow &
\end{eqnarray*}
\begin{eqnarray*}
\forall m,n\in \ZZ,
\sum_{j=0}^N\left[a_{(m +j-N)},b_{(n+j)}^+\right]
\left(\begin{matrix} N \\ j \end{matrix}\right)
(-1)^{j} = 0 
& \Leftrightarrow & \\
\forall m,n\in \ZZ,
\sum_{j=0}^N\left[a_{(m +j)},b_{(n+j-N)}^+\right]
\left(\begin{matrix} N \\ j \end{matrix}\right)
(-1)^{j}= 0
& \Leftrightarrow & \\
(z-w)^N \left[ Y(b,z)^+, Y(a,w) \right] = 0.
\end{eqnarray*}
\end{proof}

\begin{proposition}
\label{VscalarAlgebra}
 $V^{(\cdot|\cdot)}$ is a vertex subalgebra of $V$. 
\end{proposition}
\begin{proof} It is clear that  $V^{(\cdot|\cdot)}$ is a subspace of 
$V$ containing $\Omega$. Now let $a,b \in V^{(\cdot|\cdot)}$, 
$c \in V$ and $n \in \ZZ$. By Lemma \ref{adjointcommutatorLemma}
$Y(a,z)$, $Y(b,z)$ and $Y(c,z)^+$ are pairwise mutually local fields 
on $V$. Hence by \cite[Prop. 4.4.]{Kac} and Dong's Lemma 
\cite[Lemma 3.2.]{Kac} $Y(a_{(n)}b,z)$ and $Y(c,z)^+$ 
are mutually local. Since $c\in V$ was arbitrary 
Lemma \ref{adjointcommutatorLemma} then shows that $Y(a_{(n)}b,z)$ 
has a local adjoint, i.e. $a_{(n)}b\in V^{(\cdot|\cdot)}$.
\end{proof}

\begin{lemma}
\label{primaryadjointLemma} 
Let $a\in V^{(\cdot|\cdot)}$ be a 
quasi-primary vector. Then there is a quasi-primary vector 
$\overline{a}\in V^{(\cdot|\cdot)}$ with $d_{\overline{a}}=d_a$ and such 
that $z^{-2d_a}Y(a,z)^+ = Y(\overline{a},z)$, equivalently 
$a_n^+=\overline{a}_{-n}$ for all $n \in \ZZ$.
\end{lemma} 
\begin{proof}
The field  $z^{-2d_a}Y(a,z)^+$ coincides with 
$\sum_{n\in \ZZ}a_{-n}^+z^{-n-d_a}$. Since
\begin{eqnarray}\nonumber
\,\![L_{-1},a^+_{-n}] &=& -[L_{-1}^+,a_{-n}]^+ = -[L_1,a_{-n}]^+
\\ &=& -((d_a-1+n)a_{-n+1})^+ = -(d_a-1+n)a_{-n+1}^+,
\end{eqnarray}
it is translation covariant and hence, cf. \cite[Remark 1.3.]{Kac} 
$z^{-2d_a}Y(a,z)^+\Omega=e^{zL_{-1}}a_{d_a}^+\Omega$. By assumption 
$Y(a,z)^+$ is mutually local with all fields $Y(b,z)$, 
$b\in V$. Hence $z^{-2d_a}Y(a,z)^+$ is also mutually local with all fields 
$Y(b,z)$, $b\in V$. From the uniqueness theorem \cite[Thm.4.4.]{Kac} 
it then follows that $z^{-2d_a}Y(a,z)^+ = Y(\overline{a},z)$ 
where $\overline{a}=a_{d_a}^+\Omega$. 
Since $Y(\overline{a},z)^+=z^{2d_a}Y(a,z)$ we have 
$\overline{a} \in V^{(\cdot|\cdot)}$. 
Moreover, $L_0a_{d_a}^+\Omega = -[L_0,a_{d_a}]^+\Omega =d_a a_{d_a}^+\Omega$ 
and $L_1a_{d_a}^+\Omega = -[L_{-1},a_{d_a}]^+\Omega =
(-2d_a+1) a_{d_a-1}^+\Omega=0$ and hence $\overline{a}$ is quasi-primary 
of dimension $d_a$. 
\end{proof}

The following theorem is a vertex algebra formulation of the PCT 
theorem \cite{StrWight}. 

\begin{theorem} 
\label{unitarityTheorem}
Let $V$ be a vertex operator algebra with 
a normalized scalar product $(\cdot|\cdot)$. Assume that 
that $V_{0}=\CC \Omega$. Then the following are equivalent
\begin{itemize}

\item[(i)] $(V,(\cdot|\cdot))$ is a unitary VOA.

\item[(ii)] $(V,(\cdot|\cdot))$ has unitary M\"{o}bius symmetry and 
$V^{(\cdot|\cdot)}=V$, i.e. every vertex operator has a local adjoint.
\end{itemize}
\end{theorem}   
\begin{proof}
Let $Y(\nu,z)=\sum_{n\in \ZZ}L_nz^{-n-2}$ be the energy-momentum field of 
$V$. 
That $(i) \Rightarrow (ii)$ is rather trivial. Indeed, suppose that
$(V,(\cdot | \cdot))$ is a unitary VOA
and let $\theta$ be the corresponding PCT operator. From 
Eq. (\ref{uniVir}) it follows that the pair 
has unitary M\"{o}bius symmetry.
If $a\in V$ is a homogeneous vector, then by Eq. (\ref{INVadjoint})
and the properties of $\theta$ we have 
$a_{-n}^+ = (-1)^{d_a}\sum_{l \in \ZZ_{\geq 0}}\frac{1}{l!}(L^l_1\theta a)_{n}$
for all $n\in \ZZ$. Hence 
\begin{equation}
\label{Y(a,z)^+}
Y(a,z)^+=(-1)^{d_a}\sum_{l \in \ZZ_{\geq 0}}\frac{1}{l!}
Y(L^l_1\theta a,z)z^{2d_a-l}
\end{equation}
is mutually local with all fields $Y(b,z)$, $b\in V$ and since the 
homogeneous vector $a$ was arbitrary it follows that 
$V^{(\cdot|\cdot)}=V$.  

Let us now prove $(ii) \Rightarrow (i)$. Assume that  
$(V,(\cdot|\cdot))$ has unitary M\"{o}bius symmetry and
that $V^{(\cdot|\cdot)}=V$. We first show that $V$ is simple. 

Since $V_0 =\CC\Omega$, by Remark \ref{CFTtypeRemark}, $V$
is of CFT type. Let $\mathscr{J}\subset V$ be a non-zero ideal. 
Since $L_0\mathscr{J}\subset \mathscr{J}$ we have 
$$\mathscr{J}= \bigoplus_{n \geq m} \mathscr{J}\cap V_{n}$$ 
for some $m\in \ZZ_{\geq 0}$ such that $\mathscr{J}\cap V_{m}\neq \{0\}$.
Let $a$ be a non-zero vector in $\mathscr{J}\cap V_{m}$. Then $a$ is 
quasi-primary of 
dimension $m$ and by Lemma \ref{primaryadjointLemma} there exists 
$\overline{a} \in V_{m}$
such that $\overline{a}_{m}=a_{-m}^+$. Then 
$\overline{a}_{m}a = \overline{a}_{m}a_{-m}\Omega$ is a non-zero 
vector in $\mathscr{J}\cap V_0$. Accordingly $\Omega \in \mathscr{J}$ and $\mathscr{J}=V$. 
Hence $V$ is simple. 

Now let $a \in V_1$. Since $V_0=\CC\Omega$ and $L_1a\in V_0$ we 
have $L_{-1}L_1a=0$ and from unitary M\"{o}bius symmetry 
it follows that $(L_1a|L_1a)=0$ and hence $L_1a=0$. 
Accordingly $L_1V_1=\{0\}$ and by Prop. \ref{factsIB-forms}
it follows that there is a unique normalized invariant 
bilinear form $(\cdot, \cdot)$ on $V$ which is non-degenerate 
being $V$ simple.

The finite-dimensional subspaces $V_n$, $n\in\ZZ_{\geq 0}$, satisfy
$(V_n|V_m)=0$ and
$(V_n,V_m)= 0$ for $n\neq m$. Thus there exists a unique
$\theta:V\to V$ antilinear, grading preserving map such that
$(\cdot,\!\cdot)=(\theta\cdot |\cdot)$. 

By Corollary \ref{autcorollary} 
and Prop. \ref{thetaProperties}
all we have to show is that the 
above introduced conjugate linear
map $\theta$ is actually a vertex algebra antilinear automorphism. 

First of all from the non-degeneracy of $(\cdot,\cdot)$ it follows that
$\theta$ is injective and since $\theta V_n \subset V_n$ and 
$V_n$ is finite-dimensional for all $n\in \ZZ_{\geq 0}$ then $\theta$ is invertible. 
Note also that by unitary M\"{o}bius symmetry it follows that 
$\theta$ commutes with $L_n$, $n=-1,0,1$. 

Now let $a\in V= V^{(\cdot|\cdot)}$ be a quasi-primary vector. 
By Lemma \ref{primaryadjointLemma} there exists a quasi-primary vector 
$\overline{a}\in V_{d_a}$ such that $a_n^+=\overline{a}_{-n}$, 
$n\in \ZZ$. We have 
$$
( \theta a_n b| c) = (a_n b, c)
= (-1)^{d_a}(b,a_{-n}c) = (-1)^{d_a}(\theta b|a_{-n}c)
= (-1)^{d_a}(\overline{a}_n \theta b| c),
$$
for all $b,c \in V$, showing that $\theta a_n = (-1)^{d_a} \overline{a}_n 
\theta$. Since $V_0=\CC\Omega$ and $(\Omega|\Omega) =
(\Omega,\Omega) = 1$, we have that
$\theta\Omega=\Omega$. Therefore,
$\theta a=\theta a_{-d_a}\Omega =
(-1)^{d_a}\overline{a}_{-d_a}\theta\Omega= (-1)^{d_a}\overline{a}$.
Hence, for every quasi-primary vector $a$ we have  
$\theta a_{(n)}\theta^{-1} = (\theta a)_{(n)}$ for all $n\in \ZZ$.
Since $\theta$ commutes with $L_{-1}$ and since by 
unitary M\"{o}bius symmetry the vectors of the form 
$L_{-1}^k a$ with $k\in \ZZ_{\geq 0}$ and $a$ quasi-primary span $V$, then, recalling Remark \ref{RemarkTcommutator}, it follows 
that $\theta b_{(n)}\theta^{-1} =(\theta b)_{(n)}$ for all $b \in V$,  
$n\in \ZZ$ and hence $\theta$ is a vertex algebra antilinear 
automorphism. 
\end{proof}

Now let $V$ be a vertex operator algebra with a normalized scalar 
product $(\cdot|\cdot)$ and let 
$a \in V$ be a quasi-primary vector. Then we shall call the 
corresponding quasi-primary field $Y(a,z)$ {\bf Hermitian}
(with respect to $(\cdot|\cdot)$) if $(a_nb|c)=(b|a_{-n}c)$ 
for all $b,c\in V$ and all $n\in \ZZ$. This means that 
for all $n\in \ZZ$ the adjoint $a_n^+$ of $a_n$ on $V$ exists 
and coincides with $a_{-n}$. The following consequence
of Thm. \ref{unitarityTheorem} gives a useful 
characterization of simple unitary vertex operator algebras. 

\begin{proposition}
\label{PropUnitaryGenHermitian}
Let $V$ be a vertex operator algebra with 
conformal vector $\nu$ and let $(\cdot |\cdot)$ be a normalized 
scalar product on $V$. Assume that $V_0=\CC \Omega$. Then  the
following are equivalent
\begin{itemize}
\item[$(i)$] $(V,(\cdot|\cdot))$ is a unitary vertex operator 
algebra.
\item[$(ii)$] $Y(\nu,z)$ is Hermitian and $V$ is generated by 
a family of Hermitian quasi-primary fields.
\end{itemize}
\end{proposition}
\begin{proof}
$(i) \Rightarrow (ii)$. If $(V,(\cdot|\cdot))$ is a unitary vertex 
operator algebra and $\theta$ is the corresponding PCT operator 
then $Y(\nu,z)$ is Hermitian by Eq. (\ref{uniVir}). 
Moreover, if $a$ is a quasi-primary vector then, by Eq. (\ref{QPadoint})
$a_n^+ = (-1)^d_a (\theta a)_{-n}$ for all $n$. Accordingly if 
$b = \frac{1}{2}(a + (-1)^{d_a} \theta a)$ and 
$c=\frac{-i}{2}(a -(-1)^{d_a}\theta a)$ then $Y(b,z)$ are 
and $Y(c,z)$ are Hermitian quasi-primary fields such that 
$Y(a,z)=Y(b,z)+iY(c,z)$. Since $V$ is generated by its 
quasi-primary fields then it follows that it is also generated by 
its Hermitian quasi-primary fields. 

$(ii) \Rightarrow (i)$ If $Y(\nu ,z)$ is Hermitian then 
the normalized scalar product $(\cdot |\cdot)$ has clearly 
unitary M\"{o}bius symmetry. Now let $\mathscr{F} \subset V$ be the 
generating family of quasi-primary vectors corresponding to 
a generating family of Hermitian quasi-primary fields. 
Then for $a \in \mathscr{F}$ the hermiticity condition gives 
$Y(a,z)^+=z^{2d_a}Y(a,z)$ and hence $Y(a,z)$ has a local adjoint
i.e. $a \in V^{(\cdot |\cdot)}$. Hence $\mathscr{F} \subset V^{(\cdot |\cdot)}$
and since $\mathscr{F}$ generates $V$ and $V^{(\cdot |\cdot)}$ is vertex subalgebra 
of $V$ by Prop. \ref{VscalarAlgebra} it follows that 
$V=V^{(\cdot |\cdot)}$ and hence, by Thm. \ref{unitarityTheorem}, 
$(V,(\cdot |\cdot))$ is a unitary vertex operator algebra.
\end{proof}

\subsection{Unitary automorphisms and essential uniqueness of the unitary 
structure} 
Now, let $(V, (\cdot |\cdot))$ be a unitary vertex operator algebra. 
We denote by ${\rm Aut}_{(\cdot|\cdot)} (V)$ the subgroup of the
elements of ${\rm Aut}(V)$ which are unitary with respect to 
$(\cdot|\cdot)$. In other words an element $g$ of 
${\rm Aut}_{(\cdot|\cdot)}(V)$  is a VOA automorphism of $V$ such that 
$(g a|g b)=(a|b)$ for all $a,b \in V$. We will say that 
${\rm Aut}_{(\cdot|\cdot)}(V)$ is the automorphism group of 
the unitary VOA $(V, (\cdot |\cdot))$. 

\begin{remark} {\rm It follows from Prop. \ref{factsIB-forms} $(iii)$ that if $V$ is simple and $g \in {\rm Aut}(V)$ then $g \in {\rm Aut}_{(\cdot|\cdot)}(V)$
if and only if $g^{-1}\theta g = \theta$. Accordingly, if $V_\RR =\{a\in V:\theta a =a\}$ is the real form as in 
Remark \ref{RemarkV_R}, then $g \in {\rm Aut}_{(\cdot|\cdot)}(V)$ if and only if $g$ restricts to a VOA automorphism of the real vertex operator algebra $V_\RR$. Conversely, every VOA automorphism of $V_\RR$ give rise to a VOA automorphism of $V$ and hence 
we have the identification ${\rm Aut}_{(\cdot|\cdot)}(V) = {\rm Aut}(V_\RR)$. }
\end{remark}

In general ${\rm Aut}_{(\cdot|\cdot)}(V)$ is properly contained 
in ${\rm Aut}(V)$. If $g \in {\rm Aut}(V)$ is VOA automorphism  of $V$ 
which does not belong to ${\rm Aut}_{(\cdot|\cdot)}(V)$ 
then $\{\cdot | \cdot \}=(g \cdot | g \cdot)$ 
is a normalized invariant scalar product on $V$ different from 
$(\cdot|\cdot)$. In fact $\tilde{\theta}=g^{-1}\theta g$ is an
antilinear VOA automorphism of $V$ and  
$\{\tilde{\theta}\cdot | \cdot \} = (\theta g \cdot |g \cdot)$ is
an invariant bilinear form on $V$. 
In the 
case of a simple unitary VOA every 
normalized invariant scalar product arises in this way. In fact 
we have the following

\begin{proposition} 
\label{uniteq}
Let $(V, (\cdot |\cdot))$ be a simple unitary 
VOA with PCT operator $\theta$ and let $\{ \cdot|\cdot \}$ be another 
normalized invariant scalar product on $V$ with corresponding PCT 
operator $\tilde{\theta}$. Then there exists a unique  $h \in {\rm 
Aut}(V)$ such that: 
\begin{itemize}  
\item[(i)] $\{a|b\} = (ha|hb)$ for all $a, b \in V$;
\item[(ii)] $\tilde{\theta}= h^{-1}\theta h$;
\item[(iii)] $\theta h \theta = h^{-1}$;
\item[(iv)] $(a|ha)>0$ for every non-zero $a \in V$. 
\end{itemize}
\end{proposition}
\begin{proof} Let $g \equiv \theta \tilde{\theta}$. Then $g$ is an 
automorphism
of $V$. Moreover, since $\theta$ and $\tilde{\theta}$ are involutions 
we have $ \theta g \theta g  = 1_V$ and hence $\theta g \theta = g^{-1}$. 
From Prop. \ref{factsIB-forms} $(iii)$ we have
$\{ \theta g  \cdot|\cdot \} = \{ \tilde{\theta}  \cdot|\cdot \} = (\theta \cdot|\cdot)$ and hence  
\begin{eqnarray*}
(ga|b) & = & (\theta \theta g a|b) = 
\{ \theta g\theta g a|b \} \\
& = & \{ a|b \},
\end{eqnarray*}
for all $a, b \in V$.
It follows that for every integer $n$ the 
restriction of $g$ to $V_n$ is a strictly positive Hermitian operator 
(with respect to $(\cdot |\cdot)$ ) end hence that $g$ is diagonalizable 
on $V$ with positive eigenvalues. Hence we can take the square root of 
$g$ and define $h \equiv g^{1/2}$. With this $h$ (i) -- (iv) hold and we have 
to show that $h \in {\rm Aut}(V)$. It is clear that $h$ leaves 
$\Omega$ and $\nu$ invariant. Now if $a, b$ are eigenvectors of $g$ 
with eigenvalues $\lambda_a$ and $\lambda_b$ respectively and $n \in \ZZ$ then 
$$g(a_{(n)}b) = g(a)_{(n)}g(b)= \lambda_a\lambda_b a_{(n)}b.$$ 
Hence 
$$h(a_{(n)}b) = (\lambda_a\lambda_b)^{1/2}a_{(n)}b =h(a)_{(n)}h(b),$$
and by linearity, since $g$ is diagonalizable, it follows that 
$h\in {\rm Aut}(V)$.      
The uniqueness of $h$ can be easily shown using $(i)$ and 
$(ii)$. 
\end{proof}

As a consequence of the above proposition a simple VOA 
has, up to unitary isomorphisms, at most one structure 
of unitary VOA. We know from the same Prop. that this 
structure is really unique (i.e. not up unitary isomorphisms) 
iff every automorphism of $V$ is unitary. Using 
\cite[Remark 4.9c]{Kac} one can easily give examples 
of non-unitary automorphism. However there are VOA for which 
the normalized invariant scalar product is unique 
and we will give a characterization of this class using 
the topological properties of ${\rm Aut}(V)$. 

Let $(V, (\cdot |\cdot))$ be a unitary VOA. Then $V$ is a 
normed space with the norm $\| a \| = (a | a)^{1/2}$, $a\in V$. 
Using the norm on $V$ we can topologize ${\rm End}(V)$ 
with the strong operator topology. The corresponding topology on 
${\rm Aut}(V)$ coincides with the topology discussed at the end 
of Subsect. \ref{subsectionVOA&IBforms}. 
Being a subgroup of ${\rm Aut}(V)$, ${\rm Aut}_{(\cdot|\cdot)}(V)$ is 
also a topological group. We have the following 
\begin{lemma} 
\label{compact}
Let $(V, (\cdot |\cdot))$ be a unitary VOA. Then 
${\rm Aut}_{(\cdot|\cdot)}(V)$  is a compact subgroup
of ${\rm Aut}(V)$.
\end{lemma} 
\begin{proof} Let ${\rm U}(V_n)$, $n\in \ZZ$ be the compact subgroup
of ${\rm GL}(V_n)$ whose elements are the unitary endomorphisms with 
respect to the restriction of $(\cdot |\cdot)$ to $V_n$. Then 
$\prod_{n\in \ZZ}{\rm U}(V_n)$ 
is the subgroup of unitary elements the group 
$\prod_{n\in \ZZ}{\rm GL}(V_n)$ of grading 
preserving vector space automorphisms of $V$. 
Since $\prod_{n\in \ZZ}{\rm U}(V_n)$
is compact by Tychonoff's theorem 
$${\rm Aut}_{(\cdot|\cdot)}(V) ={\rm Aut}(V) \cap \prod_{n\in \ZZ} 
{\rm U}(V_n)$$
is also compact because ${\rm Aut}(V)$ is closed in 
$\prod_{n\in \ZZ}{\rm GL}(V_n)$, see the end of Subsect. 
\ref{subsectionVOA&IBforms}. 
\end{proof} 

\begin{theorem} 
\label{uniqscalar}
Let $(V, (\cdot |\cdot))$ be a simple unitary VOA  and let $\theta$ be the corresponding 
PCT operator. Then the following are equivalent:

\begin{itemize}

\item[(i)] $(\cdot|\cdot)$ is the unique normalized invariant 
scalar product on $V$. 

\item[(ii)] ${\rm Aut}_{(\cdot|\cdot)}(V)= {\rm Aut}(V)$. 

\item[(iii)] Every $g \in {\rm Aut}(V)$ commutes with $\theta$.

\item[(iv)] ${\rm Aut}(V)$ is compact.  

\item[(v)] ${\rm Aut}_{(\cdot|\cdot)}(V)$ is totally disconnected.

\end{itemize} 

\end{theorem}     

\begin{proof} The implication $(i)\Rightarrow (ii)$ is clear from the 
comments before Prop. \ref{uniteq}. Now, let $g$ be a
VOA automorphism of $V$. Then 
$g\in {\rm Aut}_{(\cdot| \cdot)}(V)$ iff 
$(g\theta a|gb)$ $=$ $(\theta a| b)$ for all $a, b \in V$. By 
Corollary \ref{autcorollary} we have 
$(\theta ga|gb) = (\theta a |b)$ for all $a, b \in V$. 
Hence $g\in {\rm Aut}_{(\cdot| \cdot)}(V)$ if and only if
$\theta$ and $g$ commute proving $(ii) \Leftrightarrow (iii)$.
The implication $(ii) \Rightarrow (iv)$ follows from Lemma \ref{compact}.

Now let $\{\cdot | \cdot \}$ be a normalized invariant scalar product on 
$V$. By Prop. \ref{uniteq} there is a VOA automorphism $h$ of $V$ 
which is diagonalizable with positive eigenvalues and such that 
$\{a|b\}=(ha|hb)$ for all $a, b \in V$. Moreover, by the same proposition 
$\lambda$ is an eigenvalue of $h$ if only if $\lambda ^{-1}$ is. Hence if 
$h$ is not the trivial automorphism then it has an eigenvalue $\lambda 
>1$ and since $h$ preserves the grading we can find a corresponding 
eigenvector $a \in V_n$ for some positive integer $n$. But then the 
sequence $h^m (a) = \lambda^m v$ is unbounded in $V_n$ and $(iv)$ cannot 
hold proving that $(iv) \rightarrow (i)$. 
Similarly if a nontrivial $h \in {\rm Aut}(V)$ has the properties given 
in Prop. \ref{uniteq} then $\RR \ni t\mapsto h^{it}$ 
is a nontrivial continuous one-parameter group in 
${\rm Aut}_{(\cdot| \cdot)}(V)$ so that $(v)$ cannot hold. Hence  
$(v) \Rightarrow (i)$. 

To conclude the proof of the theorem we now show that 
$(ii) \Rightarrow (v)$. Let us assume that 
${\rm Aut}_{(\cdot| \cdot)}(V)$ is not totally disconnected
and denote by $G$ its component of the identity. Then $G$ is 
a closed connected subgroup of ${\rm Aut}_{(\cdot| \cdot)}(V)$
which is not just the identity subgroup $\{ 1_V\}$. 
For every $N\in \ZZ_{\geq 0}$ we denote $\pi_N$ the projection of 
$\prod_{n\in \ZZ}{\rm GL}(V_n)$ onto 
$\prod_{n=0}^N{\rm GL}(V_n)$. The maps $\pi_N$, $N\in \ZZ_{\geq 0}$ separate
points in $\prod_{n\in \ZZ}{\rm GL}(V_n)$ (recall that $V_n=\{0\}$
if $n<0$). Moreover, if $N_1, N_2$ are non-negative integers 
and $N_2 \geq N_1$ we denote by $\pi_{N_2, N_1}$ the projection of 
$\prod_{n=0}^{N_2}{\rm GL}(V_n)$ onto $\prod_{n=0}^{N_1}{\rm 
GL}(V_n)$ so that $\pi_{N_2, N_1} \circ \pi_{N_2} = \pi_{N_1}$.  
For every $N\in \ZZ_{\geq 0}$ $G_N \equiv \pi_N(G)$ is a compact (and thus 
closed) connected subgroup of the finite-dimensional Lie group 
$\prod_{n=0}^{N_1}{\rm GL}(V_n)$ 
and ,for sufficiently large $N$, $G_N$ is not the identity subgroup. 
Moreover, if $N_1, N_2$ are non-negative integers and $N_2 \geq N_1$  
$\pi_{N_2, N_1}$ restricts to a group homomorphism of 
$G_{N_2}$ onto $G_{N_1}$. As a consequence, for every $N$, we can choose 
a continuous one-parameter group $t\mapsto \phi_N(t)$ in $G_N$ so that $\phi_N(t)$ is 
nontrivial for sufficiently large $N$ and 
$\pi_{N_2, N_1}(\phi_{N_2}(t))=\phi_{N_1}(t)$ for $N_2 \geq N_1$. 
Now it is not hard to show that there is a group homomorphism 
$\RR \ni t\mapsto \phi(t)$ in $G$ such that $\pi_N(\phi(t)) =\phi_N(t)$
for all $N\in \ZZ_{\geq 0}$. Clearly $\RR \ni t\mapsto \phi(t)$ is continuous 
and nontrivial. Now let  $\delta$ be the endomorphism of $V$ defined 
by $\delta(a) = \frac{d}{dt}\phi(t)a|_{t=0}$, $a\in V$. Then $\delta$ 
is a derivation of $V$ ( i.e. $\delta(a_{(n)}b)$ 
$= \delta(a)_{(n)}b +a_{(n)}\delta(b)$ for $a, b \in V$, $n\in \ZZ$)  
commuting with $L_{0}$. Moreover, $\phi(t)=e^{t\delta}$ so that 
$\delta$ is non-zero and the VOA automorphism $e^{\alpha \delta}$
cannot be unitary for every $\alpha \in \CC$.    
\end{proof}

\subsection{Unitary subalgebras}
Let $(V,(\cdot,\!\cdot))$ be a unitary VOA, 
with PCT operator $\theta$ and energy-momentum field 
$Y(\nu, z)= \sum_{n \in \ZZ}L_n z^{-n-2}$ and let 
$W\subset V$ be a vertex subalgebra. 
Recall that the invariant scalar product allows to consider the adjoints 
of vertex operators.
Obviously, if $W$ is a vertex subalgebra of $V$ and $a, b\in W$, 
then the product $a_{(n)}b$ belongs to $W$ for every $n\in \ZZ$, but there 
is no guarantee that $a_{(n)}^+b$ is in $W$, too. This fact motivates
the following definition.

\begin{definition}
{\rm A {\bf unitary subalgebra} $W$ of a unitary vertex 
operator algebra $(V,(\cdot,\!\cdot))$ is a vertex subalgebra of 
$V$ satisfying the following two additional properties:

\begin{itemize}   
\item[$(i)$] $W$ compatible with the grading, namely 
$W = \bigoplus_{n\in\ZZ}(W\cap V_n)$ 
(equivalently $L_0 W \subset W$).

\item[$(ii)$] $a_{(n)}^+b \in W$ for all $a,b\in W$ and $n\in \ZZ$.
\end{itemize}
}
\end{definition}
Note that if $(i)$ is satisfied then $(ii)$ is equivalent to 
$a_{n}^+b \in W$ for all $a,b\in W$ and $n\in \ZZ$.

The following proposition gives a useful characterization of unitary 
subalgebras of the unitary vertex operator algebra $V$. 

\begin{proposition}\label{prop:QW<W}
A vertex subalgebra $W$ of a unitary vertex operator algebra $V$ is unitary if
and only if $\theta W\subset W$ and $L_1W \subset W$.
\end{proposition}
\begin{proof}
Let $W$ be a unitary subalgebra of $V$. 
If $a\in W$ is homogeneous, by Eq. (\ref{INVadjoint}) we have
\begin{equation}
\label{a_n^+}
a_n^+ = (-1)^{d_a} \sum_{j=0}^\infty \frac{1}{j!}
(L_1^j\theta a)_{-n},
\end{equation}
for all $n\in \ZZ$. Hence 
$\sum_{j=0}^\infty \frac{1}{j!}(L_1^j\theta a)_{-n}\Omega \in W$
for all $n\in \ZZ$. For $n=0$ we find that $L_1^{d_a}\theta a \in W$.
For $n=1$ that also $L_1^{d_a-1}\theta a \in W$ and so on. 
Hence, $L_1^j\theta a\in W$ for all $j\in \ZZ_{\geq 0}$. Since the homogeneous 
vector $a\in W$ was arbitrary it follows that $\theta W \subset W$
and $L_{1}W \subset W$. 

Conversely let us assume that $W$ is a vertex subalgebra of 
$V$ such that $\theta W \subset W$ and $L_{1}W \subset W$. 
Since every vertex subalgebra is $L_{-1}$ invariant we also have 
$$L_0W= \frac{1}{2}[L_1,L_{-1}]W \subset W.$$ 
Moreover, Property $(ii)$ in the definition of unitary
subalgebras is an easy consequence of Eq. (\ref{INVadjoint}). 
\end{proof}

Using the definition and the above proposition one can give 
various examples of unitary subalgebras of a unitary vertex operator 
algebra $V$. 
\begin{example}
The vertex subalgebra $L(c,0) \subset V$ generated by the conformal vector $\nu$ of 
the unitary VOA $V$ having central charge $c$ is a unitary subalgebra. We call it the {\bf Virasoro subalgebra} of $V$. 

\end{example}

\begin{example}
{\rm For a closed subgroup 
$G\subset {\rm Aut}_{(\cdot | \cdot)}(V)$, 
the fixed point subalgebra $V^G$ (i.e.\! the set of fixed elements of $V$
under the action of elements of $G$) is unitary. In fact every $g\in G$ 
commutes with $\theta$ and $L_1$ and hence $\theta V^G \subset V^G$ 
and $L_1 V^G \subset V^G$. When $G$ is finite $V^G$ is called 
{\bf orbifold subalgebra}. 
}
\end{example}

\begin{example}
{\rm A vertex subalgebra $W\subset V$ generated by a $\theta$ 
invariant family of quasi-primary
vectors, is clearly invariant for $\theta$ and from 
Eq. (\ref{l_1commutation}) it easily follows that it is also invariant 
for $L_1$. Hence $W$ is unitary.}
\end{example}

\begin{example} 
{\rm Let $W$ be a vertex subalgebra of a unitary vertex operator algebra
$V$. Then 
$W^c = \{b \in V: [Y(a,z),Y(b,w)]=0 \; {\rm for \; all}\; a\in W \}$
is a vertex subalgebra of $V$, and we call it {\bf coset subalgebra} 
(see \cite[Remark 4.6b]{Kac} where $W^c$ is called centralizer). 
By the Borcherds commutator formula Eq. (\ref{commutatorformula})
$b\in V$ belongs to $W^c$ if and only if $a_{(j)}b=0$ for 
all $a \in W$ and $j\in \ZZ_{\geq 0}$, cf. \cite[Cor.4.6. (b)]{Kac}. Now if $W$ is a unitary subalgebra 
and $a, b\in W^c$ then, for all $c \in W$  and all 
$n,m \in \ZZ$, we have 
$[a_{(n)}^+, c_{(m)}]= [c_{(m)}^+, a_{(n)}]^+=0$, as a consequence of Eq. (\ref{a_n^+}). Hence for all $c\in W$, 
$j\in \ZZ_{\geq 0}$, $n \in \ZZ$ we have 
$c_{(j)}a_{(n)}^+b=$ $a_{(n)}^+c_{(j)} b=0$ so that $a_{(n)}^+b \in W^c$. Moreover, if $a\in W$ is homogeneous and 
$b \in W^c$ then
$ a_{(j)}L_0b = a_{j - d_a +1}L_0b =L_0a_{j - d_a +1}b + (j - d_a +1)a_{j - d_a +1}b 
= L_0a_{(j)}b  + (j - d_a +1)a_{(j)}b =0 $
for all  $j\in \ZZ_{\geq 0}$. Hence $L_0W^c \subset W^c$.  It follows that if $W\subset V$ is an unitary subalgebra then the 
corresponding coset subalgebra $W^c \subset V$ is also unitary.  
}
\end{example}

Now, suppose that $W\subset V$ is a unitary subalgebra. Then $W$, 
is a vertex algebra and it inherits from $V$ the normalized scalar product $(\cdot |\cdot )$.
We want to show that when $V$ is simple can we find a conformal vector 
for the vertex algebra $W$ making the pair $(W,(\cdot |\cdot))$  
into a simple unitary VOA. 
In order to do so, let us first note that the orthogonal
projection $e_W$ onto $W$, is a well-defined element in ${\rm End}(V)$. 
This is an easy consequence of the
fact that $W$ is compatible with the grading, and that the
subspaces $V_n\;(n\in\ZZ)$ are finite-dimensional. Note also 
that $e_W^+=e_W$. 
\begin{lemma}
\label{lemmaSubalgebraCommutation}
Let $W\subset V$ be a unitary subalgebra. Then $[Y(a,z),e_W] = 0$ 
for all $a\in W$, $[L_n,e_W]=0$ for $n=-1,0,1$ and $[\theta,e_W]=0$. 
Moreover, for every $a\in V$, $e_WY(a,z)e_W=Y(e_Wa,z)e_W$. 
\end{lemma}
\begin{proof}
Let $a \in W$. Since, for every $n\in \ZZ$, $W$ is invariant
for $a_{(n)}$ and $a_{(n)}^+$ we have $e_Wa_{(n)}= e_Wa_{(n)}e_W 
=(e_Wa_{(n)}^+e_W)^+ = (a_{(n)}^+e_W)^+=e_Wa_{(n)}$ for integer 
$n$. Hence $[Y(a,z),e_W] = 0$.  
By Prop. \ref{prop:QW<W}, $W$ is also invariant for $L_n$, 
$n=-1,0,1$ and for $\theta$. Since $L_1^+=L_{-1}$, $L_0^+=L_0$ and 
$\theta$ is an antiunitary involution it follows that $[L_n,e_W]=0$ for 
$n=-1,0,1$
and $[\theta,e_W]=0$. 

Now let $a\in V$. Then $e_WY(a,z)e_W\restriction_W$ is a field on $W$ which is 
mutually local with all vertex operators $Y(b,z)\restriction_W$ , 
$b \in W$. Moreover, $e_WY(a,z)e_W\Omega=e_We^{zL_{-1}}a
=e^{zL_{-1}}e_Wa$. By the uniqueness theorem 
\cite[Thm.4.4]{Kac}, we have $e_WY(a,z)e_W\restriction_W = 
Y(e_Wa,z)\restriction_W$. Thus $e_WY(a,z)e_W=Y(e_Wa,z)e_W$.
\end{proof}  
 
\begin{proposition}
\label{subalbebraconformalvector}
Let $(V, (\cdot|\cdot))$ be a simple unitary VOA with conformal 
vector $\nu$, $W$ be a unitary subalgebra of $V$ and  
$\nu^W=e_W\nu$. Then  $\theta \nu^W= \nu^W$ and $Y(\nu^W,z)=\sum_{n \in \ZZ}L^W_n z^{-n-2}$ 
is a Hermitian Virasoro field on $V$ such that $L^W_n\restriction_W=L_n\restriction_W$ for 
$n=-1,0,1$. In particular $\nu^W$ is a conformal 
vector for the vertex algebra $W$ and $W$ endowed with $\nu^W$ is 
a vertex operator algebra. 
Moreover, $(W,(\cdot |\cdot))$ with the conformal vector $\nu^W$ is 
a simple unitary VOA with PCT operator $\theta\restriction_W$.
\end{proposition}
\begin{proof}
By Lemma \ref{lemmaSubalgebraCommutation} $\nu_W$ is a
quasi-primary vector in $V_2$
and the coefficients in the expansion 
$Y(\nu^W,z)=\sum_{n\in \ZZ}L^W_nz^{-n-2}$ satisfy 
$L^W_ne_W=e_WL_ne_W$. Moreover, for $n=-1,0,1$ we also 
have $L^W_ne_W= L_ne_W$ and hence $L^W_n\restriction_W=L_n\restriction_W$.

From Borcherds commutator formula Eq. (\ref{commutatorformula})
and the fact that $L^W_j\nu^W=0$ for $j>2$ we have $(m, n \in \ZZ)$
\begin{eqnarray*}
[L^W_m, L^W_n] &=& (L^W_{-1}\nu^W)_{(m+n+2)} + 
(m+1)(L^W_0\nu^W)_{(m+n+1)} \\
&+& \frac{m(m+1)}{2}(L^W_1\nu^W)_{(m+n)} + 
\frac{m(m^2-1)}{6}(L^W_2\nu^W)_{(m+n-1)}.
\end{eqnarray*}
Now, $L^W_{-1}\nu^W=L_{-1}\nu^W$, $L^W_0\nu^W=L_0\nu^W=2\nu_W$ 
and $L^W_1\nu^W=L_1\nu^W=0$. Moreover, since $V$ is simple, we 
have $V_0=\CC \Omega$ by Prop. \ref{simpleunitary} so that 
$L^W_2\nu^W=\frac{c_W}{2}\Omega$ for some $c_W \in \CC$.  
Hence
\begin{eqnarray*}
[L^W_m, L^W_n] &=& -(m+n+2)(\nu^W)_{(m+n+1)} +
2(m+1)(\nu^W)_{(m+n+1)} \\
&+& \frac{c_W}{12}(m^3-m)\delta_{m,-n} 1_V \\
&=& (m-n)L^W_{m+n} + \frac{c_W}{12}(m^3-m)\delta_{m,-n}1_V, 
\end{eqnarray*}
i.e. $Y(\nu^W,z)$ is a Virasoro field with central charge
$c_W$. That $(W,(\cdot |\cdot))$ is a unitary VOA with PCT 
operator $\theta\restriction_W$ now follows directly from the fact that $W$ 
is invariant for $\theta$, and $L_n$, $n=-1,0,1$. Moreover, 
$W$ is simple by Prop. \ref{simpleunitary} because 
$W_0=W\cap V_0 =\CC\Omega$. 
\end{proof}

\begin{remark}
\label{RemarkSubalgebraCentralCharge} 
{\rm Let $W$ be a unitary subalgebra of a simple unitary vertex operator algebra. Then the following are equivalent:
\begin{itemize}
\item[$(i)$] $W=\CC \Omega$. 
\item[$(ii)$] $\nu_W=0$, where $\nu_W=e_W\nu$.
\item[$(iii)$] $c_W=0$, where $c_W$ is the central charge of $\nu_W$.
\end{itemize}
}
\end{remark}

\begin{proposition}
\label{cosetsconformalvector}
Let $(V, (\cdot|\cdot))$ be a simple unitary VOA with conformal
vector $\nu$, let $W$ be a unitary subalgebra of $V$ and
le $W^c$ be the corresponding coset subalgebra. 
Then we have $\nu = \nu^W+ \nu^{W^c}$. Moreover, the operators 
$L^W_0=\nu^W_{(1)}$ and $L^{W^c}_0=\nu^{W^c}_{(1)}$ are 
simultaneously diagonalizable on $V$ with non-negative eigenvalues.  
\end{proposition}
\begin{proof} Let $\nu'=\nu - \nu^W$ and let $a\in W$. By Prop. 
\ref{subalbebraconformalvector} we have $\nu'_{(j)}a=0$ for 
$j=0,1,2$. Hence by the Borcherds commutator formula 
Eq. (\ref{commutatorformula}) we have 
$[\nu'_{(m)}, a_{(n)}]=0$, for $m=0,1,2$ and all $n\in \ZZ$. 
Note also that since $\nu^W$ is quasi-primary then $[L_0,L^W_0]=0$ and hence
$z_1^{L^W_0}= e^{\log (z_1) L^W_0}$ is well defined on $V$.    As a consequence we find 
\begin{equation}
z_1^{L_0}Y(a,z)z_1^{-L_0}=z_1^{L^W_0}Y(a,z)z_1^{- L^W_0}
\end{equation}
and 
\begin{equation}
z_1^{ L^W_0}Y(\nu',z)z_1^{-L^W_0}=Y(\nu',z).
\end{equation}
Hence if $a\in W$ is homogeneous then 
\begin{eqnarray}
z_1^{L^W_0}[ Y(\nu',z),Y(a,w)]z_1^{L^W_0} &=&
[Y(\nu',z),z_1^{ L_0}Y(a,w)z_1^{- L_0}]  
\nonumber
\\
&=& w^{d_a}[Y(\nu',z),Y(a,z_1w)].
\end{eqnarray}
 
Hence, by locality, $[Y(\nu',z),Y(a,w)]=0$ for all homogeneous 
$a \in W$ so that $\nu' \in W^c$. The same argument also shows that 
$\nu-\nu^{W^c} \in W^{cc}$.
Accordingly, for every $b\in W^c$ we have 
\begin{equation}
[Y(\nu',z),Y(b,w)]=[Y(\nu,z),Y(b,w)]= [Y(\nu^{W^c},z),Y(b,w)]
\end{equation}
and thus $\nu' - \nu^{W^c} \in W^c\cap W^{cc}$. It follows 
that $Y(\nu' - \nu^{W^c},z)$ commutes 
with the energy-momentum field $Y(\nu,z)$ and hence with 
all vertex operators $Y(a,z)$, $a\in V$. Since $V$ is simple 
we have  $Y(\nu' - \nu^{W^c},z) \in \CC 1$ and hence 
$\nu'= \nu^{W^c}$ so that $\nu =\nu^W+\nu^{W^c}$.

Now, $L^W_0$ and $L^{W^c}_0$ coincide with their adjoints on $V$ 
and commute. Moreover, they commute with $L_0$ which is diagonalizable 
with finite-dimensional eigenspaces. Hence $L^W_0$ and $L^{W^c}_0$ 
are simultaneously diagonalizable on $V$ with real eigenvalues. 
It remains to show that these eigenvalues are in fact non-negative. 
Let $a \in V$ be a non-zero vector such that 
$L^W_0a=sa$ and $L^{W^c}_0a=ta$, $s,t \in \RR$. Assume that $s<0$. Then 
as a consequence of unitarity and of the fact that $Y(\nu^W,z)$ is a 
Virasoro field it easy to show that $(L^W_1)^na$ is non-zero for every 
positive integer $n$. Moreover, $L_0(L^W_1)^na=(s+t-n)(L^W_1)^na$ in 
contradiction with the fact that $L_0$ has non-negative eigenvalues. 
Hence $s\geq 0$ and similarly $t\geq 0$.       
\end{proof}

\begin{corollary} 
\label{CorollaryMinVirVOA}
Let $W$ be a unitary subalgebra of the unitary Virasoro VOA $L(c,0)$. Then, either $W=\CC\Omega$ or $W=L(c,0)$. 
\end{corollary}
\begin{proof} Since $L(c,0)_2 = \CC L_{-2}\Omega =\CC\nu$, see e.g. \cite{KaRa}, then either $W_2= \{0\}$ and hence $\nu_W=0$ 
so that $W=0$ by Remark  \ref{RemarkSubalgebraCentralCharge} 
or $W_2 = \CC\nu$ and hence $W=L(c,0)$, because  $L(c,0)$ is generated by $\nu$. 
\end{proof}

We conclude this section with the following example. 

\begin{example} 
\label{ExampleBabyMonster}
{\rm Let $V^\natural$ be the moonshine VOA. Then, $V^\natural$ is a framed VOA of rank 24 namely it is 
an extension of $L(1/2,0)^{\otimes 48}$, \cite{Miyamoto}.  In fact , $V^\natural$ contains the corresponding copy 
of $L(1/2,0)^{\otimes 48}$ as a unitary subalgebra. Now, let $W  \subset V^\natural$ be the unitary subalgebra of $V^\natural$ 
isomorphic to 
$L(1/2,0)$ corresponding to the embedding $L(1/2,0)\otimes \Omega^{\otimes47} \subset L(1/2,0)^{\otimes 48}\subset V^\natural$. 
Then $W^c$ is a simple unitary framed VOA of rank $47/2$, namely, the even shorter Moonshine vertex operator algebra 
$VB^\natural_{(0)}$ constructed by H\"{o}hn, 
see \cite[Sect. 1]{H02}. It has been proved by H\"{o}hn that the atomorphism group ${\rm Aut}(VB^\natural_{(0)})$ of $VB^\natural_{(0)}$
is the Baby Monster group $\mathbb{B}$, the second largest sporadic simple finite group, see \cite[Thm.1]{H02}.}
\end{example}

\section{Energy bounds and strongly local vertex operator algebras}
\label{sectionstronglylocalVOA}

Let $(V,(\cdot | \cdot))$ be a unitary VOA.   
We say that $a \in V$ (or equivalently the corresponding field 
$Y(a, z)$) satisfies (polynomial) {\bf energy bounds} if there exist 
positive integers $s, k$ 
and a constant $M >0$ such that, for all $n \in \ZZ$ and all 
$b \in V$

\begin{equation} 
\label{EnergyBounds}
\|a_n b \|\leq M (|n|+1)^s \|(L_0 +1_V)^k b \|.
\end{equation}  

If every $a\in V$ satisfies energy bounds we say 
that $V$ is {\bf energy-bounded}. Note that if $V$ is energy-bounded then, obviously, every unitary subalgebra 
$W\subset V$ is energy-bounded.

The following proposition will be useful.

\begin{proposition} 
\label{GenBoundedProp}
If $V$ is generated by a family of
homogeneous elements satisfying energy bounds then $V$ 
is energy-bounded. 
\end{proposition}
\begin{proof}
A linear combination of elements satisfying energy bounds also satisfies energy bounds. 
Moreover, if $a\in V_{d}$, then $(Ta)_n = -(n+d)a_n$ and hence if $a$ satisfies energy bounds, 
then so does $Ta$. However, starting from a generating set,
any element of $V$ can be obtained by a repeated use of:
derivatives (multiplication by $T=L_{-1}$), $(n)$-products with $n\geq 0$, $(n)$-product with
$n=-1$ (which correspond to {\bf normally ordered product} of vertex operators \cite[Sect.3.1]{Kac}), and linear combinations.
This follows from Eqs. (\ref{EqT1}) and (\ref{EqT2}), see also \cite[Sect.3.1 and Prop.4.4]{Kac}.

Derivatives or $(n)$-products of homogeneous elements are
homogeneous, and taking linear combinations ``commutes'' with
taking derivatives and with forming $(n)$-products. Thus it is
enough to show, that if $a$ and $b$ are homogeneous elements satisfying energy bounds, 
then $a_{(n)}b$ satisfies energy bounds for all $n\geq 0$ and for $n=-1$.

So suppose that $a,b\in V$ are homogeneous elements of conformal
weight $d_a$ and $d_b$, respectively, and that there exist some
positive $M_x, s_x$ and $r_x$ (where $x=a,b$) such that for all
$c\in V$ and $m\in \ZZ$, we have
\begin{equation}
\|x_m c\| \leq M_x (1+|m|)^{s_x} \|(1_V+L_0)^{r_x}c\|\;\;\;(x=a,b).
\end{equation}
As $[L_0,y_m]=-m y_m$, we have that $(1_V+L_0)^{r_x}y_m
= y_m \left((1-m)1_V+L_0\right)^{r_x}$ and so from the assumed energy
bounds it follows that for every $c\in V$
\begin{eqnarray}
\label{ABbound} \nonumber \|x_{m_1} y_{m_2}c\|\! &\leq& \! M_x
(1\! +\! |m_1|)^{s_x} \|({1_V}+L_0)^{r_x}y_{m_2}c\| \\
\nonumber \! &=& \! M_x (1\! +\! |m_1|)^{s_x} \|y_{m_2}((1\! -\!
m_2){1_V}+L_0)^{r_x}c\|
\\ \nonumber\! &\leq&\! M_x M_y (1\! +\! |m_1|)^{s_x} (1\! +\! |m_2|)^{s_y}
\|({1_V}+L_0)^{r_y}
((1\! - \! m_2){1_V}+L_0)^{r_x} c\| \\
\! &\leq&\! M_x M_y (1\! +\! |m_1|)^{s_x} (1\! +\! |m_2|)^{1+s_y}
\|({1_V}+L_0)^{r_x+r_y}c \|.
\end{eqnarray}
To have a bound for $(a_{m_1}b)_{m_2}$ rather than for
$a_{m_1}b_{m_2}$, we use the special case of the Borcherds
identity obtained by substituting $m=0$ into (\ref{B-id}):
\begin{equation}\label{specialBor}
(a_{(n)}b)_{(k)} = \sum_{j=0}^{\infty} (-1)^j\left(\begin{matrix}
n \\ j
\end{matrix}\right)\big(a_{(n-j)}b_{(k+j)} - (-1)^n
b_{(n+k-j)}a_{(j)}\big).
\end{equation}
When $n\geq 0$, there are at most $n+1$ possibly non-zero terms in
the sum appearing on the right-hand side, since if $j > n\geq 0$
then $\left(\begin{matrix} n \\ j
\end{matrix}\right) = 0$.
So using (\ref{ABbound}), it is straightforward to show that in this
case $a_{(n)}b$ satisfies energy bounds.

If $n=-1$, then in general $:\! ab\! :_m \equiv (a_{(-1)}b)_m$
cannot be reduced to a finite sum. As $:\! ab\! :$ is of conformal
weight $d_a+d_b$, by (\ref{specialBor}) we have
\begin{equation}
:\! ab\! :_m = \sum_{j\geq d_a} a_{-j} b_{m+j} +\sum_{j < d_a}
b_{m+j} a_{-j}.
\end{equation}
Nevertheless, to estimate $\|:\! ab\! :_m c\|$ for a $c\in
V_{(k)}$, we only have to deal with a finite sum, since $a_{-j}c
=0$ for $j< k$ and $b_{m+j}c=0$ for $j> k-m$. This, together with
(\ref{ABbound}), gives a $k$-depending bound for $\|:\! ab\! :_m
c\|$. But as $kc = L_0c$, the $k$-dependence can be easily
``integrated'' into the degree of $({1_V} + L_0)$.

Finally, if $c$ is not homogeneous, then it is a sum $c=\sum
c^{(k)}$ of homogeneous elements. Correspondingly, we may try to
``sum up'' our already obtained inequality for the homogeneous
vectors appearing in the sum.

Of course, in general the norm inequalities $\|v_k\|\leq \|w_k\|$
$(k=0,1,...)$ do not imply that $\|\sum_k v_k\| \leq \|\sum_k
w_k\|$. They do however, if one has some extra conditions; for
example that both $\{v_k: k=0,1,...\}$ and $\{w_k: k=0,1,...\}$
are sets of pairwise orthogonal vectors.

This is exactly our case, since by the corresponding eigenvalues
of $L_0$, one has that both $\{:\! ab\! :_mc^{(k)}\,: k=0,1,...\}$
and $\{({1_V} + L_0)^{r}c^{(k)}:k=0,1,...\}$ are sets of
pairwise orthogonal vectors. Hence the obtained bound is
applicable to every $c\in V$.
\end{proof}

\begin{corollary}\label{tensorbounded} If $V^\alpha$ and $V^\beta$ are energy-bounded VOAs then $V^\alpha\otimes V^\beta$ is 
energy-bounded.
\end{corollary}

\begin{proposition}
\label{DimensionOne/VirasoroBounded}
If V is a simple unitary VOA generated by $V_1 \cup {\mathscr F}$, where ${\mathscr F}\subset V_2$ is a family of quasi-primary $\theta$-invariant Virasoro vectors, then $V$ is energy-bounded. 
\end{proposition}

\begin{proof} From the commutator formula in Eq. (\ref{commutatorformula}) it follows that $V_1$ is a Lie algebra with brackets 
$[a,b] = a_0b$. Again from Eq. (\ref{commutatorformula})  we have that for $a, b \in V_1$,  $m, k \in \ZZ$,
\begin{equation*}
[a_{m},b_{k}] = [a,b]_{(m+k)} - m(\theta a | b) \delta_{m,-k}1_V, 
\end{equation*}
i.e. the operators $a_k$, $a\in V_1$, $k\in \ZZ$ satisfy affine Lie algebra commutator relations. 
As a consequence the vectors $a \in V_1 \cup {\mathscr F}$ satisfy the energy bounds in Eq. (\ref{EnergyBounds}) with $k=1$ (linear energy bounds), 
see e.g.  \cite[Sect.2]{BS-M}, and the conclusion follows from Prop. \ref{GenBoundedProp}.
\end{proof}

The first step in the construction of a conformal net associated with
the unitary VOA $(V, (\cdot|\cdot))$ is the definition of the complex 
Hilbert space ${\mathcal{H}}={\mathcal{H}}_{(V,(\cdot|\cdot))}$ as the 
completion 
of $V$ with respect to $(\cdot |\cdot)$. 
For every $a \in V$ and $n\in \ZZ$ we can consider $a_{(n)}$ has an 
operator on ${\mathcal H}$ with dense domain $V\subset {\mathcal H}$.
Due to the invariance of the scalar product $a_{(n)}$ has densely 
defined adjoint and hence it is closable.   
Now let $V$ be energy-bounded and let $f(z)$ be a smooth 
function on $S^1 = \{ z\in \CC: |z|=1 \}$ with Fourier coefficients 
\begin{equation}
\hat{f}_n = \int_{-\pi}^{\pi} 
f(e^{i\vartheta})e^{-in\vartheta}\frac{{\rm d}\vartheta}{2\pi}
=\oint_{S^1}f(z)z^{-n}\frac{{\rm d}z}{2\pi i z}
\end{equation}
For every $a \in V$ we define the operator $Y_0(a,f)$ 
with domain $V$ by
\begin{equation} 
Y_0(a,f)b = \sum_{n\in \ZZ} \hat{f}_n a_n b \quad {\rm for}\; b\in V. 
\end{equation} 
The sum converges in ${\mathcal H}$ due to the energy bounds and hence 
$Y_0(a,f)$ is a densely defined operator on ${\mathcal H}$ .
From the invariance of the scalar product it follows that $Y_0(a,f)$ has 
densely defined adjoint and hence it is closable. We denote $Y(a,f)$
the closure of $Y_0(a,f)$ and call it {\bf smeared vertex operator}. 
Note also that if the vector $a$ satisfies
the energy bounds
                                                                                
\begin{equation}
\|a_n b \|\leq M (|n|+1)^s \|(L_0 +1_V)^k b\|, \quad b\in V,
\end{equation}
then the operator $Y(a,f)$ satisfies 
\begin{equation}
\|Y(a,f) b \|\leq M \|f \|_s \|(L_0 +1_\H)^k b \|, \quad b\in V
\end{equation}
where 
\begin{equation}
\| f \|_s =\sum_{n \in \ZZ} (|n|+1)^s |\hat{f}_n |
\end{equation}
In particular the domain ${\mathcal H}^k$ of $(L_0+1_\H)^k$ is contained 
in the domain of $Y(a,f)$ and every core for the first operator is a core 
for the second.  
It follows that 
\begin{equation}
{\mathcal H}^\infty =\bigcap_{k \in \ZZ_{\geq 0}} {\mathcal H}^k
\end{equation}  
is a common core for the operators $Y(a,f)$, 
$f \in C^\infty (S^1)$, $a\in V$.   
Moreover, the map $f \mapsto Y(a,f)b$, 
$b\in {\mathcal H}^\infty$ is continuous and linear from 
$C^\infty (S^1)$ to ${\mathcal H}$ namely  $f \mapsto Y(a,f)$ is an operator valued distribution. 
Moreover, using the straightforward equality  
\begin{equation}
e^{itL_0}Y(a,f)e^{-itL_0}=Y(a,f_t), \quad t\in \RR,
\end{equation}
where $f_t$ is defined by $f_t(z) = f(e^{-it}z)$, 
and the energy bounds it is rather easy to show that, if $b \in \H^\infty$ then  
$Y(a,f)b \in \H^1$ and 
$$L_0Y(a,f)b=-iY(a,f')b + iY(a,f)L_0b ,$$
where $f'(e^{i\vartheta})= \frac{\rm d}{{\rm d}\vartheta}f(e^{i\vartheta})$.
It follows that $Y(a,f)b \in \H^\infty$ so that the common core ${\mathcal H}^\infty$ is invariant for all  
the smeared vertex operators. 

If $a\in V$ is homogeneous we can use the formal notation  

\begin{equation}
Y(a,f) = \oint_{S^1}Y(a,z)f(z)z^{d_a}\frac{{\rm d}z}{2\pi i z}.
\end{equation}

Note that if $a\in V$ is homogeneous and $L_1a=0$ we have the 
usual relation for the quasi-primary field $Y(a,z)$:
\begin{equation} 
(-1)^{d_a}Y(\theta a, \bar{f}) \subset Y(a,f)^*.
\end{equation}

If $a \in V$ is arbitrary $Y(a,f)^*$ still contains 
${\mathcal H}^\infty$ in its domain as a consequence of Eq. (\ref{Y(a,z)^+}).

Now we can associate with every interval $I \in {\mathcal I}$ 
a von Neumann algebra $\A_{(V,(\cdot|\cdot))} (I)$ by 
\begin{equation}
\A_{(V,(\cdot|\cdot))} (I) \equiv W^*(\{Y(a,f): a\in V, f \in C^\infty(S^1),\; {\rm supp}f \subset I 
\}).
\end{equation} 
The map $I\mapsto \A_{(V,(\cdot|\cdot))}(I)$ is obviously inclusion 
preserving. 
Moreover, it is not hard to show that $\Omega$ is cyclic for the von 
Neumann algebra 
\begin{equation} 
\A_{(V,(\cdot|\cdot))}(S^1) \equiv \bigvee_{I \in \I}\A_{(V,(\cdot|\cdot))}(I).
\end{equation} 

We now discuss covariance. The crucial fact here is that the unitary representation 
of the Virasoro algebra on $V$ associated with the conformal vector $\nu \in V$ gives rise to a strongly continuous 
unitary projective positive-energy representation of the covering group $\widetilde{\diff}$  of $\diff$
on $\H$ by \cite{GoWa1985,Tol1999} which factors through $\diff$ because
$e^{i2\pi L_0}=1$, see Subsect. \ref{SubsectPositiveEnergy}.

Hence there is a strongly continuous projective 
unitary representation $U$ of $\diff$ on $\H$ such that, 
for all $f \in C^\infty(S^1,\RR)$ and all $A \in B(\H)$,  
\begin{equation}
U(\Exp(tf\frac{\rm d}{{\rm d}\vartheta}))AU(\Exp(tf\frac{\rm d}{{\rm d}\vartheta}))^*= 
e^{itY(\nu,f)}Ae^{-iY(\nu,f)}, 
\end{equation}
Moreover, for all $\gamma \in \diff$ we have $U(\gamma)\H^\infty = \H^\infty$.

For any $\gamma \in \diff$ consider the function 
$X_\gamma :S_1 \to \RR$ defined by
\begin{equation}
X_\gamma(e^{i\vartheta})= -i\frac{{\rm d}}{{\rm d}\vartheta}\log (\gamma(e^{i\vartheta})).
\end{equation}
Since $\gamma$ is a diffeomorphism of $S^1$ preserving the orientation 
then $X_\gamma(z)>0$ for all $z \in S^1$. Moreover, $X_\gamma \in C^\infty(S^1)$. 
Another straightforward consequence of the definition is that 
\begin{equation}
X_{\gamma_1\gamma_2}(z)=X_{\gamma_1}(\gamma_2(z))X_{\gamma_2}(z). 
\end{equation}
It follows that, for any $d \in \ZZ_{>0}$ the family of continuous linear operators 
$\beta_d(\gamma)$, $\gamma \in \diff$  on the Fr\'echet space $C^\infty(S^1)$ 
defined by 
\begin{equation}
\label{covariancefunctions}
(\beta_d(\gamma)f)(z)= \left(X_\gamma(\gamma^{-1}(z))\right)^{d-1}f(\gamma^{-1}(z))
\end{equation}
gives a strongly continuous representation of $\diff$ leaving 
the real subspace of real functions invariant. 

\begin{proposition}
\label{covariancevertexoperator} 
If $V$ is a simple energy-bounded unitary VOA and $a\in V$ is a quasi-primary vector then 
$U(\gamma)Y(a,f)U(\gamma)^* = Y(a, \beta_{d_a}(\gamma)f)$
for all {\rm $\gamma \in \mob$.}
If $a\in V$ is a primary vector then 
$U(\gamma)Y(a,f)U(\gamma)^* = Y(a, \beta_{d_a}(\gamma)f)$ 
for all $\gamma \in \diff$.
\end{proposition}
\begin{proof}
Let $Y(\nu,z)=\sum_{n\in Z}L_n z^{-n-2}$ be the Virasoro field associated to the conformal vector $\nu$. 
The case in which $a$ is quasi-primary follows by a straightforward adaptation of the argument in pages 1100--1001 of \cite{CKL} and recalling the commutation relations between $a_n$ and $L_m$, $n\in \ZZ$, $m=-1, 0, 1$ given in Eq. 
(\ref{EqQuasi-Primary/PrimaryCommutation}).
 
The case in which $a$ is primary can be treated in a similar but taking into account the commutation relations  $a_n$ and $L_m$, 
$n,m \in \ZZ$ given again in Eq. (\ref{EqQuasi-Primary/PrimaryCommutation}). Note that for expository reasons in the proof in \cite{CKL} complete argument is given only for $\gamma \in \mob$ but the proof can be adapted to cover the case $\gamma \in \diff$ by noticing that as a consequence of the results in \cite{Tol1999} we have $e^{iY(\nu,f)}\H^\infty \subset \H^\infty$ for all $f\in C^\infty(\s1,\RR)$ and that $\diff$ is generated by exponentials of vector fields because it is a simple group \cite{Milnor}.
\end{proof}

We now discuss locality. It follows from Prop. \ref {PropVAandWightmannLocality} in Appendix \ref{AppendixVAandWightmannLocality}
that for any $a,b \in V$ the fields $Y(a,z)$ and $Y(b,z)$ are mutually local in the Wightman sense, i.e.  
for any $f, \tilde{f} \in C^\infty(\s1)$ with $\mathrm{supp} f \subset I$, $\mathrm{supp} \tilde{f} \subset I'$, $I \in \I$ we have 
\begin{equation}
[Y(a,f),Y(b,\tilde{f})]c=0
\end{equation}
for all $c\in \H^\infty$. As discussed in the Introduction and in Subsect.\ref{SubsectUnboundedOp&VNA}
this is {\it a priori} not enough to ensure the locality condition for the map $I\mapsto \A_{(V,(\cdot|\cdot))}(I)$.

\begin{lemma} 
\label{weakcommutantLemma}
Let $A$ be a bounded operator on ${\mathcal H}$, $a\in V$, and 
$I\in {\mathcal I}$ . 
Then $AY(a,f) \subset Y(a,f)A$ for all $f\in C^\infty(S^1)$ with 
${\rm supp}f \subset I$ if and only if 
$(A^*b|Y(a,f)c) = (Y(a,f)^*b|Ac)$
for all   $b,c \in V$,  and all real $f \in C^\infty (S^1)$ with
${\rm supp}f \subset I$.
\end{lemma}
\begin{proof}
The {\em only if} part is obvious.
The proof of {\em if } part is based on a rather straightforward adaptation 
of the proof of \cite[Lemma 5.4]{DSW}. Let us assume that 
$(A^*b|Y(a,f)c) = (Y(a,f)^*b|Ac)$
for all   $b,c \in V$,  and all real valued $f \in C^\infty (S^1)$ with ${\rm supp}f \subset I$. Then the same relation holds also for all complex valued $f \in C^\infty (S^1)$ with ${\rm supp}f \subset I$. Now let $f$ be a given function in 
$C^\infty (S^1)$ with ${\rm supp}f \subset I$. Then there is a $\delta >0$ 
such that the support of the function $f_t(z) \equiv f(e^{-it}z)$ is again 
contained in the open interval $I$ for all real numbers $t$ such 
that $|t| <\delta$. From the relation 
$e^{itL_0}Y(a,f)e^{-itL_0} = Y(a,f_t)$ for all $t\in \RR$ and the fact that $e^{itL_0}V=V$ for all $t\in \RR$ it then follows 
that, for all $b,c \in V$ and every smooth function $\varphi$ on $\RR$ 
with support in the open interval $(-\delta,\delta)$,
$(A(\varphi)^*b|Y(a,f)c) = (Y(a,f)^*b|A(\varphi)c)$, 
where $A(\varphi)= \int_\RR e^{itL_0}Ae^{-itL_0} \varphi(t) {\rm d}t$. 
Now, a standard argument shows that $A(\varphi)c \in {\mathcal H}^\infty$ 
for every $c\in V$ and from the fact that ${\mathcal H}^\infty$ is 
contained in the domain of $Y(a,f)$ we can conclude that 
$A(\varphi)Y(a,f)c = Y(a,f) A(\varphi)c$ for every smooth function
$\varphi$ on $\RR$ with support in $(-\delta,\delta)$ and every $c\in V$. 

For any real number $s \in (0,\delta)$ we fix a smooth positive function
$\varphi_s$ on $\RR$ with support in $(-s,s)$ and such that 
$\int_\RR \varphi_s(t){\rm d}t = 1$. 
For every $c\in V$ we then have 
$A(\varphi_s)Y(a,f)c = Y(a,f) A(\varphi_s)c$. Now, a standard 
argument shows that  if $s$ tends to $0$ $A(\varphi_s)$ tends to $A$ in 
the strong operator topology. Accordingly 
$\lim_{s \to 0} Y(a,f)A(\varphi_s)c = AY(a,f)c$ and 
$\lim_{s \to 0}A(\varphi_s)c = Ac$ for every $c\in V$. Since $Y(a,f)$ 
is closed it follows that $Ac$ is in domain of $Y(a,f)$ 
and $Y(a,f)Ac=AY(a,f)c$ for every $c\in V$ and since $V$ is a core 
for the closed operator $Y(a,f)$ it follows that 
$AY(a,f) \subset Y(a,f)A$.
\end{proof} 

The following proposition shows that the 
algebras $\A_{(V,(\cdot|\cdot))}(I)$ are generated by quasi-primary 
fields. 
\begin{proposition} 
\label{primarynetproposition}
Let $A$ be a bounded operator on $\H$ and let $I\in {\mathcal I}$. 
Then $A \in \A_{(V,(\cdot|\cdot))} (I)'$ if and only if  
$(A^*b|Y(a,f)c) = (Y(a,f)^*b|Ac)$ for all quasi-primary 
$a\in V$, all $b,c \in V$ and all real $f \in C^\infty(S^1)$ with 
${\rm supp}f \subset I$.
In particular 
\begin{equation}
\A_{(V,(\cdot|\cdot))} (I) = W^*(\{Y(a,f): a\in \bigcup_{k \in 
\ZZ}V_k,\;L_1a=0,\; f \in C^\infty(S^1, \RR), \; {\rm supp}f \subset I \}).
\end{equation}
\end{proposition}
\begin{proof}
Given $I\in {\mathcal I}$ we denote by ${\mathcal Q}(I)$ the set
of bounded operators $A$ such that 
$$(A^*b|Y(a,f)c) = (Y(a,f)^*b|Ac)$$ 
for all quasi-primary
$a\in V$, all $b,c \in V$ and all $f \in C_\RR^\infty(S^1)$ with ${\rm supp}f \subset I$. 
Then the same equalities hold also for all complex valued functions $f \in C_\RR^\infty(S^1)$ with 
${\rm supp}f \subset I$. 
It is evident that $\A_{(V,(\cdot|\cdot))}(I)' \subset {\mathcal Q}(I)$
and hence we have to show that 
${\mathcal Q}(I)\subset \A_{(V,(\cdot|\cdot))}(I)'$. Now, if  
$A \in {\mathcal Q}(I)$, $a \in V$ is quasi primary, $b, c \in V$ 
and $f \in C^\infty(S^1)$ has support in $I$ we have,
for all quasi-primary $a\in V$, all $b,c \in V$ and all $f \in 
C^\infty(S^1)$ with ${\rm supp}f \subset I$.

\begin{eqnarray*}
(Ab|Y(a,f)c) & = & (-1)^{d_a}(Ab|Y(\theta a,\bar{f})^*c)=
(-1)^{d_a}\overline{(Y(\theta a, \bar{f})^*c|Ab)} \\
& = & (-1)^{d_a}\overline{(A^*c|Y(\theta a, \bar{f})b)} = 
(-1)^{d_a}\overline{(A^*c|Y(\theta a, \bar{f})b)} \\ 
& = & (-1)^{d_a}(Y(\theta a, \bar{f})b|A^*c) = (Y(a,f)^*b|A^*c).
\end{eqnarray*}
It follows that $A^* \in {\mathcal Q}(I)$. 

Now let $a \in V$ be homogeneous. An elementary calculation shows 
that $(L_{-1}a)_n=-(n+d_a)a_n$ and hence that 
$Y(L_{-1}a,f) = Y(a,if' -d_a f)$ for every smooth function on 
$S^1$, where $f'(e^{i\vartheta})=\frac{d}{d\theta}f(e^{i\vartheta})$. It follows 
that, for a non-negative integer $k$, 
$Y((L_{-1})^ka,f) = Y(a, f_{(k,a)})$, where $f_{(k,a)}$ is a linear 
combination of $f, f', f'',\dots, f^{(k)}$. 
If ${\rm supp} f \subset I$ also
${\rm supp} f_{(k,a)} \subset I$ and hence if $a$ is quasi-primary we have 
\begin{eqnarray*}
(A^*b|Y((L_{-1})^ka,f)c) & = & (A^*b|Y(a,f_{(k,a)})c) = 
(Y(a,f_{(k,a)})^*b|Ac) \\
& = & (Y((L_{-1})^ka,f)^*b|Ac). 
\end{eqnarray*}
Since the Lie algebra representation determined by 
$L_{-1},L_0, L_1$ is completely reducible, $V$ is spanned by elements 
of the form $(L_{-1})^ka$ with $k$ a non-negative integer and 
$a$ quasi-primary. Hence, for all $a,b,c \in V$ we have 
$(A^*b|Y(a,f)c) = (Y(a,f)^*b,Ac)$. It follows from 
Lemma \ref{weakcommutantLemma} that $AY(a,f) \subset Y(a,f)A$
for all $a\in V$ and all $f\in C^\infty(S^1)$ with ${\rm supp}f\subset I$. 
Since also $A^* \in {\mathcal Q}(I)$ we also have 
$A^*Y(a,f) \subset Y(a,f)A^*$ and hence  
$AY(a,f)^* \subset Y(a,f)^*A$
for all $a\in V$ and all $f\in C^\infty(S^1)$ with ${\rm supp}f\subset I$.
It follows that $A \in \A_{(V,(\cdot|\cdot))}(I)'$
\end{proof}

From the covariance properties of quasi-primary fields it follows that 
the net is M\"obius covariant.

\begin{definition} {\rm We say that a unitary VOA $(V,(\cdot|\cdot))$ is 
{\bf strongly local} if it is energy-bounded and 
$\A_{(V,(\cdot|\cdot))}(I) \subset \A_{(V,(\cdot|\cdot))}(I')'$
for all $I\in \I$.} 
\end{definition} 

\begin{theorem}
\label{netV}
 Let ${(V,(\cdot|\cdot))}$ be a simple strongly 
local unitary VOA. Then the map $I\mapsto \A_{(V,(\cdot|\cdot))}(I) $ 
defines 
an irreducible conformal net $\A_{(V,(\cdot|\cdot))}$ on $S^1$. If 
$\{\cdot|\cdot\}$ is another normalized invariant scalar product on $V$ 
then $(V,\{\cdot|\cdot\})$ is again strongly local and 
$\A_{(V,(\cdot|\cdot))}$ and $\A_{(V,\{\cdot|\cdot\})}$ are isomorphic 
conformal nets. 
\end{theorem}  
\begin{proof} We only discuss covariance. The M\"{obius} covariance of the net follows from 
Prop. \ref{covariancevertexoperator} and Prop. \ref{primarynetproposition}. Then conformal (i.e. diffeomorphism) covariance follows from 
\cite[Prop.3.7]{Carpi2004}.

\end{proof}

Due to the above theorem, when no confusion arises, we shall 
denote the conformal net 
$\A_{(V,(\cdot|\cdot))}$ simply by $\A_V$. We shall say that 
$\A_V$ is the irreducible conformal net associated with the strongly local 
unitary simple vertex operator algebra $V$.

Using the strategy in \cite[Sect.5]{KL06} we can now prove the following theorem. 

\begin{theorem}\label{autnetsVOA} Let $V$ be a strongly local simple unitary VOA and let $\A_V$ be the corresponding irreducible conformal net. Then 
${\rm Aut}(\A_V) = {\rm Aut}_{(\cdot | \cdot)}(V)$. If  ${\rm Aut}(V)$ is finite then ${\rm Aut}(\A_V) = {\rm Aut}(V)$.
\end{theorem}
\begin{proof} 
Let  $\H$ be the Hilbert space completion of $V$. Then any $g\in {\rm Aut}_{(\cdot | \cdot)}(V)$ uniquely extends to a unitary operator on 
$\H$ again denoted by $g$. We have $g\Omega =\Omega$. Moreover, since $gY(a,f)g^{-1}=Y(ga,f)$ for all $a\in V$ and all 
$f \in C^\infty(S^1)$ we also have that $g\A(I)g^{-1}=\A(I)$ and hence $g \in  {\rm Aut}(\A_V)$. Conversely let $g \in  {\rm Aut}(\A_V)$.
Then $gL_ng^{-1} = L_n$ for $n=-1,0,1$. It follows that $g$ restricts to a linear invertible map $V \to V$ preserving the invariant scalar product $(\cdot | \cdot)$. For any $a\in V$ the formal series $gY(a,z)g^{-1}$ is a field on $V$ and, since $\A$ is local then, 
by  Prop. \ref{PropositionABcommute} and Prop. \ref {PropVAandWightmannLocality},
$gY(a,z)g^{-1}$ is mutually local (in the vertex algebra sense) with all $Y(b,z)$, $b\in V$. Moreover, $gY(a,z)g^{-1}\Omega = gY(a,z)\Omega = ge^{zL_{-1}}a=e^{zL_{-1}}ga$, where for the last equality we used \cite[Remark 1.3]{Kac}. Hence, by the uniqueness theorem for vertex algebras \cite[Thm.4.4]{Kac} we find 
that $gY(a,z)g^{-1}=Y(ga,z)$ and hence $g$ is a (linear) vertex algebra automorphism of $V$. Since $g$ commutes with $L_0$ we have 
$gV_n=V_n$ for all $n \in \ZZ$ and hence $g\nu=\nu$ by Corollary \ref{autcorollary}  so that $g\in {\rm Aut}_{(\cdot | \cdot)}(V)$. Now, if ${\rm Aut}(V)$ is finite then ${\rm Aut}(V) = {\rm Aut}_{(\cdot | \cdot)}(V)$ by Thm. \ref{uniqscalar} and hence 
${\rm Aut}(\A_V) = {\rm Aut}(V)$.

\end{proof}

We end this section with a new proof of the uniqueness result for diffeomorphism symmetry for irreducible conformal nets 
given in \cite[Thm.6.1.9]{WeiPhD}. The theorem was first proved in \cite{CW2005} using the additional assumption of 
4-regularity.

\begin{theorem}
\label{uniqdifftheorem}
Let $\A$ be an irreducible M\"{obius} covariant net on $S^1$ and let $U$ be the corresponding unitary representation of {\rm $\mob$}. 
If $U_\alpha$ and 
$U_\beta$ are two strongly-continuous projective unitary representations of $\diff$ extending $U$ and making into $\A$ an irreducible conformal net. 
Then $U_\alpha = U_\beta$.
\end{theorem}
\begin{proof} 
Let $\H$ be the vacuum Hilbert space of $\A$ and let $\H^{fin}$ be the algebraic direct sum of the eigenspaces 
${\rm Ker}(L_0 - n1_\H)$, $n \in \ZZ_{\geq 0}$. Then, by Thm. \ref{TheoremVirDiff},
then one can differentiate the representations $U_\alpha$ and $U_\beta$ in order to define two unitary 
representations of the Virasoro algebra on $\H^{fin}$ by operators $L^\alpha_n$, $n\in \ZZ$ and $L^\beta_n$, $n\in \ZZ$, see also \cite{Carpi2004,CW2005,loke}.  By assumption we have 
$L^\alpha_n=L^\beta_n$ for $n =-1,0,1$. The formal series $L^\alpha(z) = \sum_{n\in \ZZ}L^\alpha_n z^{-n-2}$ and 
$L^\beta (z) = \sum_{n\in \ZZ}L^\beta_n z^{-n-2}$
are fields on $\H^{fin}$ that are local and mutually local in the Wightman sense as a consequence of the locality of $\A$ and of 
Prop. \ref{PropositionABcommute}. 
Hence they are local and mutually local (in the vertex algebra sense) by Prop. \ref {PropVAandWightmannLocality}.
Let $V$ be the cyclic subspace generated from the action of the operators $L^\alpha_n, L^\beta_n$, $n\in \ZZ$ on the vacuum vector $\Omega$. By the existence theorem for vertex algebras, 
cf. \cite[Thm.4.5]{Kac}, $V$ is a conformal vertex algebra  of CFT type and it has two conformal vectors, $\nu^\alpha=L^\alpha_{-2}\Omega$ and 
$\nu^\beta=L^\beta_{-2}\Omega$. It satisfies $V_0=\CC\Omega$ and $L_1V_1=0$. 
Hence by \cite[Thm.1]{roitman} there exists a unique normalized invariant bilinear form $(\cdot,\cdot)$ on $V$ and this form satisfy
$(\Omega,a)=(\Omega|a)$ for all $a\in V$. By the invariance property of $(\cdot,\cdot)$ and the unitarity of the Virasoro algebra 
representations it follows that for any $b\in V$ we have $(a,b)=0$ for all $a\in V$ if and only if $(a|b)=0$ for all $a\in V$ i.e. if and only if 
$b=0$. Therefore, $(\cdot,\cdot)$ is non-degenerate. Accordingly, by Prop. \ref{uniquestress} and Remark \ref{remarkuniquestress} we have that $\nu^\alpha=\nu^\beta$ and hence $U^\alpha=U^\beta$. 
\end{proof}

\section{Covariant subnets and unitary subalgebras}

Let $W \subset V$ be a unitary subalgebra of the simple unitary vertex operator 
algebra $V$. Then, by Prop. \ref{subalbebraconformalvector}, 
$W$ is simple unitary vertex operator algebra. 

\begin{theorem} 
\label{stronglylocalsubalgebra}
Let $W$ be a unitary subalgebra of a strongly local 
simple unitary VOA $(V,(\cdot|\cdot))$.
Then the simple unitary VOA 
$(W,(\cdot|\cdot))$ is strongly local and $\A_W$ embeds canonically as a 
covariant subnet of $\A_V$.
\end{theorem} 
\begin{proof} 
Let $\H$ be the Hilbert space completion of $V$ and let $e_W$ be the 
orthogonal projection of $\H$ onto the closure $\H_W$ of $W$.
Then we have $$W=e_W V= \H_W \cap V.$$  
The vertex operator $\tilde{Y}(a,z)$, $a\in W$ of $W$  
coincides with the restriction to $W$ of $Y(a,z)$ and 
therefore it is obvious that $W$ satisfies energy 
bounds. Moreover, for $b\in V$, $f\in C^\infty(S^1)$ we have 
$$Y(a,f)e_W b \in \H_W,\quad Y(a,f)^*e_W b \in \H_W.$$

Hence for $a\in W$, $b, c \in V$ we have 
\begin{eqnarray*}
(b|e_W Y(a,f)c) & = & (Y(a,f)^*e_W b|c) = (Y(a,f)^*e_W b| e_W c) \\
& = & (e_W b|Y(a,f)e_W c) = (b|Y(a,f)e_W c)
\end{eqnarray*}
and being $V$ a core for $Y(a,f)$ it follows that $Y(a,f)$
commutes (strongly) with $e_W$.  

Now, define a covariant subnet $\B_W \subset \A_V$ by  
$$\B_W(I)= \A_V(I)\cap \{e_W \}' \quad I\in \I.$$  
It follows from the previous discussion that $Y(a,f)$ is affiliated with 
$\B_W(I)$ if $a\in W$ and $\mathrm{supp} f \subset I$. As a consequence 
$\H_{\B_W}=\H_W$ and hence the subnet net $\B_W$ is irreducible when 
restricted to $\H_W$. In particular, for all $I \in \I$ we have  
$$\left(\B_W(I)_{e_W}\right)'=\B_W(I')_{e_W}.$$   

Note also that, since for $a\in W$, $Y(a,f)$ commutes
with $e_W$ and $Y(a,f)b = \tilde{Y}(a,f)b$ for all
$b\in W$, then 
$$\D(\tilde{Y}(a,f))=e_W \D(Y(a,f)) = \D(Y(a,f)) \cap \H_W.$$
Hence, if ${\rm supp}f \subset I$, $\tilde{Y}(a,f)$ is affiliated 
with $\left( \B_W(I')_{e_W}\right)'=\B_W(I)_{e_W}$. It follows that the 
von Neumann algebras $\A_W(I)$, $I\in \I$ on $\H_W$ 
defined by 
\begin{equation*}
\A_W (I) \equiv W^*(\{\tilde{Y}(a,f): a\in W, {\rm supp}f \subset I
\})
\end{equation*}
satisfy $\A_W(I) \subset \B_W(I)_{e_W}$ for all $I\in \I$ proving 
that $(W,(\cdot|\cdot)$ is strongly local. Finally 
from Thm. \ref{netV} and Haag duality for conformal 
nets we find $\A_W(I) = \B_W(I)_{e_W}$ for all $I\in \I$. 
\end{proof}  

We now want to prove a converse of Thm. \ref{stronglylocalsubalgebra}. 
We begin with the following lemma. 

\begin{lemma}
\label{corelemma}
Let $A$ be a self-adjoint operator on a Hilbert space $\H$ and let $U(t)\equiv e^{itA}$, 
$t\in \RR$ be the corresponding strongly-continuous one-parameter group of unitary operators on $\H$. For any $k\in \ZZ_{\geq 0}$ let $\H^k$ denote the domain of $A^k$ and let $\H^\infty = \cap_{k \in \ZZ_{\geq 0}} \H^k$. Assume that there exists 
a real number $\delta >0$ and two dense linear subspaces $\D_\delta$ and $\D$ of $\H^\infty$ such that $U(t)\D_\delta \subset \D$ if $|t| <\delta$. Then, for every positive integer $k$, $\D$ is a core for $A^k$. 
\end{lemma}
\begin{proof} Let $k$ any positive integer and  
let $B$ denote the restriction of $A^k$ to $\D$. 
We have to show that $(A^k)^*= B^*$ and since $(A^k)^*\subset B^*$ it is enough to prove that 
$B^*\subset (A^k)^* =A^k$. 

Let $\D(B^*)$ denote the domain of $B^*$ and let $b \in \D(B^*)$. Then, by assumption
we have  
$$(U(t)A^k a | b) = (A^k U(t)a | b) = (U(t)a | B^*b), $$
for all $a \in \D_\delta$ and all $t\in (-\delta, \delta)$. Now let $\varphi: \RR \to \RR$ be a smooth non-negative function whose support is a subset of the interval $(-\delta, \delta)$. We can assume that 
$$\int_{-\infty}^{+\infty} \varphi(x) {\rm d}x =1.$$
For any positive integer $n$ let $\varphi_n :\RR \to \RR$ be defined by $\varphi_n (x) = n \varphi (nx)$, $x\in \RR$ so that 
${\rm supp}\varphi_n \subset (-\delta, \delta)$ and 
$$\widehat{\varphi_n}(p)  \equiv \int_{-\infty}^{+\infty} \varphi_n(x)e^{-ipx} {\rm d}x = \widehat{\varphi}(\frac{p}{n}),$$
for all $p\in \RR$. From equality $(U(t)A^k a | b) = (U(t)a | B^*b)$, $t\in \RR$ and the spectral theorem from self-adjoint operators it follows that 
$$(A^k\widehat{\varphi_n}(A)a | b)=(\widehat{\varphi_n}(A)a | B^*b),$$
for all $n\in \ZZ_{>0}$ and all $a \in \D_\delta$ and since  $A^k\widehat{\varphi_n}(A)$, and $\widehat{\varphi_n}(A)$ belong to $B(\H)$ for 
for every positive integer $n$ we also have that 
$$(A^k\widehat{\varphi_n}(A)a | b)=(\widehat{\varphi_n}(A)a | B^*b),$$
for all $n\in \ZZ_{>0}$ and all $a \in \H$. Now, it follows from the spectral theorem for self-adjoint operators that 
$\widehat{\varphi_n}(A)a \to a$ and $A^k \widehat{\varphi_n}(A)a \to A^ka$ for $n \to +\infty$, for all $a \in \H^k$. Hence 
$(A^k a | b)=(a | B^*b),$ for all $a \in \H^k$ so that $b\in \H^k$ and $A^k b=B^*b$. Thus, 
since $b \in \D(B^*)$ was arbitrary we can conclude that $B^*\subset A^k$. 
\end{proof}

 We will need the following proposition, cf. the appendix of \cite{Carpi1999} and \cite[Thm.2.1.3]{WeiPhD}

\begin{proposition}
\label{corenetproposition}
Let $\A$ be an irreducible M\"{o}bius covariant net on $S^1$ ant let $\H$ be its vacuum Hilbert space. Then 
$\A(I)\Omega \cap \H^\infty$ is a core for $(L_0 +1_\H)^k$ for all $I\in \I$ and all $k\in \ZZ_{\geq 0}$.
\end{proposition} 
\begin{proof} We first show that $\A(I)\Omega \cap \H^\infty$ is dense in $\H$ for all $I\in\I$. The argument is rather standard. For any 
$I\in \I$, let $I_1 \in \I$ be such that 
$\overline{I_1} \subset I$. Then there is a real number $\delta > 0$ such that $e^{it}I_1 \subset I$ for all 
$t\in (-\delta,\delta)$. Now now let $\varphi_n$, $n\in \ZZ_{>0}$, as in the proof of  Lemma \ref{corelemma}. 
Then, for any $A\in \A(I_1)$ we consider the operators $A_{\varphi_n}$, $n\in \ZZ_{>0}$ defined by 
$$(a | A_{\varphi_n}b) = \int_{-\infty}^{+\infty} \varphi_n(t)(a | e^{itL_0}Ae^{-itL_0}a){\rm d}t ,\quad a,b \in \H.$$
Then $A_{\varphi_n} \in \A(I)$ for all $n\in \ZZ_{>0}$. Moreover, 
$$A_{\varphi_n}\Omega = \widehat{\varphi_n} (L_0)A\Omega \in \H^\infty,\quad n\in \ZZ_{>0}.$$
Since $\widehat{\varphi_n} (L_0)A\Omega \to A\Omega$ for $n\to +\infty$ and $A\in \A(I_1)$ was arbitrary we can 
conclude that the closure of  $\A(I)\Omega \cap \H^\infty$ contains $\A(I_1)\Omega$ and hence it coincides with $\H$
by the Reeh-Schlieder property.  Hence, since $I$ was arbitrary we have shown that 
$\A(I)\Omega \cap \H^\infty$ is dense in $\H$ for all $I\in \I$. 

Now, let $I_1$ and $I$ and $\delta$ as above. We know that $\A(I_1) \cap \H^\infty$ is dense in $\H$.  
Moreover, 
\begin{eqnarray*}
e^{it(L_0+1_\H)}\left(\A(I_1)\Omega \cap \H^\infty\right) &=& \A(e^{it}I_1)\Omega \cap \H^\infty \\
&\subset&  \A(I)\Omega \cap \H^\infty,
\end{eqnarray*}
for all $t\in (-\delta,\delta)$. Hence, the conclusion follows from Lemma \ref{corelemma}.
\end{proof}

\begin{theorem}
\label{subnetsubalgebra} 
Let $(V, (\cdot|\cdot))$ be a simple strongly local unitary
VOA and let $\B$ a M\"{obius} covariant subnet of $\A_V$. Then 
$W = \H_\B \cap V$ is a unitary subalgebra of $V$ such that 
and $\A_W=\B$. 
\end{theorem} 
\begin{proof}
Since $\H_\B$ is globally invariant for the unitary representation of 
the M\"{o}bius group on $\H$ we have 
$\Omega \in W$ and $L_nW \subset W$ for 
$n=-1,0,1$. In particular $W$ is compatible with the grading of $V$
i.e it is spanned by the subspaces $W\cap V_n$, $n\in \ZZ_{\geq 0}$. 
Now let $a\in W$ and assume that, for a given positive 
integer $n$, $a_{(-n)}\Omega \in W$. Then 
$$a_{(-n-1)}\Omega = \frac{1}{n}[L_{-1},a_{(-n)}]\Omega 
=\frac{1}{n}L_{-1}a_{(-n)}\Omega 
\in W.$$  Since $a_{(-1)} \Omega = a \in W$ it follows that 
$a_{(n)}\Omega \in W$ for all $n\in \ZZ$ and all $a \in W$. 
Hence $Y(a,f)\Omega \in \H_\B$ for every smooth function 
$f$ on $S^1$ and every $a \in W$.  
Now let $e_\B$ be the projection of $\H_V$ onto $\H_\B$, 
$a \in W$, $f\in C^\infty(S^1)$ and, for $I \in \I$ let $\epsilon_{I'}$ be 
the unique vacuum preserving normal conditional expectation 
of $\A_V(I')$ onto $\B(I')$, see e.g. \cite[Lemma 13]{LongoCMP2003}. If ${\rm supp}f \subset I$ 
and $A\in \A_V(I')$ we find 
\begin{eqnarray*} 
Y(a,f)e_\B A\Omega & = & Y(a,f) \epsilon_{I'}(A) \Omega 
= \epsilon_{I'}(A) Y(a,f) \Omega \\
& = & e_\B A Y(a,f) \Omega =e_\B Y(a,f) A\Omega. 
\end{eqnarray*} 
Since $\A_V(I')\Omega $ is a core for $Y(a,f)$ by Prop. \ref{corenetproposition}
it follows that $Y(a,f)$ commutes with $e_\B$. 
Hence, $Y(a,f)$ and $Y(a,f)^*$ are affiliated 
with $\A(I)\cap {e_\B}'=\B(I)$. Now if $f$ is an 
arbitrary smooth function on $S^1$ it is now easy to see that 
$Y(a,f)$ and $e_\B$ again commute if $a \in W$. As a consequence 
we find that $a_nb \in W$ for all $a,b \in W$ and all $n\in \ZZ$ 
and hence $W$ is a vertex subalgebra. 
Moreover, using the fact that also $Y(a,f)^*$ and $e_\B$ commute 
for every smooth function $f$ on $S^1$ and all $a\in W$, we have 
$a_n^*b \in W$ for all $a, b\in W$ and all $n\in \ZZ$. 
Hence, since we also have $L_0W \subset W$, $W$ is a unitary 
subalgebra of $V$.
Finally that $\B(I)= \A_W(I)$ follows easily. 
\end{proof}

As a direct consequence of Thm. \ref{stronglylocalsubalgebra} and Thm. \ref{subnetsubalgebra} we get the following theorem. 

\begin{theorem}
\label{TheoremsubnetsubalgebraOneToOne} 
Let $V$ be a strongly local simple unitary vertex operator algebra. Then the map
$W\mapsto \A_W$ gives a one-to-one correspondence between the unitary subalgebras $W\subset V$ and the M\"obius covariant 
subnets $\B\subset \A_V$. 
\end{theorem}

\begin{proposition} 
\label{AVG}
Let $V$ be a simple unitary strongly local VOA and let $G$ be a closed subgroup of 
${\rm Aut}_{(\cdot | \cdot)}(V) = {\rm Aut}(\A_V)$. Then $\A_V^G=\A_{V^G}$. 
\end{proposition}
\begin{proof} For any $g\in G$ we have $gY(a,f)g^{-1} =Y(a,f)$ for all $a\in V^G$ and all $f\in C^\infty(S^1)$. Hence 
$g \in \A_{V^G}(I)'$ for all $I\in \I$ so that $\A_{V^G}\subset \A_V^G$.  Conversely, by Thm. \ref{subnetsubalgebra}  there is a 
unitary subalgebra $W \subset V$ such that $\A_V^G=\A_W$. Clearly $W\subset V^G$ and hence 
$\A_V^G \subset \A_{V^G}$.  
\end{proof}

We now can prove the following Galois correspondence for compact automorphism groups of strongly local vertex operator algebras 
(``Quantum Galois theory''), cf.  
\cite{DM1999,HMT}. 

\begin{theorem}
\label{GaloisVOA}
Let $V$ be a simple unitary strongly local VOA  and let $G$ be a closed subgroup of ${\rm Aut}_{(\cdot | \cdot)}(V)$. Then the map 
$H \mapsto V^H$ gives a one-to-one correspondence between the closed subgroups $H \subset G$ and the unitary subalgebras 
$W \subset V$ containing $V^G$.
\end{theorem}
\begin{proof} 
Let $W$ be a unitary subalgebra of $V$ such that $W \supset V^G$. Fix an interval $I_0 \in \I$. By Thm. \ref{stronglylocalsubalgebra} and 
Prop. \ref{AVG} we have 
$$\A_V(I_0)^G \subset \A_W(I_0) \subset \A_V(I_0).$$
Moreover, by \cite[Prop.2.1]{Carpi1999b}, the subfactor  $\A_V(I_0)^G \subset \A_V(I_0)$ is irreducible, i.e. 
$\left(\A_V(I_0)^G\right)' \cap \A_V(I_0) =\CC1$. Since ${\rm Aut}_{(\cdot | \cdot)}(V)$ and 
$G\subset {\rm Aut}_{(\cdot | \cdot)}(V)$ is closed then, $G$ is compact. Hence, by \cite[Thm.3.15]{ILP1998} there is unique closed subgroup $H \subset G$ such that $\A_W(I_0)=\A_V(I_0)^H$. Hence, by conformal covariance $\A_W(I)=\A_V^H(I)$ for all $I\in \I$ and hence, again by Prop. \ref{AVG}, $\A_W(I)=\A_{V^H}(I)$ and thus  $W=V^H$.  
\end{proof}

The following proposition shows that in the strongly local case the coset construction for VOAs corresponds exactly to the coset construction for conformal nets.

\begin{proposition}
\label{PropCosetSunetSubalgebra}
Let $V$ be a strongly local unitary simple VOA and let $W\subset V$ be a unitary subalgebra. 
Then $\A_W^c=\A_{W^c}$.
\end{proposition}
\begin{proof} Let $U_W$ be the projective unitary representation of $\widetilde{\diff}$ on $\H$ obtained from the representation of the 
Virasoro algebra on $V$ given by the operators $L^W_n$, $n \in \ZZ$ defined in Prop. \ref{cosetsconformalvector}. For an element 
$\tilde{\gamma} \in \widetilde{\diff}$ we denote by $\gamma \in \diff$ its image under the covering map
$\widetilde{\diff} \to \diff$. Then for any $\tilde{\gamma}\in\widetilde{\diff}$ and any $I\in \I$ we have 
$U_W(\tilde{\gamma})AU_W(\tilde{\gamma})^* = U(\gamma)AU(\gamma)^*$ for all $A \in \A_W(I)$ and 
$U_W(\tilde{\gamma})AU_W(\tilde{\gamma})^* = A$ for all $A \in \A_{W^c}(I)$. It follows that $A \in \A_{W^c}(I)$ commute with 
$\A_W(I_1)$ for every $I_1 \in \I$ and thus $\A_{W^c}(I) \subset \A_W^c(I)$ so that $\A_{W^c} \subset \A_W^c$.
On the other hand, by Thm. \ref{subnetsubalgebra}  there is a unitary subalgebra $\tilde{W} \subset V$ such that
$\A_W^c = \A_{\tilde{W}}$. Let $a\in \tilde{W}$. Then $Y(a,f)$ is affiliated with  $\A_W(S^1)'$ for all $I \in \I$ and all $f\in C^\infty(S^1)$
with ${\rm supp}f \subset I$. It follows that $Y(a,f)$ is affiliated with  $\A_W(S^1)'$ for all $f\in C^\infty(S^1)$. As a consequence 
$[Y(a,z),Y(b,w)]=0$ for all $b\in W$ and hence $a \in W^c$. Since $a\in \tilde{W}$ was arbitrary we can conclude that 
$\tilde{W} \subset W^c$ and hence that $\A_W^c \subset \A_{W^c}$. 

\end{proof}

We conclude this section with a result on finiteness of intermediate subalgebras for inclusions of strongly local vertex operator algebras, 
cf. \cite{KM_FPX,Xu2013}.

\begin{theorem}
\label{TheoremIntermediateSubalgebras} Let $V$ be a simple unitary strongly local vertex operator algebra and let $W \subset V$ be a unitary subalgebra. Assume that $[\A_V:\A_W] < +\infty$. Then the set of unitary subalgebras $\widetilde{W} \subset V$ such that 
$W \subset \widetilde{W}$ is finite. 
\end{theorem}
\begin{proof}
The claim follows directly from Thm. \ref{TheoremsubnetsubalgebraOneToOne} and the fact that that, since the index
$[\A_V:\A_W]$ is finite, the set of intermediate covariant subnets for the inclusion $\A_W \subset \A_V$ is also finite, see 
Subsect. \ref{SubsectCovariantSubnet}.
\end{proof}

\section{Criteria for strong locality and examples}
\label{Criteria for strong locality and examples}
In this section we consider some useful criteria which imply strong locality. We then apply them in order to give various examples of strongly local vertex operator algebras.  

Let $V$ be a simple unitary VOA satisfying energy bounds. 
If ${\mathscr F}$ is a subset of $V$ and $I\in {\mathcal I}$ we  
define a von Neumann subalgebra ${\mathcal A}_{\mathscr F}(I)$ of 
${\mathcal A}_V(I)$ 
\begin{equation} 
\label{EqA_F}
{\mathcal A}_{\mathscr F} (I) = W^*(\{Y(a,f): a\in {\mathscr F}, {\rm supp}f \subset I\}).
\end{equation}                                                           
                     
The following theorem is inspired by \cite[Thm.6.1]{DSW}.
\begin{theorem}
\label{generatingprimaryTheorem} Let ${\mathscr F} \subset V$ be a subset of 
the simple unitary energy-bounded VOA $V$. Assume that ${\mathscr F}$ 
contains only quasi-primary elements. Assume moreover that ${\mathscr F}$ generates $V$ and that, 
for a given $I \in {\mathcal I}$, ${\mathcal A}_{\mathscr F} (I')\subset 
{\mathcal A}_{\mathscr F} (I)'$. Then, $V$ is strongly local and  ${\mathcal A}_{\mathscr F} (I)={\mathcal A}_V (I)$
for all $I\in {\mathcal I}$. 
\end{theorem} 
\begin{proof} 
As a consequence of Lemma \ref{weakcommutantLemma} we have ${\mathcal A}_{\mathscr F}(I)= 
{\mathcal A}_{{\mathscr F}\cup \theta {\mathscr F}}(I)$, 
for all $I \in \I$. Accordingly we can assume that ${\mathscr F}=\theta {\mathscr F}$. 
We first observe that the map 
$I \mapsto {\mathcal A}_{\mathscr F} (I)$
is obviously isotonous and since every element of 
${\mathscr F}$ is quasi-primary it is also M\"{o}bius covariant as a consequence of Prop. \ref{covariancevertexoperator}.  
Hence ${\mathcal A}_{\mathscr F} (I')\subset{\mathcal A}_{\mathscr F} (I)'$ 
for all $I\in {\mathcal I}$. 

Now, let ${\mathcal P}_{\mathscr F}$ be the algebra generated by the operators 
$Y(a,f)$ with $a\in {\mathscr F}$, and $f \in C^\infty(S^1)$. Moreover, for 
$I\in {\mathcal I}$, let ${\mathcal P}_{\mathscr F}(I)$ 
be the subalgebra of ${\mathcal P}_{\mathscr F}$ corresponding to functions 
$f \in C^\infty(S^1)$ with ${\rm supp}f \subset I$. 
Both algebras have ${\mathcal H}^\infty$ as invariant domain
and are $*$-algebras because ${\mathscr F}$ is $\theta$ invariant.
Moreover, since ${\mathscr F}$ is generating $V\subset {\mathcal P}_{\mathscr F}\Omega$
and hence the latter subspace is dense in ${\mathcal H}$. 
With a slight modification of  the argument in \cite[page 544]{FJ} it 
can be shown that, for every $I\in{\mathcal I}$, 
$\overline{{\mathcal P}_{\mathscr F}(I)\Omega}$  is invariant for the action of the 
M\"{o}bius group and hence it is independent from the choice of $I$ and we 
denote it by ${\mathcal H}_{\mathscr F}$. Then, it can be shown that 
${\mathcal H}_{\mathscr F} \cap {\mathcal H}^\infty $ is left invariant 
by the algebras ${\mathcal P}_{\mathscr F}(I)$ for all $I\in {\mathcal I}$.
As a consequence ${\mathcal P}_{\mathscr F}\Omega \subset {\mathcal H}_{\mathscr F}$ 
and hence ${\mathcal P}_{\mathscr F}(I)\Omega$ is dense in  ${\mathcal H}$
for all $I\in {\mathcal I}$ (Reeh-Schlieder property for fields). 
Now, let $f \in C^\infty(S^1)$ have support in a given  
$I \in {\mathcal I}$ and $a\in {\mathscr F}$. Since $Y(a,f)$ is affiliated 
with ${\mathcal A}_{\mathscr F}(I)$ there is a sequence $A_n\in {\mathcal A}_{\mathscr F}(I)$
such that $\lim_{n \to \infty}A_nb =Y(a,f)b$ for all $b \in {\H^\infty}$. 
It follows that $\overline{{\mathcal A}_{\mathscr F}(I)\Omega}\cap 
{\mathcal H}^\infty$ is left invariant by the action of 
${\mathcal P}_{\mathscr F}(I)$ and hence 
${\mathcal P}_{\mathscr F}(I)\Omega \subset \overline{{\mathcal A}_{\mathscr F}(I)\Omega}$
which implies that also ${\mathcal A}_{\mathscr F}(I)\Omega$ is dense in 
${\mathcal H}$. Accordingly the map $I \mapsto {\mathcal A}_{\mathscr F}(I)$ 
also satisfies the cyclicity of the vacuum conditions and it thus define 
a local irreducible M\"{o}bius covariant net on $S^1$ acting on $\H$.
  
We have to show that ${\mathcal A}_V(I)\subset {\mathcal A}_{\mathscr F}(I)$ 
for all $I\in \I$. By M\"{o}bius covariance it is enough to prove the inclusion 
when $I$ is the upper semicircle $S^1_+$. Let $\Delta$ and $J$ be the 
Tomita's modular operator and modular conjugation associated with 
${\mathcal A}_{\mathscr F}(S^1_+)$ and $\Omega$ and let $S=J\Delta^{\frac{1}{2}}$. 
It follows from \cite[Prop.1.1]{GuLo96} that
$J{\mathcal A}_{\mathscr F}(I)J={\mathcal A}_{\mathscr F}(j(I))$ and 
$JU(\gamma)J= U(j\circ \gamma \circ j)$ for every  $I \in \I$ 
and every M\"{o}bius transformation $\gamma$ of $S^1$, where 
$j: S^1 \mapsto S^1$ is defined by $j(z)=\overline{z}$ ($|z|=1$).
It follows that $JL_nJ=L_n$ for $n=-1,0,1$. In particular $JV =V$ and
for every $a\in V$ 
the formal series $\Phi_a(z) = \sum_{n\in Z}Ja_{(n)}Jz^{-n-1}$ is 
a well defined field on $V$ such that 
$[L_1, \Phi_a(z) ]= \frac{d}{dz}\Phi_a(z)$ and 
$\Phi_a(z)\Omega|_{z=0}=Ja$ so that 
$\Phi_a(z)\Omega = e^{zT}Ja$. 
From the properties of the action 
of $J$ on the net ${\mathcal A}_{\mathscr F}$ one can show that, for $a,b,c \in {\mathscr F}$, 
$\Phi_a(z)$, $Y(b,z)$ and $Y(c,z)$ are pairwise mutually local fields (in the vertex algebra sense) 
as a consequence of the locality of $\A$ of 
Prop. \ref{PropositionABcommute} and Prop. \ref {PropVAandWightmannLocality}.
Hence, since ${\mathscr F}$ generates $V$, $\Phi_a(z)$ and $Y(b,z)$ are 
mutually local for every $a \in {\mathscr F}$ and every $b\in V$ as a consequence 
on Dong's lemma \cite[Lemma 3.2]{Kac}. It then easily follows that 
for all $a\in {\mathscr F}$ and all $b\in V$ also $Y(a,z)$ and $\Phi_b(z)$ 
are mutually local. Using again Dong's lemma and the fact that ${\mathscr F}$
generate $V$ we obtain that $\Phi_a(z)$ and $Y(b,z)$ are mutually local    
for all $a,b \in V$. Hence it follows from the uniqueness theorem 
for vertex algebras \cite[Thm.4.4]{Kac} that 
$\Phi_a(z)=Y(Ja,z)$ for every $a \in V$ and thus that 
$J$ defines an antilinear automorphism of $V$. 

Now let $a \in {\mathscr F}$ and let $f \in C^\infty(S^1)$ with 
${\rm supp}f\subset S^1_+$. Since $Y(a,f)$ is affiliated with 
$\A_{\mathscr F}(S^1_+)$ we have 
$J\Delta^{\frac{1}{2}}Y(a,f)\Omega = Y(a,f)^*\Omega$. On the other hand, 
since $a$ is quasi-primary, using the Bisognano-Wichmann property for $\A_{\mathscr{F}}$ and Thm. \ref{bw2} in Appendix \ref{BWmob}
we find $$\theta \Delta^{\frac{1}{2}}Y(a,f)\Omega = 
\theta e^{\frac{1}{2}K}Y(a,f)\Omega = Y(a,f)^*\Omega.$$ 
By the Bisognano-Wichmann property of M\"{obius} covariant nets on $S^1$ and the fact that 
$\theta L_n \theta = JL_nJ=L_n$ for $n=-1,0,1$ we see that both $J\Delta^{\frac{1}{2}}J$ and $\theta 
\Delta^{\frac{1}{2}}\theta$ 
are equal to $\Delta^{-\frac{1}{2}}$. Hence we find that 
$JY(a,f)\Omega = \theta Y(a,f)\Omega$. Since $\theta$ and 
$J$ commute with $L_0$ we  find that 
$JY(a,f_t)\Omega= \theta Y(a,f_t)\Omega$ for all $t\in \RR$. By partition 
of unity it follows that $JY(a,f)\Omega = \theta Y(a,f)\Omega$  for 
all $f \in C^\infty(S^1)$ and hence that $Ja=\theta a$. Since $a \in {\mathscr F}$ 
was arbitrary, $\theta$ and $J$ are antilinear automorphisms and ${\mathscr F}$ 
generates $V$ it follows that $\theta =J$. 
Hence, again by Thm. \ref{bw1} in Appendix \ref{BWmob}
we find that, for every 
quasi-primary element $a\in V$ and every $f\in C^\infty(S^1)$ with 
${\rm supp}f \subset S^1_+$, $Y(a,f)\Omega$ is in the domain of 
$S$ and $SY(a,f)\Omega =Y(a,f)^*\Omega $.

Now, let $I$ be an open interval containing the closure of 
$S^1_+$ and let $A \in \A_{\mathscr F}(I')$. Then there is a $\delta>0$ such 
that $e^{itL_0}Ae^{-itL_0} \in \A_{\mathscr F}(S^1_-)$ for all $t\in \RR$ such 
that $|t| <\delta$. Hence if $\varphi_s$, $s \in (0, \delta)$ 
and $A(\varphi_s)$ are defined as in the proof of Lemma 
\ref{weakcommutantLemma} we have that $A(\varphi_s) \in \A_{\mathscr F}(S^1_-)$ 
for all $s\in (0, \delta)$. 
Let $X_1, X_2 \in \P_{\mathscr F}(S^1_-)$ and $B \in \A_{\mathscr F}(S^1_+)$. Then we have 
\begin{eqnarray*}
(X_1^*A(\varphi_s)X_2\Omega|SB\Omega) & = & 
(X_1^*A(\varphi_s)X_2\Omega|B^*\Omega) \\
& = & (BX_1^*A(\varphi_s)X_2\Omega|\Omega) \\
& = & (X_1^*A(\varphi_s)X_2B\Omega|\Omega) \\
& = & (B\Omega|X_2^*A(\varphi_s)^*X_1\Omega). 
\end{eqnarray*}
As a consequence $X_1^*A(\varphi_s)X_2\Omega$ is in the domain of 
$S^*$ and 
$$S^*X_1^*A(\varphi_s)X_2\Omega = X_2^*A(\varphi_s)^*X_1\Omega.$$ 
Using this fact we find that, 
for every quasi-primary $a\in V$
every $f\in C^\infty(S^1)$ with ${\rm supp}f \subset S^1_+$ and all
$X_1, X_2 \in \P_{\mathscr F}(S^1_-)$,
\begin{eqnarray*}
(X_1\Omega|A(\varphi_s)Y(a,f)X_2\Omega) & = & 
(X_2^*A(\varphi_s)^*X_1\Omega|Y(a,f)\Omega) \\
& = & (S^*X_1^*A(\varphi_s)X_2\Omega|Y(a,f)\Omega) \\ 
& = & (SY(a,f)\Omega |X_1^*A(\varphi_s)X_2\Omega)\\
& = & (Y(a,f)^*\Omega |X_1^*A(\varphi_s)X_2\Omega) \\ 
& = & (Y(a,f)^*X_1\Omega |A(\varphi_s)X_2\Omega)\\
& = &(X_1\Omega |Y(a,f)A(\varphi_s)X_2\Omega), 
\end{eqnarray*}
$s\in (0,\delta)$. Hence, since $\P_{\mathscr F}(S^1_-)\Omega$ is dense 
we find that 
$A(\varphi_s)Y(a,f)X\Omega$ \linebreak $=Y(a,f)A(\varphi_s)X\Omega$ 
for all $X \in \P_{\mathscr F}(S^1_-)$ and all $s\in(0,\delta)$. 
Now, we have $\lim_{s\to 0}A(\varphi_s)c$ $= Ac$ for all $c\in \H$
and hence, for every $X \in \P_{\mathscr F}(S^1_-)$, 
$AX\Omega$ is in the domain of $Y(a,f)$ and 
$Y(a,f)AX\Omega = AY(a,f)X\Omega$.   
Since $V$ is energy-bounded by assumption, there exists a positive integer $k$ such that any core for $(L_0+1_\H)^k$ is a core 
for $Y(a,f)$. We want to show that $\P_{\mathscr F}(S^1_-)\Omega$ is a core for $(L_0+1_\H)^k$. To this end let $I \in \I$ whose closure is 
contained in $S^1_-$. Then there exists a real number $\delta >0$ such that $e^{it} I \subset S^1_-$ for all $t\in (-\delta,\delta)$. Hence, 
by the M\"{o}bius covariance of the vertex operators we see that $U(t)\P_{\mathscr F}(I)\Omega \subset\P_{\mathscr F}(S^1_-)\Omega$ for all 
$t\in (-\delta,\delta)$ and hence, by Lemma \ref{corelemma}, $\P_{\mathscr F}(S^1_-)\Omega$  is a a core for $(L_0+1_\H)^k$ and consequently a core
for $Y(a,f)$. It follows that $AY(a,f)\subset Y(a,f)A$ and since the latter relation holds for every 
$A \in \A_{\mathscr F}(I')$ it follows that $Y(a,f)$ is affiliated 
with $\A_{\mathscr F}(I)=\A_{\mathscr F}(I')'$ for all quasi-primary $a\in V$ and all 
$f\in C^\infty(S^1)$ with ${\rm supp}f \subset S^1_+$ . Hence 
using Prop. \ref{primarynetproposition} we can conclude 
that $\A_V(S^1_+) \subset \A_{\mathscr F}(I)$ whenever the interval $I\in \I$
contains the closure of $S^1_+$. Now, it follows easily from 
M\"{o}bius covariance that  
$$\A_{\mathscr F}(S^1_+) = \bigcap_{I\supset \overline{S^1_+}}\A_{\mathscr F}(I).$$
Hence we can conclude that 
$\A_V(S^1_+) \subset \A_{\mathscr F}(S^1_+)$. 
\end{proof}

\begin{corollary}\label{tensorcorollary} Let $V^\alpha$ and $V^\beta$ be strongly local simple unitary VOAs. Then 
$V^\alpha \otimes V^\beta$ is strongly local and $\A_{V^\alpha\otimes V^\beta} =\A_{V^\alpha}\otimes \A_{V^\beta}$.
\end{corollary}

\begin{proof} By Corollary \ref{tensorbounded} the simple unitary VOA $V^\alpha\otimes V^\beta$ is energy-bounded. Now let ${\mathscr F}_a$ be the 
family of all quasi-primary vectors in $V^\alpha$ and let ${\mathscr F}_\beta$ be the family of all quasi-primary vectors in $V^\beta$. 
Then, $V^\alpha \otimes V^\beta$
is generated by the family ${\mathscr F}$ of quasi-primary vectors in $V^\alpha \otimes V^\beta$ defined by  
${\mathscr F}\equiv \left({\mathscr F}_\alpha \otimes \Omega \right) \cup \left(\Omega\otimes {\mathscr F}_\beta \right)$ and 
$\A_{\mathscr F}(I)=\A_{V^\alpha}(I) \overline{\otimes} \A_{V^\beta}(I)$ 
for all $I\in \I$ so that 
$\A_{\mathscr F}(I') = \A_{V^\alpha}(I') \overline{\otimes} \A_{V^\beta}(I') \subset 
\left( \A_{V^\alpha}(I) \overline{\otimes} \A_{V^\beta}(I) \right)'$. 
Then the conclusion follows from Thm. \ref{generatingprimaryTheorem}.
\end{proof}

The following consequence of Thm.  \ref{generatingprimaryTheorem} is more directly applies to many interesting models.

\begin{theorem}
\label{DimensionOne/VirasoroTheorem}
If V is a simple unitary VOA generated by $V_1 \cup {\mathscr F}$, where ${\mathscr F}\subset V_2$ is a family of quasi-primary $\theta$-invariant Virasoro vectors, then $V$ is strongly local. 
\end{theorem}

\begin{proof} By Prop. \ref{DimensionOne/VirasoroBounded} (and its proof) $V$ is energy-bounded and 
the vectors $a \in V_1 \cup {\mathscr F}$ satisfy the energy bounds in Eq. (\ref{EnergyBounds}) with $k=1$ (linear energy bounds). Then, the argument in \cite[Sect.2]{BS-M} based on \cite{DF1977}, see also \cite[Sect.19.4]{GJ}, can be used to show that the von Neumann algebras 
$\A_{V_1 \cup {\mathscr F}}(I)$, $I\in \I$, satisfy the locality condition in Thm. \ref{generatingprimaryTheorem} so that 
$\A_{V_1 \cup {\mathscr F}}(I)=\A_V(I)$ for all $I\in \I$ and thus $V$ is strongly local.
\end{proof}

We now give various examples of VOAs that can be easily shown to be strongly local as a consequence of Thm. 
\ref{DimensionOne/VirasoroTheorem}. 

\begin{example} {\rm  The simple unitary vertex algebra $L(c,0)$ is strongly local. The corresponding irreducible conformal net 
$ \A_{L(c,0)}$ is the Virasoro net $\A_{\mathfrak{Vir}, c}$ defined in Subsect. \ref{SubsectPrelNets}. }
\end{example}

We use the above example to give an application of Thm. \ref{subnetsubalgebra} by giving a a new proof of the main result in 
\cite{Carpi1998}. 

\begin{theorem} 
\label{TheoremMinVirNet} 
Let $\B$ be a M\"{o}bius covariant subnet of the Virasoro net $\A_{\mathfrak{Vir}, c}$. Then, either $\B=\CC1_\H$ or 
$\B =\A_{\mathfrak{Vir}, c}$. 
\end{theorem}
\begin{proof} By Thorem \ref{subnetsubalgebra} there is a unitary subalgebra $W\subset L(c,0)$ such that $\B=\A_W$. The conclusion then follows from Corollary \ref{CorollaryMinVirVOA}.
\end{proof}

\begin{example} 
\label{ExampleStronglyLocalV_H}
{\rm Let $V_H$ be the (rank one) Heisenberg conformal vertex operator algebra \cite{Kac}. Then $V_H$ is generated by the one-dimensional subspace $(V_H)_1={\rm Ker}(L_0-1_{V_H})$ and hence it is strongly local.  The central charge is given by $c=1$. The corresponding conformal net $\A_{V_H}$ coincides with free Bose chiral field net $\A_{{\rm U}(1)}$ considered in \cite{BMT}.}
\end{example}

\begin{example}
\label{ExampleAffineLie}
{\rm Let $\g$ be a complex simple Lie algebra and let $V_{\g_k}$ be the corresponding level $k$ simple unitary 
VOA, see \cite{Kac1990,Kac,LL2004}. Then $V_{\g_k}$ is generated by $(V_{\g_k})_1 \simeq \g$ and hence it is strongly local.  The real Lie subalgebra 
$\g_\RR \equiv \{a \in \g: \theta a=a \}$ is a compact real form for $\g$.  Let $G$ be the compact connected simply connected real Lie group with simple Lie algebra $\g_\RR$.  Then $\A_{V_{\g_k}}$ coincides with the loop group conformal net $\A_{G_k}$ associated to the level 
$k$ positive-energy projective unitary representations of the loop group $LG$ \cite{GoWa1984,PS,Tol1999}, see 
\cite{FrG,JonesSemBourbaki,Tol1997,WasICM,WasA} (see also \cite[Sect.5]{KLX}).}
\end{example}

\begin{example}
\label{ExampleRankOneLattice} 
{\rm Let $n$ be a positive integer and let  $L_{2n} \equiv \ZZ \sqrt{2n}$ be the rank-one positive definite even lattice equipped with the 
$\ZZ$-bilinear form $\langle m_1 \sqrt{2n}, m_2 \sqrt{2n}\rangle \equiv 2n m_1 m_2$. Moreover, let $V_{L_{2n}}$ be the simple unitary lattice VOA with central charge $c=1$ associated with $L_{2n}$, see e.g. \cite[Sect.2]{DG1998}. Then $V_{L_{2n}}$ contains  the Heisenberg vertex operator algebra $V_H$ as a unitary subalgebra. Moreover, $V_{L_{2n}}$ describes the same CFT model as the irreducible conformal net $\A_{\mathrm{U}(1)_{2n}} \supset \A_{\mathrm{U}(1)}$ with $c=1$  and $\mu$-index equal to $2n$ considered in 
\cite{Xu2005}.  The net $\A_{\mathrm{U}(1)_{2n}}$ is denoted by $\A_N$, $N=n$ in \cite{BMT}. 
We have $V_{L_2} \simeq V_{\mathfrak{g}_1 }$ for $\mathfrak{g}=\mathfrak{sl}(2,\CC) = A_1$. For $n>1$ $V_{L_{2n}}$ can be realized, by a coset construction, as  a unitary subalgebra of  $V_{\mathfrak{g}_1 }$ for $\mathfrak{g}= D_{2n}$, see \cite[Sect.5B]{BMT}. It follows that 
$V_{L_{2n}}$ is strongly local for every $n \in \ZZ_{>0}$ and using the classification results in \cite{BMT} and \cite{Xu2005} it is not difficult to show that $\A_{V_{L_{2n}}} = \A_{\mathrm{U}(1)_{2n}}$.}
\end{example}

\begin{example}

{\rm The known $c=1$ simple unitary vertex operator algebras are 
\begin{equation}
\label{c=1VOA}
V_{L_2}^G, \; V_{L_{2n}}, \; V_{L_{2n}}^{\ZZ_2},
\end{equation}
where $G$ is a closed subgroup of $\mathrm{SO}(3)$ and $n$ is not the square of an integer, see \cite[Sect.7]{DVVV1989} and 
\cite[Sect.4]{Xu2005}. It follows from Example 
\ref{ExampleRankOneLattice} that all these vertex operator algebras are strongly local.
The corresponding $c=1$ irreducible conformal nets are the $c=1$ irreducible conformal nets classified in \cite{Xu2005} by assuming a certain ``spectrum condition''.
}
\end{example}

We now show another application of our general results by giving a new proof of \cite[Thm.3.2]{Carpi1999b}. 
Let us consider the case $\g = \mathfrak{sl}(2,\CC)$ and level $k=1$. Then $V_{\mathfrak{sl}(2,\CC)_1}$ has central charge 
$c=1$ and hence we have the embedding $L(1,0)\subset V_{\mathfrak{sl}(2,\CC)_1}$.

\begin{lemma} 
\label{LemmaVirMinSU(2)_1)}
Let $W$ be a unitary subalgebra of $V_{\mathfrak{sl}(2,\CC)_1}$. Then either $W=\CC \Omega$ or 
$W\supset L(1,c)$. 
\end{lemma}

\begin{proof} 
Assume first that $W_1 \neq \{0\}$. Then we can find a vector $a\in W$ such that $L_0a=a$, $\theta a=a$ and 
$\|a\|=1$. By the proof of Prop. \ref{DimensionOne/VirasoroBounded} we see that the operators $a_n$ satisfies the Heisenberg Lie algebra commutation relations
\begin{equation*}
[a_{m},a_{n}] = m\delta_{m,-n}1, 
\end{equation*}
for all $m, n \in \ZZ$ and hence a generate a copy of the Heisenberg vertex operator algebra $V_H$ inside $W$, cf. Example 
\ref{ExampleUnitaryV_H}. Since the central charge of $V_H$ is 1, $V_H$ have to contain the Virasoro subalgebra  $L(1,0)$ of $V$.  
Accordingly, $L(1,0) \subset W$.  

Assume now that $W_1 =\{0\}$. The characters formulae in \cite{Kac1990} gives for $q\in (0,1)$,
$${\rm Tr}_{V_{\mathfrak{sl}(2,\CC)_1}}  q^{L_0}=  \sum_{j \in \ZZ} q^{j^2}p(q), $$
where $p(q)=\prod_{n\in \ZZ_{>0}}(1-q^n)^{-1}$. Hence, 
\begin{equation}
{\rm Tr}_{V_{\mathfrak{sl}(2,\CC)_1}}  q^{L_0}= 1+3q +4q^2 + \cdots 
\end{equation}
so that the dimension of $\left({V_{\mathfrak{sl}(2,\CC)_1}}\right)_2$ is $4$.

 Since $W_1=\{0\}$, then 
$$(a | b_{-2}\Omega)= (a |L_{-1}b)=(L_1a|b)=0,$$
for all $a\in W$ and all $b \in  \left({V_{\mathfrak{sl}(2,\CC)_1}}\right)_1$. Hence 
$W_2$ is orthogonal to the three-dimensional subspace $\{a_{-2}\Omega : a\in \left({V_{\mathfrak{sl}(2,\CC)_1}}\right)_1 \}$. But also the conformal vector $\nu$ is orthogonal to the latter subspace since for any $a\in \left({V_{\mathfrak{sl}(2,\CC)_1}}\right)_1$ we have 
$$(\nu | a_{-2}\Omega)=(\Omega | [L_2,a_{-2}]\Omega ) = (\Omega | 2a_0\Omega)=0.$$ 

Hence $W_ 2 \subset \CC \nu$. Now, by Remark \ref{RemarkSubalgebraCentralCharge} if $\nu^W=0$ then $W=\CC\Omega$. Hence if 
$W\neq \CC\Omega$ then $W_2 =\CC \nu$ and hence $L(1,0) \subset W$. 
\end{proof}

Now, let $a\in \left({V_{\mathfrak{sl}(2,\CC)_1)}}\right)_1$. Then, by \cite[Remark 4.9c]{Kac} $e^{a_0}$  converges on
$V_{\mathfrak{sl}(2,\CC)_1}$ and defines an element in ${\rm Aut}\left(V_{\mathfrak{sl}(2,\CC)_1}\right)$. In fact, if $\theta a=a$
then $e^{a_0}$ is unitary i.e. $e^{a_0} \in {\rm Aut}_{(\cdot | \cdot)}\left(V_{\mathfrak{sl}(2,\CC)_1}\right)$,  and the group generated by such unitaries is isomorphic to ${\rm SO}(3)$. 
The following proposition was first proved in \cite{DG1998}, see also \cite{Reh1994}.

\begin{proposition} 
\label{PropVirFixedPoint}
The fixed point subalgebra  $V_{\mathfrak{sl}(2,\CC)_1}^{{\rm SO}(3)}$ coincides with the Virasoro subalgebra $L(1,0)$.
\end{proposition}

\begin{proof} By characters formulae for the unitary representations of affine Lie algebras, see e.g. \cite{Kac1990}, and for the unitary representations of the Virasoro algebra, see e.g.  \cite{KaRa}, one finds 
$${\rm Tr}_{V_{\mathfrak{sl}(2,\CC)_1}^{{\rm SO}(3)}}  q^{L_0}=  (1-q)p(q) =   
{\rm Tr}_{L(1,0)} q^{L_0},$$
see  \cite{DG1998,Reh1994}. Since $L(1,0) \subset V_{\mathfrak{sl}(2,\CC)_1}^{{\rm SO}(3)}$ the conclusion follows. 
\end{proof}

\begin{corollary} ${\rm Aut}_{(\cdot | \cdot)}\left(V_{\mathfrak{sl}(2,\CC)_1}\right)={\rm SO}(3)$. 
\end{corollary}

\begin{theorem}
\label{TheoremClassSubalgebrasSU(2)_1}
The map $H \mapsto V_{\mathfrak{sl}(2,\CC)_1}^H$ gives a one-to-one correspondence between the closed subgroups 
$H \subset {\rm SO}(3)$ and the unitary subalgebras $W \subset V_{\mathfrak{sl}(2,\CC)_1}$ such that $W\neq \CC\Omega$.
\end{theorem}
\begin{proof}
Let $W \subset V_{\mathfrak{sl}(2,\CC)_1}$ be a unitary subalgebra such that $W\neq \CC\Omega$. 
By Lemma \ref{LemmaVirMinSU(2)_1)} and by 
Prop. \ref{PropVirFixedPoint} $W$ contains the fixed point subalgebra $V_{\mathfrak{sl}(2,\CC)_1}^{{\rm SO}(3)}$
and the conclusion follows from Thm. \ref{GaloisVOA}.
\end{proof}

The following theorem is \cite[Thm.3.2]{Carpi1999b}

\begin{theorem} 
\label{TheoremClassSubnetsSU(2)_1}

The map $H \mapsto \A_{{\rm SU}(2)_1}^H$ gives a one-to-one correspondence between the closed subgroups 
$H \subset {\rm SO}(3)$ and the subnets  $\B \subset \A_{{\rm SU}(2)_1}$ of the loop group net  $\A_{{\rm SU}(2)_1}$ such that 
$\B \neq \CC 1$.
\end{theorem}
\begin{proof} It follows from Example \ref{ExampleAffineLie} that $\A_{{\rm SU}(2)_1}$ is the irreducible conformal net associated with the strongly local simple unitary  vertex operator algebra  $V_{\mathfrak{sl}(2,\CC)_1}$. The claim then follows from Thm. 
\ref{TheoremsubnetsubalgebraOneToOne}  and Thm. \ref{TheoremClassSubalgebrasSU(2)_1}.
\end{proof}

The next example is given by the moonshine vertex operator algebra $V^\natural$. As explained in Example \ref{ExampleUnitaryMoonshine} 
$V^\natural$ is a simple unitary VOA. We now show that it is strongly local. Note that the following theorem also gives a 
a new proof of \cite[Thm.5.4]{KL06}.

\begin{theorem} 
\label{TheoremStronglyLocalMoonshine} The moonshine vertex operator algebra $V^\natural$ is a simple unitary strongly local VOA. 
If $\A_{V^\natural}$ denotes the corresponding irreducible conformal net then ${\rm Aut}(\A_{V^\natural})$ is the Monster group $\mathbb{M}$. Moreover, up to unitary equivalence, $\A_{V^\natural} = \A^\natural$ where $\A^\natural$ is the moonshine conformal net constructed in \cite{KL06}. 
\end{theorem}
\begin{proof} 
By \cite[Lemma 5.1]{KL06} the moonshine vertex operator algebra $V^\natural$ is generated by a family ${\mathscr{F}}^\natural$ 
of  Hermitian quasi-primary Virasoro vectors in 
$V^\natural_2$ and hence , it is strongly local by Thm. \ref{DimensionOne/VirasoroTheorem}.  Moreover, by 
Thm. \ref{generatingprimaryTheorem}, 
$\A_{V^\natural}=\A_{\mathscr{F}^\natural}$, where $\A_{\mathscr{F}^\natural}$ is defined as in Eq. (\ref{EqA_F}). Since 
${\rm Aut}(V^\natural)=\mathbb{M}$ is finite then, by Thm. \ref{autnetsVOA}, ${\rm Aut}(\A_{V^\natural}) = \mathbb{M}$.  Moreover, by \cite[Corollary 5.3]{KL06}, 
$\A^\natural = \A_{\mathscr{F}^\natural}$ and hence $\A^\natural = \A_{V^\natural}$.
\end{proof}

As a consequence of Thm. \ref{stronglylocalsubalgebra}  also the unitary subalgebras of the above examples, such as orbifolds, cosets, etc., are strongly local. Further examples of strongly local VOAs are obtained by taking tensor products. 
All these examples give a rather large and interesting class of  strongly local VOAs. Moreover, they show that our results gives a uniform procedure to construct conformal nets associated to the corresponding CFT models. 
As an example we consider here the case of the even shorter moonshine vertex operator algebra $VB^\natural_{(0)}$, 
cf. Example \ref {ExampleBabyMonster}.

\begin{theorem}
\label{TheoremBabyMonster}
The even shorter moonshine vertex operator algebra $VB^\natural_{(0)}$ is a a simple unitary strongly local VOA. 
If $\A_{VB^\natural_{(0)}}$ denotes the corresponding net then ${\rm Aut}(\A_{BV^\natural_{(0)}})$ is the Baby Monster group
 $\mathbb{B}$.
\end{theorem} 
\begin{proof}
As explained in  Example \ref {ExampleBabyMonster} $VB^\natural_{(0)}$ is a unitary subalgebra of the moonshine 
vertex operator algebra $V^\natural$ and hence $VB^\natural_{(0)}$ is a simple unitary VOA. Since $V^\natural$ is strongly local 
by Thm. \ref{TheoremStronglyLocalMoonshine} then, also  $VB^\natural_{(0)}$ is strongly local as a consequence of 
Thm. \ref{stronglylocalsubalgebra}. Since ${\rm Aut}(VB^\natural_{(0)})=\mathbb{B}$ is finite then, by Thm. \ref{autnetsVOA}, 
${\rm Aut}(\A_{VB^\natural_{(0)}}) = \mathbb{B}$. 
\end{proof}

We conclude this section with two conjectures.

\begin{conjecture}
Let $L$ be an even positive definite lattice. Then the corresponding sumple unitary lattice VOA $V_L$ is strongly local and 
the corresponding conformal net $\A_{V_L}$ coincides with the lattice conformal net $\A_L$ constructed in 
\cite{DX2006}.
\end{conjecture} 

\begin{conjecture} 
Every simple unitary vertex operator algebra is strongly local and hence generates an irreducible conformal net $\A_V$.
\end{conjecture}

\section{Back to vertex operators}\label{FJsection} 
In this section we discuss problem of (re-) constructing vertex operator algebras starting from a given irreducible conformal net $\A$. This problem 
is related to the problem of constructing quantum fields from local net of von Neumann algebras. In particular we will prove that
for any strongly local vertex unitary operator algebra $V$ it is possible to recover all the vertex operators, and hence $V$ together with its VOA structure, from the conformal net $\A_V$.  To this end we will crucially rely on the ideas developed by Fredenhagen and J\"{o}r{\ss} in 
\cite{FJ} where pointlike-localized fields where defined starting from irreducible M\"{o}bius covariant nets. In fact we will give a variant of the construction in \cite{FJ} which avoids the scaling limit procedure considered there and completely relies on Tomita-Takesaki modular theory 
together with the results in Appendix  \ref{BWmob} of this article.

We first need to recall some facts by the Tomitata-Takesaki theory, see e.g. \cite[Sect.1.2]{Stratila} for details and proofs. Let 
${\mathcal M}$ be a von Neumann algebra on a Hilbert space $\H$ and let $\Omega \in \H$ be cyclic and separating for ${\mathcal M}$. As usual we denote by $S$ the Tomita operator associated with the pair $({\mathcal M},\Omega)$ and by $\Delta$ and $J$ the corresponding modular operator and modular conjugation respectively. 
Hence $S=J\Delta^{1/2}$. For $a\in \H$ consider the operator ${\mathscr L}_a^0$ with dense domain ${\mathcal M}'\Omega$ and  defined by 
${\mathscr L}_a^0A\Omega=Aa$, $A\in {\mathcal M}'$. If $a$ is in the domain $\D(S)$ it is straightforward to see that 
${\mathscr L}_{Sa}^0 \subset ({\mathscr L}_a^0)^*$ and hence 
${\mathscr L}_{Sa}^0$ and  $\mathscr{L}_{a}^0$ are closable and their closures ${\mathscr L}_{Sa}$ and ${\mathscr L}_a$ satisfy 
${\mathscr L}_{Sa} \subset {\mathscr L}_{a}^*$. Moreover, ${\mathscr L}_{Sa}$ and ${\mathscr L}_a$ are affiliated with ${\mathcal M}$. 
As pointed out in \cite{Carpi2005}, in certain situations the operators ${\mathscr L}_a$, $a \in \D(S)$ can be considered as abstract analogue of the smeared vertex operators, see also \cite{BBS}. Our variant of the Fredenhagen and J\"{o}r{\ss} construction will clarify this point of view.

Let $\A$ be an irreducible M\"{o}bius covariant net on $\s1$ acting on its vacuum Hilbert space $\H$. For any $I$ we can consider the Tomita operator $S_I=J_I \Delta_I^{1/2}$. The covariance of the net implies that for any $\gamma \in \mob$ we have 
$U(\gamma)S_IU(\gamma)^*=S_{\gamma I}$, $U(\gamma)J_IU(\gamma)^*=J_{\gamma I}$ and 
$U(\gamma)\Delta_I U(\gamma)^*=\Delta_{\gamma I}$.  Moreover, by the Bisognano-Wichmann property we have 
$\Delta_{S^1_+}^{it}=e^{iKt}$, $t\in \RR$. where $K\equiv i\pi \overline{(L_{1}-L_{-1})}$. Hence $\Delta_{S^1_+}^{1/2}=e^{\frac12 K}$. 
We will denote $J_{S^1_+}$ by $\theta$ (PCT operator). Then $\theta$ commutes with $L_{-1}$, $L_0$ and $L_1$. 

Now, let $a \in \H$ be a quasi-primary vector of conformal weight $d_a \in \ZZ_{\geq 0}$. 
Then, for every $f \in C^\infty(\s1)$ we can consider the vector $a(f)$ defined in  Appendix \ref{BWmob}, namely 
\begin{equation}
a(f)=\sum_{n\in \ZZ_{\geq 0}} \hat{f}_{-n-d_a}  \frac{1}{n!}L_{-1}^na.
\end{equation}
In the following for unexplained notations and terminology we refer the reader to Appendix \ref{BWmob}.

By Thm. \ref{bw1}, if ${\rm supp} f \subset S^1_+$ then $a(f)$ is in the domain of $S_{S^1_+}$ and 
\begin{equation}
S_{S^1_+}a(f)=(-1)^{d_a}(\theta a)(\overline{f}). 
\end{equation}
Hence the operator $A\Omega \mapsto Aa(f)$, $A \in \A(S^1_+)'$, is closable and its closure ${\mathscr L}^{S^1_+}_{a(f)}$ is affiliated with $\A(S^1_+)$. 
By the above stated covariance property of the modular operators $\Delta_I$, $I\in \I$ and Prop. \ref{covariancea(f)} we see that 
we can define in a similar way an operator ${\mathscr L}^I_{a(f)}$ for any $I\in \I$ and any $f\in C^\infty(S^1)$ with ${\rm supp}f \subset I$. Then by the discussion above and Prop. \ref{covariancea(f)} we have 
\begin{equation} 
\label{covarianceL}
U(\gamma){\mathscr L}^I_{a(f)}U(\gamma)^*={\mathscr L}^{\gamma I}_{a(\beta_{d_a}(\gamma)f)},
\end{equation}
for all $I\in \I$, all $f \in C^\infty(S^1)$ with ${\rm supp}f \subset I$ and all $\gamma \in \mob$. Moreover, 
\begin{equation}
\label{adjointL}
(-1)^{d_a} {\mathscr L}^I_{(\theta a)(\overline{f})} \subset ({\mathscr L}^I_{a(f)})^*
\end{equation}
for all $I\in \I$, and all $f \in C^\infty(S^1)$ with ${\rm supp}f \subset I$. 
Note also that also that for any $I\in \I$ and any $b\in \A(I)'\Omega$ the linear map $:C^\infty_c(I) \to \H$ given by  
$f \mapsto {\mathscr L}^I_{a(f)}b$ is continuous, namely $f \mapsto {\mathscr L}^I_{a(f)}$ is an operator valued distribution on 
$C^\infty_c(I_1)$. Note also that 
if $I_1 \subset I_2$, $I_1,I_2 \in \I$, and $f \in C^\infty_c(I_1)$ then ${\mathscr L}^{I_2}_{a(f)} \subset {\mathscr L}^{I_1}_{a(f)}$. 

All the above properties justify the following notation and terminology. For every quasi-primary vector $a\in \H$ and all $f\in C^\infty_c(I)$ 
we define  $Y_I(a,f)$ by $Y_I(a,f)\equiv {\mathscr L}^I_{a(f)}$. We call the operators $Y_I(a,f)$, $I \in \I$, $f \in C^\infty_c(I)$ 
Fredenhagen-J\"{o}r{\ss} (shortly FJ) smeared vertex operators or FJ fields.

The FJ smeared vertex operators have many properties in common with the smeared vertex operators. These are obtained 
simply by a change of notations for the corresponding properties of the operators ${\mathscr L}^I_{a(f)}$, $I\in \I$, $f\in C^\infty(S^1)$. 
First of all, for any $I\in \I$, $f \mapsto Y_I(a,f)$ is an operator valued distribution on $C^\infty_c(I_1)$ in the sense that 
the map $:C^\infty_c(I) \to \H$ given by  $f \mapsto Y_I(a,f)b$ is linear and continuous for every $b \in \A(I)'\Omega$. 
Moreover, the following compatibility condition holds
\begin{equation}
\label{EqFJcompatibility}
Y_{I_2}(a,f) \subset Y_{I_1}(a,f)
\end{equation}
if $I_1 \subset I_2$, $I_1,I_2 \in \I$, and $f \in C^\infty_c(I_1)$ so that if 
$b\in \A(I_2)'\Omega$ the vector valued distribution $C^\infty_c(I_2) \ni f \mapsto Y_{I_2}(a,f)b$ extends 
$C^\infty_c(I_1) \ni f \mapsto Y_{I_1}(a,f)b$. Finally, from Eq. (\ref{covarianceL}) and Eq. (\ref{adjointL}) we get the following covariance
and hermiticity relations
\begin{equation} 
\label{covarianceFJ}
U(\gamma)Y_I(a,f)U(\gamma)^*=Y_{\gamma I}(a, \beta_{d_a}(\gamma)f)),
\end{equation}
for all $I\in \I$, all $f \in C^\infty(S^1)$ with ${\rm supp}f \subset I$ and all $\gamma \in \mob$. Moreover, 
\begin{equation}
\label{adjointFJ}
(-1)^{d_a} Y_I(\theta a, \overline{f}) \subset Y_I(a,f)^*
\end{equation}
for all $I\in \I$, and all $f \in C^\infty(S^1)$ with ${\rm supp}f \subset I$. 

As usual for distributions we can use the formal notation 
\begin{equation}
Y_I(a,f) = \int_{I}Y_I(a,z)f(z)z^{d_a}\frac{{\rm d}z}{2\pi i z}.
\end{equation}
Then we say that the family $\{ Y_I(a,z) : I\in \I\}$ is an FJ vertex operator or an FJ field. Unfortunately there it is not known if the FJ smeared vertex operators admit a common invariant domain. Hence we cannot extend the family of distributions 
$\{ Y_I(a,z), z\in I : I\in \I\}$ to a unique distribution $\tilde{Y}(a,z)$. In particular the FJ fields cannot in general be considered as quantum fields in the sense of Wightman \cite{StrWight}.
\smallskip

The following proposition is a slightly weaker form of the result vi) stated in \cite[Sect.2]{FJ} and proved in \cite[Sect.4]{FJ}. 

\begin{proposition} The FJ smeared vertex operators generate the irreducible M\"{o}bius covariant net $\A$, namely 
$$\A(I)=
W^*(\{Y_{I_1}(a,f):a\in\bigcup_{k \in \ZZ_{\geq 0}}{\rm Ker}(L_0 - k1_\H),\,  L_1a=0,\, f \in C^\infty_c(I_1),\, I_1 \in \I,\, I_1\subset I\})$$
for all $I \in \I$.
\end{proposition}
\begin{proof} For any $I\in \I$ we define $\B(I)$ by 
$$\B(I)\equiv 
W^*(\{Y_{I_1}(a,f):a\in\bigcup_{k \in \ZZ_{\geq 0}}{\rm Ker}(L_0 - k1_\H),\, L_1a=0,\, f \in C^\infty_c(I_1),\, I_1 \in \I,\, I_1\subset I\}).$$
Clearly the family $\{\B(I) : I\in \I \}$ is a M\"{o}bius covariant subnet of $\A$. Let $\H_\B \equiv \overline{\B(S^1)\Omega}$ 
be the corresponding vacuum Hilbert space. Then $a(f) \in \H_\B$ for every quasi-primary vector $a\in \H$ and every 
$f\in C^\infty(S^1)$. Since the representation $U$ of $\mob$ is completely reducible the linear span of the vectors 
$a(f)$ with $a$ quasi-primary and  $f\in C^\infty(S^1)$ is dense in $\H$ so that $\H_\B =\H$ and thus $\B=\A$. 
\end{proof}

Our next goal in this section is to prove that the FJ smeared vertex operators of a conformal net $\A_V$ associated with 
a strongly local simple unitary VOA $V$ coincide with the ordinary smeared vertex operator of $V$. 

\begin{theorem} 
\label{backVOA}
Let $V$ be a simple unitary strongly local VOA and let $\A_V$ be the corresponding irreducible conformal net. 
Then, for any quasi-primary vector $a\in V$ we have $Y_I(a,f) = Y(a,f)$ for all $I\in \I$ and all 
$f\in C^\infty_c(I)$, i.e. the smeared vertex operator of $V$ coincide with the FJ smeared vertex operator of $\A_V$. 
In particular one can recover the VOA structure on $V= \H^{fin}$ from the conformal net $\A_V$. 
\end{theorem}
\begin{proof} We first observe that, for any $f\in C^\infty_c(I)$, $Y(a,f)$ is affiliated with $\A(I)$ and hence its domain 
contains $\A(I)'\Omega \supset \A(I') \cap \H^\infty$. Since the latter is a core for $Y(a,f)$, by Prop. 
\ref{corenetproposition} then 
also $\A(I)'\Omega$ is a core for the same operator. On the other hand $\A(I)'\Omega$ is a core for $Y_I(a,f)$ by definition. 
Using  Prop. \ref{Y(a,f)Omega=a(f)} in Appendix \ref{BWmob}, for any $A \in \A(I)'$ we find 
$$Y(a,f)A\Omega=AY(a,f)\Omega=Aa(f)= Y_I(a,f)A\Omega.$$
Accordingly the closed operators $Y(a,f)$ and $Y_I(a,f)$ coincides on a common core and hence they must be equal.
\end{proof}

We now consider a general irreducible conformal net $\A$. We want to find conditions on $\A$ which allow to prove that 
$\A = \A_V$ for some simple unitary strongly local VOA $V$. As a consequence of Thm. \ref{backVOA} a necessary 
condition is that for every primary vector $a\in \H$ the corresponding FJ vertex operator $\{ Y_I(a,z) : I\in \I\}$ satisfies energy bounds i.e. there exist a real  number $M>0$ and positive integers $k$ and $s$ such that 
\begin{equation}
\|Y_I(a,f)b\| \leq M \|f \|_s\|(L_0 +1_\H )^kb\| 
\end{equation} 
for all $I\in \I$, all $f\in C^\infty_c(I)$ and all $b\in \A(I)'\Omega \cap \H^\infty$. We will see that the condition is also sufficient and that actually it can be replaced by an apparently weaker condition. 
\smallskip

We say that a family ${\mathscr F} \subset \H$ of quasi-primary vectors generates $\A$ if  the corresponding 
FJ smeared vertex operators generates the local algebras i.e. if 
\begin{equation}
\A(I)= W^*(\{Y_{I_1}(a,f): a\in \mathscr{F},\; f \in C^\infty_c(I_1),\; I_1 \in \I,\; I_1\subset I\}).\end{equation} 

\begin{theorem}
\label{constructVOA}
Let $\A$ be an irreducible conformal net that is generated by a family of quasi-primary vectors ${\mathscr F}$. Assume 
$\theta {\mathscr F} ={\mathscr F}$ and that 
for every $a \in {\mathscr F}$ the FJ vertex operator $\{ Y_I(a,z) : I\in \I\}$ satisfies energy bounds. Moreover, assume that 
${\rm Ker}(L_0 - n1_\H)$ is finite-dimensional for all $n\in \ZZ_{\geq 0}$. Then, the vector space $V\equiv \H^{fin}$ admits 
a VOA structure making $V$ into a simple unitary strongly local VOA such that $\A_V = \A$. 
\end{theorem}  

\begin{proof} By the same argument used for the ordinary smeared vertex operator in Sect. \ref{sectionstronglylocalVOA} it can be shown that the energy bounds imply that $\H^\infty$ is a common invariant core for the operators $Y_I(a,f)$, $I\in \I$, $f\in C^\infty_c(I)$, 
$a\in {\mathscr F}$.
Let $\{I_1, I_2\}$, $I_1, I_2 \in \I$ be a cover of $S^1$ and let $\{\varphi_1, \varphi_2\}$, 
$\varphi_1, \varphi_2 \in C^\infty(S^1,\RR)$ be a partition of unity on $\s1$ subordinate to $\{I_1, I_2\}$, namely 
${\rm supp} \varphi_k \subset I_k$, $k=1,2$, and $\sum_{j=1}^2\varphi_k(z)=1$ for all $z\in S^1$.
For any $a\in {\mathscr F}$ and any $f\in C^\infty(\s1)$ we define an operator $\tilde{Y}(a,f)$ on $\H$ with domain $\H^\infty$ by 
$$\tilde{Y}(a,f)b=\sum_{j=1}^2Y_{I_j}(a,\varphi_j f)b, \quad b \in \H^\infty.$$
Let $\{\tilde{I}_1, \tilde{I}_2\}$, $\tilde{I}_1, \tilde{I}_2 \in \I$ be another cover of $S^1$ and 
$\tilde{\varphi}_1, \tilde{\varphi}_2 \in C^\infty(S^1,\RR)$ be a partition of unity on $\s1$ subordinate to 
$\{\tilde{I}_1, \tilde{I}_2\}$. Then, using the compatibility conditions in Eq. (\ref{EqFJcompatibility}) for the FJ smeared vertex operator 
we find that 
\begin{eqnarray*}
\sum_{j=1}^2Y_{I_j}(a,\varphi_j f)b &=& \sum_{j,m=1}^{2} Y_{I_j}(a,\tilde{\varphi}_m\varphi_j f)b \\
&=& \sum_{j,m=1}^{2} Y_{\tilde{I}_m}(a,\tilde{\varphi}_m\varphi_j f)b \\
&=&  \sum_{m=1}^{2} Y_{\tilde{I}_m}(a,\tilde{\varphi}_m f)b
\end{eqnarray*}
for all $b \in \H^\infty$. 
Hence, $\tilde{Y}(a,f)$ does not depend on the choice of the partition of unity $\{\varphi_1, \varphi_2 \}$ nor on the choice of the 
cover $\{I_1, I_2 \}$. It follows that for any $I\in \I$ and any $f\in C^\infty_c(I)$ we have 
$\tilde{Y}(a,f)b=Y_{I}(a,f)b$ for all $b\in \H^\infty$. Moreover, we have the covariance property
\begin{equation*} 
U(\gamma)\tilde{Y}(a,f)U(\gamma)^*=\tilde{Y}(a, \beta_{d_a}(\gamma)f)),
\end{equation*}
the adjoint relation 
$$(-1)^{d_a} \tilde{Y}(\theta a, \overline{f}) \subset \tilde{Y}(a,f)^*$$
and the state field correspondence ${\tilde Y}(a,f)\Omega=a(f)$
for all $f \in C^\infty(S^1)$ and all $\gamma \in \mob$.  

By assumption the FJ vertex operator  $\{ Y_I(a,z) : I\in \I\}$ satisfies energy bounds with a real number $M>0$ and positive integers $s,k$. 
Given $\varphi ,f \in C^\infty(\s1)$ we have 
$$
(|n| + 1)^s|\widehat{(\varphi f)}_n| \leq \sum_{j \in \ZZ}(|n| + 1)^s |\hat{f}_j|\cdot |\hat{\varphi}_{n-j}|  
$$
hence 
\begin{eqnarray*}
\|\varphi f \|_s &=& \sum_{n\in \ZZ} (|n| + 1)^s|\widehat{(\varphi f)}_n|   \\
&\leq& \sum_{n,j \in \ZZ}(|n| + 1)^s |\hat{f}_j|\cdot |\hat{\varphi}_{n-j}| \\
&=& \sum_{j,m \in \ZZ}(|m+j| + 1)^s |\hat{f}_j|\cdot |\hat{\varphi}_{m}| \\
&\leq& \|\varphi \|_s  \|f \|_s.
\end{eqnarray*}
It follows that 
\begin{eqnarray*}
\|\tilde{Y}(a,f)b\| &=& \|\sum_{j=1}^2Y_{I_j}(a,\varphi_j f)b \| \\
&\leq& M \left(\sum_{j=1}^2  \|\varphi_j \|_s\right) \|f \|_s\|(L_0+1_\H)^kb \|
\end{eqnarray*}
for all $f\in C^\infty(\s1)$ and all $b \in \H^\infty$, i.e. the operators $\tilde{Y}(a,f)$, $f\in C^\infty(\s1)$ satisfy energy bounds
with the same positive integers $s,k$ and the positive constant $\tilde{M}\equiv M \left(\sum_{j=1}^2  \|\varphi_j \|_s\right)$.  

Now, let $e_n \in C^\infty(\s1)$, $n\in \ZZ$, be defined by $e_n(z)=z^n$, $z\in \s1$.  For every $a\in {\mathscr F}$ we define 
$a_n \equiv \tilde{Y}(a,e_n)$, $n\in \ZZ$. 
We have
$$\|a_nb\| \leq 2^s \tilde{M}(|n|+1)^s\|(L_0 +1_\H)^kb\|,$$
for all $n\in \ZZ$ and all $b\in \H^\infty$.  By the covariance property we have $e^{itL_0}a_ne^{-itL_0} = e^{-int} a_n$ for all 
$t \in \RR$. It follows that  $[L_0,a_n]b =-na_nb$ for all $n\in \ZZ$ and all $b\in \H^\infty$ and hence that $a_n\H^{fin} \subset \H^{fin}$ for all $n \in \ZZ$. The covariance properties also implies that $[L_{-1},a_n]b=(-n-da+1)a_{n-1}b$ and $[L_{1},a_n]b=-(n-da+1)a_{n+1}b$
for all $n\in \ZZ$ and all $b\in \H^\infty$. Moreover, we have $a_{-d_{a}} \Omega = a(e_{-d_a})=a$ for all $a \in {\mathscr F}$. 
Now let, $V\subset \H^{fin}$ be the linear span of the vector of the form 
$$a^1_{n_1}a^2_{n_2} \cdots a^k_{n_k}\Omega,$$
with $a^1, a^2, \cdots, a^k \in {\mathscr F}$ and $n_1, n_2, \cdots, n_k \in \ZZ$.
We want to show that $V=\H^{fin}$. Let $\H_V \subset \H$ be the closure of $V$ and $e_V$ be the orthogonal projection onto $\H_V$. First note that the series $\sum_{n \in \ZZ}\hat{f}_n e_n$ converges to $f$ in $C^\infty(S^1)$ and thus
$$\sum_{n \in \ZZ}\hat{f}_n  a_nb = \tilde{Y}(a,f) b$$
for all $a\in {\mathscr F}$, all $b\in \H^\infty$ and all $f\in C^\infty(\s1)$. It follows that $\tilde{Y}(a,f)b$ and $\tilde{Y}(a,f)^*b$ belong to 
$\H_V$ for all $a\in {\mathscr F}$, all $b\in V$ and all $f\in C^\infty(\s1)$. 

From the fact that $\equiv {\rm Ker}(L_0 -n1_\H)$ is finite-dimensional for all $n \in \ZZ_{\geq 0}$ it follows that $e_V \H^{fin} =V$. 
As consequence we have $[e_V,\tilde{Y}(a,f)]b=0$ for all $a\in {\mathscr F}$, all $b\in \H^{fin}$ and all $f\in C^\infty(\s1)$. 
Recalling that $\H^{fin}$ is a core 
for every FJ smeared vertex operator we can conclude that $e_VY_I(a,f) \subset Y_I(a,f)e_V$ for all $a \in {\mathscr F}$, all $I \in \I$ and all 
$f \in C^\infty_c(I)$. Hence, since the family ${\mathscr F}$ generates the net $\A$, we see that $ e_V=1_\H$ by the irreducibility of $\A$ so that 
$V = \H^{fin}$. 

The above properties imply that the formal series 
$$\Phi_a(z) \equiv \sum_{n\in \ZZ} a_n z^{-n-d_a} , \quad a \in {\mathscr F}$$
are fields on $V$ that are local and mutually local (in the vertex algebra sense) as a consequence of the locality of the conformal net $\A$ 
and Prop. \ref{PropVAandWightmannLocality}.
In fact they satisfy all the assumption of the existence theorem for vertex algebras \cite[Thm.4.5]{Kac}. Accordingly $V$ is a vertex algebra whose vertex operators satisfy $Y(a,z)=\Phi_a(z)$ for all $a\in {\mathscr F}$. A unitary representation of the Virasoro algebra on $V$ by operators $L_n$, 
$n\in \ZZ$ is obtained by differentiating the representation $U$ of $\diff$ making $\A$ covariant, see Thm. \ref {TheoremVirDiff} and
\cite{Carpi2004, CW2005,loke}. 
Then, $L(z)=\sum_{n \in \ZZ} L_n z^{-n-2}$ is a local field on $V$, which, as a consequence of the locality of $\A$, is mutually local with all 
$Y(a,z)$, $a\in V$. Moreover, $L(z)\Omega= e^{zL_{-1}}L_{-2}\Omega$. By the uniqueness theorem for vertex algebras \cite[4.4]{Kac}
we have $L(z)= Y(\nu,z)$ where $\nu\equiv L_{-2}\Omega$. Hence $\nu$ is a conformal vector and hence $V$ is a VOA.

Now, the scalar product on $\H$ restrict to a normalized scalar product on $V$ having unitary M\"{o}bius symmetry in the sense of 
Subsect. \ref{SubsecEquivUnitary}. For every $a \in {\mathscr F}$ the adjoint vertex operator $Y(a,z)^+$ defined in Eq. (\ref{EqAdjointVO})
satisfies 
$$ Y(a,z)^+ = (-1)^{d_a}Y(\theta a, z)$$ 
and hence it is local and mutually local with respect to all the vertex operators $Y(b,z)$, $b\in V$. Now, let 
$${\mathscr F}_{+}= \{ \frac{a+(-1)^{d_a}\theta a}{2}: a \in {\mathscr F}  \} $$ 
and let 
$${\mathscr F}_{-}= \{ -i \frac{a-(-1)^{d_a}\theta a}{2}: a \in {\mathscr F}  \}. $$ 
Then, 
$\{Y(a,z): a \in {\mathscr F}_+ \cup {\mathscr F}_-\}$
is a family of Hermitian quasi-primary fields which generates $V$. Hence, $V$ is unitary by Prop. \ref{PropUnitaryGenHermitian}. 
Moreover, by Prop. \ref{simpleunitary} $V$ is simple because $V_0=\CC\Omega$. By  Prop. \ref{GenBoundedProp}, $V$, 
being generated by the family ${\mathscr F}$ of elements satisfying energy bounds, is energy-bounded. Since the net 
$\A_{\mathscr {\mathscr F}}$, cf. Eq. (\ref{EqA_F}),
coincides, by assumption, with $\A$, we can apply  Thm. \ref{generatingprimaryTheorem} to conclude that $V$ is strongly local and 
$\A_V=\A$.
\end{proof}

We end this section with the following conjecture. 
\begin{conjecture} For every irreducible conformal net $\A$ there exists a simple unitary strongly local vertex operator algebra $V$ such that 
$\A=\A_V$.  

\end{conjecture}

\appendix
\section{Vertex algebra locality and Wightman locality}
\label{AppendixVAandWightmannLocality}

The axiom of locality for vertex algebras is a purely algebraic formulation of the locality axiom for Wightman fields, 
see \cite[Chapter 1]{Kac}. In this work we use in various occasions some consequences of the correspondence of these two formulations of the axiom of locality. In the this appendix, using the simplifying assumption of the existence of polynomial energy bounds, we give a proof of the equivalence of these two formulations in a framework which is sufficiently general for all the applications in this paper.

Let $\H$ be a Hilbert space and let $L_0$ be self-adjoint operator on $\H$ with spectrum contained in 
$\ZZ_{\geq 0}$. We denote by $V$ the algebraic direct sum 
\begin{equation}
\H^{fin} \equiv \bigoplus_{n \in \ZZ_{\geq 0}} {\rm Ker}(L_0 - n 1_\H).
\end{equation}
Then $V$ is dense in $\H$. Moreover, we denote by $\H^\infty \subset \H$, the dense subspace of $C^\infty$ vectors for $L_0$ 
namely 
\begin{equation}
\H^\infty \equiv \bigcap_{k \in \ZZ_{> 0}} \mathcal{D}\left((L_0+1_\H)^k\right).
\end{equation}
Let $a_{n}$, $b_{n}$, $n\in \ZZ$ be operators on $\H$ with common domain $V$ and assume that 
$$e^{itL_0}a_ne^{-itL_0} = e^{-int} a_n , \quad  e^{itL_0}b_ne^{-itL_0} = e^{-int} b_n$$ 
for all $t\in \RR$ and all $n \in \ZZ$. It follows that 
$$a_n \mathrm{Ker}(L_0 - k 1_\H) \subset \mathrm{Ker}(L_0 - (k - n) 1_\H)$$
and 
$$b_n \mathrm{Ker}(L_0 - k 1_\H) \subset \mathrm{Ker}(L_0 - (k - n) 1_\H)$$ 
for all $n \in \ZZ$ and all $k \in \ZZ_{\geq 0}$ so that 
the operators $a_n, b_n$ restrict to endomorphisms of $V$ and for every $c\in V$ we have $a_nc  = b_nc =0$ for $n$ sufficiently large. 
As a consequence the formal series $\Phi_a(z)=\sum_{n\in \ZZ} a_n z^{-n}$ and $\Phi_b(z)=\sum_{n\in \ZZ} b_n z^{-n}$ are fields on $V$ in the sense of Subsect.\ref{susectVA}, see also \cite[Sect.3.1]{Kac}. We assume that the fields $\Phi_a(z)$ and $\Phi_b(z)$ satisfy (polynomial) energy bounds in the sense of Sect.\ref{sectionstronglylocalVOA} i.e. 
there exist positive integers $s, k$ and a constant $M >0$ such that, for all $n \in \ZZ$ and all  $c \in V$
\begin{equation} 
\label{EnergyBoundsAppendix}
\|a_n c \|\leq M (|n|+1)^s \|(L_0 +1_\H)^k c \|, \; \|b_n c \|\leq M (|n|+1)^s \|(L_0 +1_\H)^k c \|.
\end{equation}  
Accordingly we can define the smeared fields 
\begin{equation}
\Phi_a(f) = \sum_{n \in \ZZ} a_n \hat{f_n},\; \Phi_b(f) = \sum_{n \in \ZZ} b_n \hat{f_n}, 
\end{equation}
$f \in C^\infty(\s1)$, having $\H^\infty$ as  common invariant domain. 

According to Subsect.\ref{susectVA} we say that the fields $\Phi_a(z)$ and $\Phi_b(z)$ are {\bf mutually local in the vertex algebra sense} if there exists a non-negative integer $N$ such that  
\begin{equation}
(z-w)^N [\Phi_a(z),\Phi_b(w)] c = 0
\end{equation}
for all $c\in V$. 
Moreover, we say that the fields $\Phi_a(z)$ and $\Phi_b(z)$ are {\bf mutually local in the Wightman sense} if 
\begin{equation}
[\Phi_a(f),\Phi_b(\tilde{f})] c = 0
\end{equation}
for all $c\in \H^{\infty}$ if $\mathrm{supp}f \subset I$ and $\mathrm{supp}\tilde{f} \subset I'$, $I \in \I$, cf. \cite{StrWight}. We now show that under our assumptions these two locality conditions are equivalent. 
\begin{proposition}
\label{PropVAandWightmannLocality} The fields $\Phi_a(z)$ and $\Phi_b(z)$ are mutually local in the vertex algebra sense if and only if they are mutually local in the Wightman sense.
\end{proposition}
\begin{proof}
For every $c, d \in V$ the two variable formal series 
\begin{equation*}
(d| [\Phi_a(z),\Phi_b(w)] c) = \sum_{n,m \in \ZZ} (d| [a_n,b_m]c)z^{-n}w^{-m}
\end{equation*}
can be considered as a formal distribution $\varphi_{c,d}(z,w)$ on $\s1 \times \s1$ i.e. a linear functional on the complex vector space of the two variable trigonometric polynomials, see \cite[Sect.2.1]{Kac}. Because of the energy bounds this formal distribution extends by continuity to a unique ordinary distribution, again denoted by $\varphi_{c,d}(z,w)$, on $\s1 \times \s1$, i.e. a continuous linear functional 
on $C^\infty(\s1 \times \s1)$. If the fields $\Phi_a(z)$ and $\Phi_b(z)$ are mutually local in the vertex algebra sense then, by 
\cite[Thm.2.3 (i)]{Kac}, the distribution  $\varphi_{c,d}(z,w)$ has support in the diagonal $z=w$ of $\s1\times \s1$ and hence 
$(d|[\Phi_a(f),\Phi_b(\tilde{f})] c) = 0$ if $\mathrm{supp}f \subset I$ and $\mathrm{supp}\tilde{f} \subset I'$, $I \in \I$. Since $c,d \in V$ where 
arbitrary it follows that $[\Phi_a(f),\Phi_b(\tilde{f})] c = 0$
for all $c\in V$ if $\mathrm{supp}f \subset I$ and $\mathrm{supp}\tilde{f} \subset I'$, $I \in \I$. Now, as a consequence of the energy bounds, 
the same equalities also hold for for any $c\in \H^\infty$ and hence the fields $\Phi_a(z)$ and $\Phi_b(z)$ are mutually local in the Wightman sense.

Conversely let us assume that the fields $\Phi_a(z)$ and $\Phi_b(z)$ are mutually local in the Wightman sense. Then, 
the distribution  $\varphi_{c,d}(z,w)$ has support in the diagonal $z=w$ of $\s1\times \s1$. Moreover, as a consequence of the energy bounds, there is an integer $N >0$ such that, for all $c,d \in V$, $\varphi_{c,d}(z,w)$ is a distribution of order less then $N-1$, i.e. for every 
$c,d \in V$ there is a constant $M_{c,d}>0$ such that 
\begin{equation*}
|\langle \varphi_{c,d}, f \rangle | \leq M_{c,d}\; \mathrm{max} \big\{ \big| \partial^\alpha f(e^{i\vartheta_1},e^{i\vartheta_2}) \big|:\; e^{i\vartheta_1}, 
e^{i\vartheta_2} \in \s1, |\alpha| \leq N-1  \big\}
\end{equation*}
for all $f\in C^\infty (\s1\times \s1)$, where, as usual, for a multi-index $\alpha \equiv (\alpha_1,\alpha_2)$, 
$\alpha_1,\alpha_2 \in \ZZ_{\geq 0}$, $|\alpha|$ denotes the sum $\alpha_1 + \alpha_2$ and $\partial^\alpha$ denotes the 
partial differential operator of order $|\alpha|$ defined by 
\begin{equation*}
\partial^\alpha \equiv \left( \frac{\partial}{\partial \vartheta_1}\right)^{\alpha_1} \left( \frac{\partial}{\partial \vartheta_2}\right)^{\alpha_2}, 
\end{equation*}
see \cite[Chapter 6]{RudinFA}. Then, it follows by a rather straightforward adaptation of \cite[Thm.6.25]{RudinFA} and by \cite[Thm.2.3]{Kac}, that $(z-w)^N \varphi_{c,d}(z,w)=0$ for all $c,d \in V$ and hence that the fields $\Phi_a(z)$ and $\Phi_b(z)$ are mutually local in the vertex algebra sense.
\end{proof}

\section{On the Bisognano-Wichmann property for representations of the M\"{obius} group}
\label{BWmob}

Let $U$ be a strongly continuous unitary positive-energy representation of the M\"{o}bius group $\mob \simeq \psl2$ on a Hilbert space $\H$. Let $L_0$ be the self-adjoint generator of the one parameter subgroup of $U$ of (anti-clockwise) rotations. Then the spectrum 
of $L_0$ is a subset $\ZZ_{\geq 0}$. Accordingly, the (algebraic) direct sum $\H^{fin}$ of the subspaces ${\rm Ker}(L_0 -n1_\H)$, 
$n\in \ZZ_{\geq 0}$ is dense in $\H$. As it is well known the vectors in $\H^{fin}$ are smooth vectors for the representation $U$ and 
it is invariant for the representation of ${{\mathfrak sl}(2,\RR)} $ obtained by differentiating $U$, see \cite{LongoSibiu2008, pukanzsky} and \cite[Prop.A.1]{Carpi2004}. 
Accordingly there is a representation of ${{\mathfrak sl}(2,\RR)} $ on $\H^{fin}$ by essentially skew-adjoint operators and hence a unitary representation of its complexification $\sltwoC$. The latter Lie algebra representation is spanned by $L_0$ together with  operators $L_1$, $L_{-1}$ satisfying $L_1 \subset L_{-1}^*$ and the commutation relations $[L_1,L_{-1}]=2L_0$, $[L_1,L_0]=L_1$ and $[L_{-1},L_0]=-L_{-1}$. 

We say that a vector $a\in \H^{fin}$ is quasi-primary if $L_1a=0$ and $L_0a = d_a a$ for some $d_a \in \ZZ_{\geq 0}$. We say that 
$d_a$ is the conformal wight of $a$. If $a$ is quasi-primary we consider the vectors $a^n \in \H^{fin}$, $n\in \ZZ_{\geq 0}$ defined 
by $a^n \equiv \frac{1}{n!}L_{-1}^na$. The linear span $\H^{a, fin}$ of $\{a^n: n\in \ZZ_{\geq 0} \}$ is invariant for the action of the operators
$L_{-1}$, $L_0$, $L_1$ and the corresponding representation of $\sltwoC$ on $\H^{a, fin}$ is the irreducible unitary representation of 
$\sltwoC$ with lowest conformal energy $d_a$. Note that $L_0a^n=(n+d_a)a^n$ for all $n\in \ZZ_{\geq 0}$. Moreover, the closure $\H^a$ of $\H^{a,fin}$ is an irreducible $U$-invariant subspace of 
$\H$ carrying the unique (up to unitary equivalence) strongly continuous unitary positive-energy representation of $\mob$ with lowest conformal energy $d_a \in \ZZ_{\geq0}$.

If $d_a=0$ then $a^n=0$ for all $n>0$ and the corresponding representation of $\mob$ is the trivial one. 
For $d_a>0$ it is rather straightforward to prove by induction that $L_1a^n =(2d_a +n -1)a^{n-1}$ for all $n \in \ZZ_{>0}$ and that, 
as a consequence, 
\begin{equation}
\|a^n\|^2= \left(\begin{matrix} 2d_a + n -1 \\ n \end{matrix}\right) \| a\|^2, \; \textrm{for all}\; n \in \ZZ_{\geq 0}. 
\end{equation}
The above computation shows that for every $f\in C^\infty(\s1)$ the series 
$$\sum_{n\in \ZZ_{\geq 0}} \hat{f}_{-n-d_a} a^n$$
converges to an element $a(f)\in \H^a \subset \H$. Moreover, $f \mapsto a(f)$ is a linear continuous map $:C^\infty(\s1) \to \H^a$. 
Now, for any $\gamma \in \diff$ let $\beta_{d_a}(\gamma): C^\infty(\s1) \to C^\infty(\s1)$ be the map defined in 
Eq. (\ref{covariancefunctions}). Following the strategy for the proof of Prop. \ref{covariancevertexoperator} 
one can prove the following proposition which in fact can also be easily proved to be
a consequence of Prop. \ref{covariancevertexoperator} together with Prop. \ref{Y(a,f)Omega=a(f)} here below. 
\begin{proposition}
\label{covariancea(f)} Let $a \in \H$ be a quasi-primary vector of of conformal weight $d_a>0$. Then 
$U(\gamma)a(f) = a\big(\beta_{d_a}(\gamma)f \big)$ for all {\rm $\gamma \in \mob$} and all $f \in C^\infty(\s1)$.
\end{proposition}

Now, for every $I \in \I$ we define the closed real linear subspace $H^a(I) \subset \H^a$ to be the closure of the real linear subspace subspace 
$$\{a(f): f\in C^\infty(\s1, \RR), {\rm supp}f \subset I \}.$$ 
Then, as a consequence of Prop. \ref{covariancea(f)}, the family 
$\{H^a(I): I\in \I \}$ is M\"obius covariant, namely $U(\gamma)H^a(I)=H^a(\gamma I)$ for all $I\in \I$ and all $\gamma \in \mob$. 
Moreover, the family obviously satisfies isotony, namely $H^a(I_1) \subset H^a(I_2)$ if $I_1 \subset I_2$, $I_1, I_2 \in \I$. 

We now want to show that the family is a  local M\"obius covariant net of real linear subspaces of $\H^a$ in the sense of 
\cite[Def. 4.1]{LongoSibiu2008}, see also  \cite{LongoLectureNotes} and  \cite{BGL2002}. To this end we need to show that the family satisfies locality. 
Let $f_1, f_2 \in C^\infty(\s1,\RR)$. Then 

\begin{eqnarray*}
 \Im (a(f_1)|a(f_2)) = \frac{1}{2i}\sum_{n=0}^\infty \left( \widehat{(f_1)}_{n+d_a}\widehat{(f_2)}_{-n-d_a}  -   
\widehat{(f_2)}_{n+d_a}\widehat{(f_1)}_{-n-d_a} \right) \|a^n\|^2  \\
 = \frac{1}{2i}\sum_{n=0}^\infty \left( \widehat{(f_1)}_{n+d_a} \widehat{(f_2)}_{-n-d_a}  -   
\widehat{(f_2)}_{n+d_a} \widehat{(f_1)}_{-n-d_a} \right)
\left(\begin{matrix} 2d_a + n -1 \\ n \end{matrix} \right) \|a\|^2  \\
= \frac{1}{2i}\sum_{n=d_a}^\infty \left( \widehat{(f_1)}_{n} \widehat{(f_2)}_{-n}  -   
\widehat{(f_2)}_{n} \widehat{(f_1)}_{-n} \right)
\left(\begin{matrix} d_a + n -1 \\ n-d_a \end{matrix} \right) \|a\|^2. 
\end{eqnarray*}

Now let $p_{d_a}(x)$ be the polynomial defined by
\begin{equation}
p_{d_a}(x) \equiv \frac{(d_a + x -1)(d_a + x - 2)\cdots  (d_a + x -2d_a +1)}{2d_a -1}.
\end{equation}

Then
$$
p_{d_a}(n) = \left(\begin{matrix} d_a + n -1\\ n-d_a \end{matrix} \right)
$$
for every integer $n\geq d_a$. Moreover, $p_{d_a}(n)=0$ for $n=0,1,\cdots,d_a-1$. Note also that $p_{d_a}(x)=x$ for $d_a=1$ while 
for $d_a>1$ we have 
\begin{equation}
p_{d_a}(x)=\frac{x}{2d_a -1}\prod_{n=1}^{d_a-1}(x^2-n^2),
\end{equation}
so that $p_{d_a}(x)$ is an odd polynomial.
Hence
\begin{eqnarray*}
\Im (a(f_1)|a(f_2)) &=& \frac{\|a\|^2}{4i}\sum_{n \in \ZZ} \left( \widehat{(f_1)}_{n} \widehat{(f_2)}_{-n}  -   
\widehat{(f_2)}_{n} \widehat{(f_1)}_{-n} \right) p_{d_a}(n)\\ 
&=& \frac{\|a\|^2}{4\pi i} \int_0^{2\pi} f_2(e^{i\vartheta})p_{d_a}(-i\frac{{\rm d}}{{\rm d}\vartheta})f_1(e^{i\vartheta}) {\rm d}\vartheta.
\end{eqnarray*}
As a consequence, if ${\rm supp}f_1 \subset I$  and ${\rm supp}f_2 \subset I'$, $I\in \I$ then $\Im(a(f_1)|a(f_2))=0$ and hence 
the M\"{o}bius covariant isotonous family $\{H^a(I): I\in \I \}$ satisfies locality so that it is a local M\"obius covariant net of real linear subspaces of $\H^a$ in the sense of \cite[Def. 4.1]{LongoSibiu2008}. 

\begin{lemma} 
\label{bwlemma}
Let $a \in \H$ be a quasi-primary vector of conformal weight $d_a>0$ and let $K\equiv i\pi \overline{(L_{1}-L_{-1})}$. Then, 
there exists $\alpha_a \in \CC$ with $|\alpha_a|=1$ such that $a(f)$ is in the domain of $e^{\frac12 K}$ and 
$e^{\frac12 K}a(f) = \alpha_a a(f\circ j)$ for all $f\in C^\infty(\s1)$ with ${\rm supp}f \subset \s1_+$, where $j(z)=z^{-1}$ for all $z\in \s1$.
\end{lemma}
\begin{proof}

Let $H \equiv H^a(\s1_+)$. Then by \cite[Thm.4.2]{LongoSibiu2008} $H$ is a standard subspace of $\H^a$, see  
\cite[Sect.3]{LongoSibiu2008}. Hence one can define on $\H^a$ the antilinear closed operator $S_H$ having polar decomposition 
$J_{H}\Delta_H^{1/2}$ as in \cite[Sect.3]{LongoSibiu2008}. Le $\delta(t)$ be the one-parameter subgroup of $\mob$ defined by 
$$\delta(t)(z)= \frac{z\cosh t/2  - \sinh t/2}{-z\sinh t/2 + \cosh t/2},$$
(``dilations'') corresponding to the vector field $\sin \vartheta \frac{{\rm d}}{{\rm d}\vartheta}$ on $\s1$, $z=e^{i\vartheta}$. We have 
$\delta(t)S^1_+ = S^1_+$ for all $t\in \RR$. 

Since $e^{iKt} = U(\delta(-2\pi t))$ for all $t\in \RR$, it follows from \cite[Thm.4.2]{LongoSibiu2008} that $\Delta_H^{1/2}$ coincides with the restriction to $\H^a$ of $e^{\frac12 K}$. Accordingly, if $f:S^1\to \CC$ is a smooth function with  ${\rm supp}f \subset \s1_+$ then $a(f)$ is in the domain of $e^{\frac12 K}$ and $J_He^{\frac12 K}a(f) = S_Ha(f) = a(\overline{f})$ and thus $e^{\frac12 K}a(f)=J_Ha(\overline{f})$. Now, again by \cite[Thm.4.2]{LongoSibiu2008}, $J_H$ commutes with the restrictions of $L_{-1}$, $L_0$, $L_1$ to $\H_a$ and hence there exists $\alpha_a \in \CC$ with $|\alpha_a|=1$ such that $J_Ha^n =\alpha_a a^n$ for all $n\in \ZZ_{\geq 0}$. It follows that 
$J_Ha(f)=\alpha_a a\left(\overline{f\circ j}\right)$ for all $f\in C^\infty(\s1)$. Hence, if ${\rm supp}f \subset \s1_+$ then 
$e^{\frac12 K}a(f)= \alpha_a a(f \circ j)$.
\end{proof}

Our next goal is to compute the constant $\alpha_a$ in Lemma \ref{bwlemma} for every quasi-primary $a \in \H$ with conformal weight 
$d_a>0$.

\begin{proposition} 
\label{alpha_a}
$\alpha_a = (-1)^{d_a}$ for every quasi-primary vector $a\in \H$ of conformal weight $d_a>0$. 
\end{proposition}
\begin{proof} Let $f$ be a smooth real function on $\s1$ whose support is a subset of $\s1_+$ and let 
$f^{d_a,t}\equiv \beta_{d_a}(\delta(-2\pi t))f$, $t\in \RR$. 

Consider the function $\varphi :\RR \to \CC$ defined by
$$\varphi(t) \equiv (a | a(f^{d_a,t}))=\|a\|^2\widehat{(f^{d_a,t})}_{-d_a}= \frac{\|a\|^2}{2\pi}\int_{-\pi}^{\pi} f^{d_a,t}(e^{i\vartheta})e^{id_a\theta}
{\rm d}\vartheta.  $$
Recalling the explicit form of $f^{d_a,t}(e^{i\vartheta})$ and Eq. ( \ref{covariancefunctions}) we find 
$$\varphi(t)=\frac{\|a\|^2}{2\pi}\int_{-\pi}^{\pi} \left(X_{\delta(-2\pi t)}(\delta(2\pi t)(e^{i\vartheta}))\right)^{d_a-1}f(\delta(2\pi t)(e^{i\vartheta}))e^{id_a\theta}
{\rm d}\vartheta.$$
We now change the variable in the integral by setting $e^{i\vartheta}=\delta(-2\pi t)(e^{i\alpha})$, $\alpha \in [-\pi,\pi]$ with the following result
\begin{eqnarray*}
\varphi(t)=\frac{\|a\|^2}{2\pi}\int_{-\pi}^{\pi} f(e^{i\alpha})\left(X_{\delta(-2\pi t)}(e^{i\alpha})\right)^{d_a-2}
\left(\delta(-2\pi t)(e^{i\alpha})\right)^{d_a}{\rm d}\alpha \\
=\frac{\|a\|^2}{2\pi} \int_{-\pi}^{\pi} f(e^{i\alpha})
\left(-i\frac{{\rm d}}{{\rm d}\alpha}\log(\delta(-2\pi t)(e^{i\alpha}))\right)^{d_a-2}\left(\delta(-2\pi t)(e^{i\alpha})\right)^{d_a}{\rm d}\alpha\\
=\frac{\|a\|^2}{2\pi} \int_{0}^{\pi} f(e^{i\alpha})
\left(-i\frac{{\rm d}}{{\rm d}\alpha}\log(\delta(-2\pi t)(e^{i\alpha}))\right)^{d_a-2}\left(\delta(-2\pi t)(e^{i\alpha})\right)^{d_a}{\rm d}\alpha\\
\equiv \int_{0}^{\pi} k_{d_a}(t,\alpha) f(e^{i\alpha}){\rm d}\alpha
\end{eqnarray*}
where we used the fact that $f(e^{i\alpha})=0$ for $\alpha\in [-\pi,0]$ by assumption. Now, using the explicit expression 
$$\delta(-2\pi t)(e^{i\alpha})= \frac{e^{i\alpha}\cosh(\pi t)  + \sinh(\pi t)}{e^{i\alpha}\sinh(\pi t) + \cosh(\pi t)},$$
 it is straightforward to check that, for any $\alpha \in [0,\pi]$.  $t\mapsto k_{d_a}(t,\alpha)$ extends to a continuos function $z\mapsto k_{d_a}(z,\alpha)$ on  the closed strip  $\mathfrak{S} \equiv \{z\in \CC: \Im z \in [-1/2,0]\}$, which is holomorphic in the interior of 
 $\mathfrak{S}$. Moreover, the function of two variables $(z,\alpha)\mapsto k_{d_a}(z,\alpha)$ is continuous on $\mathfrak{S}\times [0,\pi]$. Accordingly 
 $$\Phi_1(z) \equiv   \int_{0}^{\pi} k_{d_a}(z,\alpha) f(e^{i\alpha}){\rm d}\alpha $$
is continuous on the strip $\mathfrak{S}$ and holomorphic in its interior. Clearly $\Phi_1(t)=\varphi(t)$ for all $t\in \RR$. 
Moreover, one finds that  
$$k_{d_a}(-i/2,\alpha)=  \frac{\|a\|^2}{2\pi} (-1)^{d_a} e^{-id_a \alpha} $$
and thus 
\begin{eqnarray*}
\Phi_1(-i/2) &=& (-1)^{d_a}\frac{\|a\|^2}{2\pi}\int_{0}^{\pi}f(e^{i\alpha})e^{-id_a \alpha} {\rm d}\alpha \\
&=& (-1)^{d_a} \frac{\|a\|^2}{2\pi}\int_{-\pi}^{0}f(e^{-i\alpha})e^{id_a \alpha} {\rm d}\alpha \\
&=& (-1)^{d_a}(a | a(f\circ j)).
\end{eqnarray*}

On the other hand, since $\varphi(t)=(a|e^{iKt}a(f))$ for all $t\in \RR$ and $a(f)$ is in the domain of $e^{\frac{K}{2}}$ there is a function 
$\Phi_2(z)$ which is continuous on the strip $\mathfrak{S}$ and holomorphic in its interior, such that $\Phi_2(t)=\varphi(t)$. Moreover, 
by Lemma \ref{bwlemma} we have $\Phi_2(-i/2)= \alpha_a (a | a(f\circ j) )$. Now, $\Phi_1(t)=\Phi_2(t)$ for all $t \in \RR$ and hence, 
by the Schwarz  reflection principle we have $\Phi_1(z)=\Phi_2(z)$ for all $z \in \mathfrak{S}$ and hence 
$(-1)^{d_a}(a | a(f\circ j)) = \alpha_a (a | a(f\circ j) )$. The conclusion then follows from the fact that we can take $f \in C^\infty(\s1,\RR)$ 
with support in $\s1_+$ and such that $(a | a(f\circ j) ) \neq 0$ \end{proof}

The following theorem is a straightforward consequence of Lemma \ref{bwlemma} together with Prop. \ref{alpha_a}.

\begin{theorem} 
\label{bw1} 
Let $K\equiv i\pi  \overline{(L_{1}-L_{-1})}$ and let $f\in C^\infty(\s1)$ with ${\rm supp}f \subset \s1_+$. Then $a(f)$ is in the domain 
of $e^{\frac12 K}$ and $e^{\frac12 K}a(f) = (-1)^{d_a}a(f\circ j)$, where $j(z)=z^{-1}$ for all $z\in \s1$.
\end{theorem}

\begin{proposition} 
\label{Y(a,f)Omega=a(f)}
Let $V$ be a simple unitary energy-bounded VOA and let $a\in V_{d_a}$ be quasi-primary. Then 
$Y(a,f)\Omega = a(f)$.
\end{proposition}
\begin{proof} It follows directly from Eq. (\ref{l_-1commutation}) that $a_{-n-d_a}\Omega = \frac{1}{n!}L_1^n\Omega$ for all 
$n\in \ZZ_{\geq 0}$. Moreover, $a_n\Omega =0$ for all integers $n > -d_a$. Hence the conclusion follows from the definition of $a(f)$. 
\end{proof}

The following theorem plays a crucial role in the proof of Thm. \ref{generatingprimaryTheorem}.
\begin{theorem} 
\label{bw2} 
Let $V$ be a simple unitary energy-bounded VOA and let $a\in V_{d_a}$ be quasi-primary.
Let $K\equiv i\pi \overline{(L_{1}-L_{-1})}$ and let $f\in C^\infty(\s1)$ with ${\rm supp}f \subset \s1_+$. Then $Y(a,f)\Omega$ is in the domain of $e^{\frac12 K}$ and $e^{\frac12 K}Y(a,f)\Omega = (-1)^{d_a}Y (a, f\circ j ) \Omega$, where $j(z)=z^{-1}$ for all $z\in \s1$.
\end{theorem}
\begin{proof}
The theorem follows directly from Thm. \ref{bw1} and Prop. \ref{Y(a,f)Omega=a(f)} in the case $d_a>0$. In the case $d_a=0$ it holds true trivially.
\end{proof}

\medskip

\noindent\textbf{Acknowledgements.} 
We thank V. Kac and F. Xu for useful discussions and comments. S. C. would like to thank Y. Tanimoto for pointing him references 
\cite{Epstein1984} and \cite{Mather1974}.

\end{document}